\documentclass[a4paper,10pt]{article}

\usepackage{autograph2}
\usepackage{mathdots}
\usepackage{graphicx}
\usepackage{epic,latexsym,bezier,amsbsy,color,enumerate,amsfonts,amsmath,amsthm,amscd,amssymb}
\usepackage{float}
\usepackage[latin1]{inputenc}

\newtheorem{defi}{Definition}[subsection]
\newtheorem{thm}[defi]{Theorem}
\newtheorem{lem}[defi]{Lemma}
\newtheorem{prop}[defi]{Proposition}
\newtheorem{cor}[defi]{Corollary}
\newtheorem{rem}[defi]{Remark}

\newtheorem{conv}[defi]{Convention}

\newcommand{\lcr}{[\![}
\newcommand{\rcr}{]\!]}

\begin{document}

\title{Abelian Sandpile Model on Randomly Rooted Graphs and Self-Similar Groups}

\renewcommand{\thefootnote}{\fnsymbol{footnote}}
\author{M. Matter \footnotemark[1] \footnotemark[2] and T. Nagnibeda \footnotemark[1] \footnotemark[3]}

\footnotetext[1]{The authors acknowledge the support of the Swiss National Science Foundation Grant PP0022-118946.}
\footnotetext[2]{Section de mathématiques, 2-4 Rue du Lièvre CP 64, 1211 Genève 4, \texttt{Michel.Matter@unige.ch}}
\footnotetext[3]{Section de mathématiques, 2-4 Rue du Lièvre CP 64, 1211 Genève 4, \texttt{Tatiana.Smirnova-Nagnibeda@unige.ch}}

\date{\today}

\maketitle
\begin{abstract}

Abelian sandpile model is an archetypical model of the physical phenomenon of self-organized criticality. It is also well studied in combinatorics under the name of chip-firing games on graphs. One of the main open problems about this model is to provide rigorous mathematical explication for  predictions about the values of its critical exponents, originating in physics. The model was initially defined on the cubic lattices $\mathbb Z^d$, but the only case where the value of some critical exponent was established so far is the case of the infinite regular tree -- the Bethe lattice.

This paper is devoted to the study of the Abelian sandpile model on a large class of graphs that serve as approximations to Julia sets of postcritically finite polynomials  and occur naturally in the study of automorphism group actions on infinite rooted trees. While different from the square lattice, these graphs share many of its geometric properties: they are of polynomial growth, have one end, and random walks on them are recurrent. This ensures that the behaviour of sandpiles on them is quite different from that observed on the infinite tree. We compute the critical exponent for the decay of mass of sand avalanches on these graphs and prove that it is inverse proportional to the rate of polynomial growth of the graph, thus providing the first rigorous derivation of the critical exponent different from the mean-field (the tree) value.
\end{abstract}

\noindent\emph{\textbf{keywords:}} Abelian sandpile model, critical exponent, avalanche,
random weak limits of graphs, self-similar group, Schreier graph, polynomial growth.

\noindent\emph{\textbf{2010 Mathematics Subject Classification:}} Primary: 60K35; Secondary: 82C22, 20E08.

\section{Introduction}

\indent The \emph{Sandpile Model} was introduced in the late eighties by physicists Bak, Tang and Wiesenfeld \cite{BakTangWies88} in the aim of constructing an analytically tractable model of a phenomenon often observed in nature and called \emph{self-organized criticality}. Its mathematical study was initiated by Dhar in \cite{Dhar90}; in particular, he proved that the model is abelian. This result was also recovered independently in the work of Bj\"orner, Lovasz and Shor \cite{BLSch} where the same model was studied under the name of \emph{chip-firing game on graphs}. A detailed treatment of the ASM can be found in \cite{MeestRedZnam01}, \cite{Dhar99}, \cite{RedigHouches05}.\\
\indent A \textbf{configuration} is a distribution of an amount of chips (or of grains of sand) on the vertices of a connected, possibly infinite, locally finite multigraph. When the number of chips on a given vertex $v$ exceeds its degree, the vertex is declared to be \textbf{unstable} and is fired: a chip is sent along each edge incident to $v$ to the corresponding neighbour of $v$, providing a new configuration of the model. The term \emph{abelian} stands for the following convenient feature of the model: the order in which we stabilize unstable vertices of a configuration does not affect the result \cite{Dhar90}.\\
\indent Once a stable configuration is reached, the game can be reactivated by adding an extra chip on a randomly chosen vertex. In the case of a finite graph, this defines a Markov chain whose stationary distribution is the uniform distribution supported by the unique recurrent class (more details in Section \ref{Chip-firing} below).
The dynamics of the model is described by \textbf{avalanches}, that is, sequences of consecutive firings triggered by adding an extra chip
to a random recurrent configuration. Given a growing sequence of finite subgraphs $\{\Gamma_n\}_{n\geq 1}$ of an infinite graph $\Gamma$, \textbf{criticality} of the ASM on $\Gamma$ is manifested in that various spatial statistics associated with avalanches (such as their \emph{mass, length, diameter}, etc.) decay asymptotically according to a power law (with a cut-off), as $\Gamma_n\nearrow\Gamma$.
Although many numerical simulations have been done in order to exhibit criticality of ASM on lattices, as well as to determine various \textbf{critical exponents}, there are only very few rigorously proven cases so far.\\
\indent In the case of the $d$-regular tree, $d\geq 3$, (also called the Bethe lattice), Dhar and Majumdar proved in \cite{DharMajumdar90} that the critical exponent corresponding to the mass $M$ of an avalanche is $\delta_M=3/2$ in large volume limit.\\
\indent On the one-dimensional lattice $\mathbb{Z}$, the probability of observing an avalanche of mass $M>0$ (respectively length $L>0$) on a segment of length $n$ is independent of $M$ (respectively $L$) and 
the behaviour of the model is not critical \cite{RuSen}.\\
\indent Numerical experiments as well as non-rigorous scaling arguments yield the conjecture that on the two-dimensional lattice $\mathbb{Z}^2$, the critical exponent for the mass of an avalanche is $\delta_M=5/4$, \cite{PriezzKitIva96}, whereas for $d>4$ this critical exponent is expected to be $3/2$, by universality \cite{Priezz00}. Criticality of ASM on $\mathbb{Z}^d$ is confirmed in \cite{DharMajum91}: the correlation between the indicator functions of having no chip on vertex $0$ and no chip on a vertex $x\in\mathbb{Z}^d$ behaves as $C|x|^{-2d}$ in large volume limit.\\
\indent Another family of graphs on which extensive simulations of avalanches have been performed is the Sierpi\'nski gasket where it is shown that $\delta_M\approx1.46$ \cite{DaerdenVander04}. See also \cite{Sierpinski}.\\

\indent In this paper we exhibit a family of infinite graphs for which we can explicitly compute the critical exponent for the decay of avalanches (on the approximating sequence of finite graphs).
Our examples are regular graphs with geometric properties significantly different from those of regular trees: the generic number of ends is $1$, they have polynomial growth rate, and the simple random walk on them is recurrent. The method that we develop to prove criticality of the model is also quite different from the Majumdar-Dhar's technique used to prove criticality in the case of regular trees.\\

\indent Our examples come from the theory of \textbf{self-similar groups} developed in the past ten years by Grigorchuk, Nekrashevych and others (see \cite{Grig05}, \cite{NekBook} and references therein) -- a natural source of families of finite graphs with interesting infinite limits of self-similar nature. More precisely, a finitely generated group $G<Aut(T)$ acting by automorphisms on a regular rooted tree $T$ defines a covering sequence $\{\Gamma_n\}_{n\geq 1}$ of finite Schreier graphs describing the action of $G$ on each level of the tree. These graphs converge to infinite orbital Schreier graphs $\{\Gamma_{\xi}\}_{\xi\in\partial T}$ of the limit action of $G$ on the boundary $\partial T$ of the tree (see Section
\ref{SelfSimilarGroups} for details).\\
\indent One eminent example in the class of self-similar groups is the so-called \emph{Basilica group} introduced by Grigorchuk and \.{Z}uk in \cite{grizuk}.
It can be realized as the iterated monodromy group of the complex polynomial $z^2-1$, which means in particular that its Schreier graphs form an approximating sequence of the Julia set of $z^2-1$, the so-called Basilica fractal \cite{NekBook}. It is a $2$-generated group which acts by automorphisms on the binary tree.\\
\indent We show that the Basilica group provides us with an uncountable family of $4$-regular one-ended graphs of quadratic growth where the critical exponent for the mass of avalanches in the ASM is equal to $1$ (Theorem \ref{Main}). It also gives an uncountable family of $2$-ended graphs of quadratic growth with non-critical ASM -- first such examples not quasi-isometric to $\Bbb Z$ (Theorem \ref{Main2ends}).\\
\indent Technically, our approach relies on the fact that the Schreier graphs of the Basilica group are \emph{cacti}, i.e., separable graphs whose blocks are either cycles or single edges (see Section \ref{SubsectAvOnCT}). 
The groups of automorphisms of rooted trees whose Schreier graphs are cacti, and to which our method therefore applies, form a large class of groups characterized by Nekrashevych as iterated monodromy groups of post-critically finite backward iterations of topological polynomials \cite{Nekr09}. Another example from this class of groups is the so-called \lq\lq interlaced adding machines\rq\rq , or the $IMG(-z^3/2+3z/2)$  \cite{Nek}. This group shares many properties with the Basilica group, and the same goes for their Schreier graphs. For the ASM, this group provides examples of graphs with the critical exponent $\delta_M=2 \log 2/\log 3 >1$ (see Section \ref{SectionIAM} and Theorem \ref{IAM}).\\
\indent More generally, in Theorem \ref{Growth} we establish a connection between the critical exponent for the mass of avalanches in the ASM on one-ended Schreier graphs and the degree of their polynomial growth. Quadratic polynomials $z^2+c$ with the values of $c$ taken in smaller and smaller hyperbolic components attached to the main cardioid of the Mandelbrot set provide examples of iterated monodromy groups whose Schreier graphs have polynomial growth of arbitrarily high degree. Consequently, probability distributions of the mass of avalanches on these Schreier graphs decay as power laws with arbitrarily small critical exponent. These examples are discussed in Subsection \ref{k(v)}.\\

\indent In order to address all these examples we develop the study of the Abelian sandpile model for unimodular random rooted graphs, a natural generalization of homogeneous graphs (see \cite{AlLyons07} and Subsection \ref{ASMonGamma_n} below) which are by definition invariant probability distributions on the space
$\mathcal{X}$ of (rooted isomorphism classes of) locally finite, connected \textbf{rooted graphs}. They occur naturally as \textbf{random weak limits} of finite graphs. Introduced by Benjamini and Schramm \cite{RandWeakLimit}, the random weak limit stands for passing to the limit in the space $\mathcal{X}$, for a sequence of (unrooted) graphs $\{\Gamma_n\}_{n\geq 1}$, by choosing the root uniformly at random, thus considering each unrooted graph $\Gamma_n$ in the sequence as a probability distribution $\rho_n$ on $\mathcal{X}$. The random weak limit of a sequence $\{\Gamma_n\}_{n\geq 1}$ of connected graphs of bounded degree is defined to be the weak limit of the measures $\rho_n$ in the space of probability measures on $\mathcal{X}$. \\
\indent In Subsection \ref{ASMonGamma_n} below, we introduce the ASM on sequences of graphs converging in the space $\mathcal{X}$ of rooted graphs and discuss such issues as the choice of dissipative vertices, the choice of root, and criticality in the random weak limit.\\

\indent As previously mentioned, exhibiting the criticality of the ASM goes through studying the statistical behaviour of avalanches, by looking at different observable quantities related to them. In this paper, we focus our attention on the mass of avalanches (i.e. the number of distinct vertices fired), however a similar approach may be applied for studying their diameter (i.e. the diameter of the subgraph spanned by vertices touched by the avalanche) or their length (i.e. the total number of firings). Indeed, the key step in Subsection \ref{CT} (Proposition \ref{RemOnAv}) depends on the avalanche and not only on its mass. It turns out that the diameter can be studied in a very similar way to the mass (see Remark \ref{RemOnDiam}). For the length, however, computations become more tedious. Also, there is no clear relation between the length of avalanches and global geometrical properties of the underlying graph, as it is the case for the mass (or the diameter) of avalanches and the degree of polynomial growth of the graph.\\


\indent The paper is structured as follows: in Section \ref{ASM}, we collect some facts and notations about the ASM and then consider general properties of the model on separable graphs and, in particular, on cacti. In Subsection \ref{ASMonGamma_n}, we introduce and discuss the ASM on sequences of graphs converging in the space $\mathcal X$ of rooted graphs, as well as criticality of the ASM in the random weak limit. Section \ref{SelfSimilarGroups} recalls basic notions about groups of automorphisms of rooted trees, self-similar groups and their Schreier graphs. We show that any covering sequence of finite regular graphs of even degree can be realized as Schreier graphs for an action of a finitely generated group on a spherically homogeneous rooted tree, by automorphisms. In Section \ref{SubsecErgodicityAval}, we go back to the study of avalanches and show that for covering sequences of regular cacti critical in the random weak limit, the critical exponent is almost surely constant. Section \ref{Basilica} recalls results from \cite{DanDonMat09} about the structure of finite and infinite Schreier graphs of the Basilica group. In Section \ref{AvOnBasilica} we study the ASM on these graphs, in particular we show that almost all orbital Schreier graphs of the Basilica group are critical with the critical exponent equal to $1$. In Section \ref{SectionIAM}, we consider the group generated by two interlaced adding machines and exhibit examples with the critical exponent equal to $2\log 2/\log 3>1$. In Section \ref{LastSection}, a relation is established between the critical exponent for the mass of avalanches and the degree of polynomial growth, for $1$-ended cacti; and graphs with arbitrarily small critical exponents are discussed.

\section{Abelian Sandpile Model} \label{ASM}

\subsection{Chip-firing Game on a graph} \label{Chip-firing}

Let $\Gamma=(V,E)$ be a finite connected graph, possibly with multiple edges and loops, with a vertex set $V\equiv V(\Gamma)$ and an edge set $E\equiv E(\Gamma)$. Let $P\subset V$ be a non-empty set of vertices that will be called \textbf{dissipative vertices}. We will write $V_0:=V\backslash P$. A \textbf{configuration} on $\Gamma$ is a function $\eta:V_0\longrightarrow \mathbb{N}$. We say that $\eta$ is \textbf{stable} if $\eta(v)<\deg(v)$ for all $v\in V_0$ where $\deg(v)$ denotes the degree of $v$, that is, the number of edges incident to $v$ (each loop contributes two to the degree.) An unstable configuration evolves by \textbf{firing} its unstable vertices as long as there are some. Firing an unstable vertex $v$ corresponds to sending one chip along each edge incident to $v$ to the corresponding neighbour. We will adopt the convention that all chips reaching a dissipative vertex $p\in P$ leave the graph. The basic theorem about the game asserts that every configuration reaches through a finite number of firings a stable configuration. Moreover, the resulting stable configuration, the set of vertices fired in the stabilization and the number of times each of these vertices were fired are all independent of the order in which the unstable vertices are fired \cite{Dhar90}. Given a configuration $\eta$, a consecutive sequence of firings resulting in the stabilization of $\eta$ is called an \textbf{avalanche}. The number of vertices (respectively distinct vertices) fired during the avalanche is called its length (respectively mass). By the result cited above, both the mass and the length are the same for all avalanches leading to the stabilization of a given configuration.\\
\indent Let $\Omega$ denote the set of all stable configurations, and let us consider the following Markov chain on $\Omega$ \cite{MeestRedZnam01}. Starting from some initial stable configuration $\eta_0$, we add an extra chip to $\eta_0$ on a vertex $v\in V_0$ chosen accordingly to some initially fixed probability distribution $\pi:V_0\longrightarrow ]0,1]$ satisfying the condition $\pi(v)>0$ for all $v\in V_0$. Then, we let the configuration $\eta_0+\delta_v$ stabilize and denote by $\eta_1\in\Omega$ the resulting stable configuration. We then repeat the previous operation with $\eta_1$, and so on. Recurrent states of this Markov chain form a single (communication) class, denoted by $\mathcal{R}_\Gamma$; consequently, the Markov chain admits a unique stationary measure $\mu$ which is supported by $\mathcal{R}_\Gamma$. It turns out that the set $\mathcal{R}_\Gamma$ of \textbf{recurrent} (or \textbf{critical}) configurations can be given the structure of a group, and therefore $\mu$ is in fact the uniform measure on $\mathcal{R}_\Gamma$, independently of the distribution $\pi$ (see more on this in the very end of this subsection).\\
\indent The set of recurrent configurations can be constructed by a deterministic procedure called the \textbf{Burning Algorithm} \cite{Dhar90}. Given a configuration $\eta$ and a subgraph $H\subset\Gamma$ not containing dissipative vertices, we say that the restriction $\eta_H$ is a \textbf{forbidden sub-configuration} of $\eta$ if $\eta(v)<\deg_H(v)$ for every $v\in H$ (where $\deg_H(v)$ denotes the degree of $v$ in $H$).
Dhar has shown that a stable configuration on $\Gamma$ is recurrent if and only if it does not contain any forbidden sub-configuration. The Burning Algorithm decides, given a configuration $\eta$, whether it contains a forbidden sub-configuration or not, as follows. For $t\geq 1$, we define inductively the sets $B_t$ and $U_t$, where $B_t$ stands for the set of vertices \lq\lq burnt\rq\rq\ at time $t$, and $U_t$ stands for the set of vertices \lq\lq un-burnt\rq\rq\ up to time $t$. We also denote by $\Gamma_t$ the subgraph of $\Gamma$ spanned by the vertices in $U_t$, whereas $\Gamma_0$ denotes the subgraph of $\Gamma$ spanned by $V_0$.

$B_1:=\{v\in V_0|\eta(v)\geq\deg_{\Gamma_0}(v)\}$;

$U_t:=V_0\backslash\bigcup_{s=1}^{t}B_s$;

$B_{t+1}:=\{v\in U_{t}|\eta(v)\geq\deg_{\Gamma_{t}}(v)\}$.

\noindent If there exists $t_0$ such that $B_{t_0+1}$ is empty, then $\eta_{U_{t_0}}$ is forbidden. Otherwise, every vertex of $\Gamma$ is eventually burnt, which implies that $\eta$ does not contain any forbidden configuration.\\
\indent We will use the following equivalent reformulation of Dhar's theorem \cite{MeestRedZnam01}.

\begin{thm} \label{BA}
A configuration $c$ on $\Gamma$ is recurrent if and only if there exists a sequence of firings (with respect to $c$) $p_1\dots p_kv_1v_2\dots v_{|V_0|}$ which is an enumeration of $V$. Here firing a dissipative vertex $p_i$ means that we add on each neighbour $v$ of $p_i$ as many chips as there are edges between $v$ and $p_i$ in $\Gamma$. We call such a sequence a \textbf{burning sequence for $c$ on $\Gamma$}.
\end{thm}

\noindent Note that applying a burning sequence to a recurrent configuration $c$ returns $c$.\\
\indent It can be deduced from this theorem that, if $|P|=1$,  the Burning Algorithm establishes a bijection between recurrent configurations and spanning trees in $\Gamma$ (in the general case the bijection is between recurrent configurations and spanning forests where each tree contains exactly one dissipative vertex \cite{ChungEllis02}.)\\
\indent An interesting result about recurrent configurations is that, when endowed with the operation of adding the configurations coordinatewise and then stabilizing, they form an abelian group denoted by $K(\Gamma)$ and called the \textbf{critical group} \cite{Biggs99} of the graph. In general, it is not easy to determine the algebraic structure of $K(\Gamma)$ and a particularly intriguing problem consists in establishing connections between the decomposition of the critical group into invariant factors and the graph structure \cite{Lor08}, \cite{BachHarpNagn97}. There are but a few examples of families of graphs where the critical group and its decomposition have been computed, including the complete graphs \cite{BachHarpNagn97}, the wheel graphs \cite{Biggs99}, finite balls in a regular tree \cite{Lev09}, \cite{Toump07}, cacti \cite{Mat09} (see also Subsection \ref{SubsectAvOnCT} below). More detailed information about the critical group, such as explicit description of the neutral element and of the inverses in terms of recurrent configurations is very sparse, see \cite{ChenShed07} for "thick trees", \cite{Mat09} for cacti, and \cite{BorRoss02} for a study of the neutral element on growing rectangles in $\mathbb{Z}^2$.

\subsection{Avalanches on unimodular random rooted graphs} \label{ASMonGamma_n}

As explained in the introduction, in this paper we propose to study avalanches on unimodular random rooted graphs. Let $\mathcal{X}$ denote the space (of rooted isomorphism classes) of locally finite, connected graphs having a distinguished vertex called \textbf{the root}; $\mathcal{X}$ can be endowed with the following metric: given two rooted graphs $(\Gamma,v)$ and $(\Gamma',v')$,

\begin{equation} \label{EqDist}
Dist((\Gamma,v),(\Gamma',v')):=\inf\left\{\frac{1}{r+1};\textrm{$B_{\Gamma}(v,r)$
is isomorphic to $B_{\Gamma'}(v',r)$}\right\}
\end{equation}

\noindent where $B_{\Gamma}(v,r)$ is the ball of radius $r$ in $\Gamma$ centered in $v$. We say that a sequence of rooted graphs $\{(\Gamma_n,v_n)\}_{n\geq1}$ converges to a \textbf{limit graph} $(\Gamma,v)$ if
$\lim_{n\rightarrow\infty}Dist((\Gamma,v),(\Gamma_n,v_n))=0$. If one supposes moreover that elements in $\mathcal{X}$ have uniformly bounded degrees, then $(\mathcal{X},Dist)$ is a compact space.\\
\indent Let us consider the ASM on an (infinite) rooted graph $(\Gamma,v)\in\mathcal{X}$. As usual, we shall approximate it by an exhaustive sequence of subgraphs, however, in this non-homogeneous situation, we shall require that all subgraphs in the exhaustion contain the root.

\begin{conv} \label{convDissipative} \rm \emph{(Choice of dissipative vertices)}\\
Given a infinite rooted graph $(\Gamma,v)$, and an exhaustion $\{H_n\}_{n\geq 1}$ of $(\Gamma,v)$ such that $v\in V(H_n)$ for each $n\geq 1$, set $P_n$, the set of dissipative vertices in each $H_n$, to be the \textbf{internal boundary} of $H_n$ in $\Gamma$, i.e., the vertices of $H_n$ that have neighbours in the complement $\Gamma\setminus H_n$.
\end{conv}

For all $n\geq 1$, consider the probability space $(\mathcal{R}_{H_n},\mu_n)$ (with the natural $\sigma$-algebra) where $\mathcal{R}_{H_n}$ is the set of recurrent configurations on the subgraph $H_n$ of $\Gamma$, and $\mu_n$ denotes the uniform distribution on $\mathcal{R}_{H_n}$. Define the random variable $Mav_{H_n}(\cdot,v):(\mathcal{R}_{H_n},\mu_n)\longrightarrow \mathbb{N}$ that maps a recurrent configuration on $H_n$ to the \textbf{mass} (i.e. the number of distinct vertices fired) of the avalanche triggered by adding to this configuration an extra chip on the root $v$. Note that the choice of the dissipative vertices specified in Convention \ref{convDissipative} ensures that the distance between the vertex on which we add an extra chip to some recurrent configuration to trigger avalanches (we have chosen the root) and the dissipative vertices grows as $n\to\infty$.

\begin{defi} \label{DefCrit}
Let $(\Gamma,v)$ be an infinite rooted graph and let $\{H_n\}_{n\geq 1}$ be as in Convention \ref{convDissipative}. We say that the ASM on the sequence $\{H_n\}_{n\geq 1}$ approximating $(\Gamma,v)$ has critical behaviour (with respect to the mass of avalanches) if there are constants $C_1,C_2>0$ such that, for any $\epsilon>0$ there exists $M_\epsilon\geq 1$ such that for any $M>M_\epsilon$,

\begin{equation} \label{eqCriticality}
C_1M^{-\delta-\epsilon}\leq\lim_{n\to\infty}\mathbb{P}_{\mu_n}(Mav_{H_n}(\cdot,v)=M)\leq C_2M^{-\delta+\epsilon}
\end{equation}

\noindent for some exponent $\delta>0$ (called the critical exponent). If this is the case, we write\\ $\lim_{n\to\infty}\mathbb{P}_{\mu_n}(Mav_{H_n}(\cdot,v)=M)\sim M^{-\delta}$.
\end{defi}

\begin{rem}\rm
If we denote $L(M):=\lim_{n\to\infty}\mathbb{P}_{\mu_n}(Mav_{H_n}(\cdot,v)=M)$, then criticality (condition \eqref{eqCriticality}) implies that $\lim_{M\to\infty}\log(L(M))/\log(M)=-\delta$.
\end{rem}

\indent Note that, depending on the geometry of the underlying graph, it may happen that not every integer $M$ can be realized as the mass of an avalanche. In such situations, we restrict our considerations to those integers which can be realized as the mass of an avalanche.\\
\indent The existence of the limit in Definition \ref{DefCrit} is, \emph{a priori}, not obvious; it is well-defined if the measures $\mu_n$ converge weakly, as $n\to\infty$, to some probability measure $\mu$. This has been proven in the case of the regular tree \cite{MaesRedSaa02} and of the lattice $\mathbb{Z}^d$. In the case of $\mathbb{Z}^d$ with $d=1$, the limit $\mu$ is the Dirac measure concentrated on the constant recurrent configuration $c\equiv1$ \cite{MaesRedSaaMoff00}. If $d\geq 2$, it is proven in \cite{AthJar04} that the measures $\mu_n$ weakly converge to a translation invariant probability measure $\mu$; for $2\leq d\leq 4$, this holds for any exhaustion $\{H_n\}_{n\geq 1}$ of $\mathbb{Z}^d$, and the measure $\mu$ is independent of the exhaustion; for $d>4$, there is an extra condition on the geometry of the $H_n$, but the authors conjecture that the former stronger version also holds. The proof is based on the bijection between recurrent configurations and spanning trees, and uses the fact that the uniform distribution on spanning trees on $H_n$ (with \emph{wired} boundary conditions) converges weakly to a probability measure supported by spanning forests on $\mathbb{Z}^d$ and called the Wired Uniform Spanning Forest (WUSF) \cite{BenLyPeSch01}. If $2\leq d\leq 4$, the WUSF is almost surely a one-ended tree. The proof of convergence of the $\mu_n$'s in this case directly applies to any infinite graph $\Gamma$ such that the WUSF on $\Gamma$ is almost surely a one-ended tree. (Of course, if $\Gamma$ is not transitive, one cannot expect translation invariance of the limit measure $\mu$.) If $d>4$, then the WUSF has almost surely infinitely many connected components, which makes the proof of the convergence of the measures $\mu_n$ more complicated, and not directly adaptable to general infinite graphs whose WUSF has many connected components. Luckily, all examples that we consider in this paper satisfy the condition that the WUSF is almost surely a one-ended tree. The proof of the following statement is the same than the proof of Theorem 1 in \cite{AthJar04}, in the case $2\leq d\leq 4$.

\begin{thm} \label{ThmWeakConv}
Let $\Gamma$ be an infinite graph such that the WUSF on $\Gamma$ is almost surely a one-ended tree. Then, for any rooting $(\Gamma,v)$ of $\Gamma$ and for any exhaustion $\{H_n\}_{n\geq 1}$ of $(\Gamma,v)$ satisfying Convention \ref{convDissipative}, the measures $\mu_n$ converge weakly to a measure $\mu$ which is independent of the choice of the exhaustion.
\end{thm}

\begin{cor}
Under the assumptions of Theorem \ref{ThmWeakConv}, the limit $\lim_{n\to\infty}\mathbb{P}_{\mu_n}(Mav_{H_n}(\cdot,v)=M)$ exists and does not depend on the exhaustion $\{H_n\}_{n\geq 1}$.
\end{cor}

\begin{proof}
Note that, for fixed $0<M<\infty$, observing an avalanche of mass $M$ triggered at $v$ is a cylinder event. Indeed, the set of vertices of $H_n$ fired during an avalanche triggered by adding an extra chip on $v$ induces a connected subgraph containing $v$. For every $n$ large enough, there exists $r_M$ such that the ball of radius $r_M$ centered in $v$ and contained in $H_n$ contains all vertices which may be involved in an avalanche of mass not greater than $M$. (See also Remark 1. (v) in \cite{AthJar04}.)
\end{proof}

We now turn to criticality of the ASM on unimodular random rooted graphs. Let us recall (see \cite{AlLyons07} and references therein) that a \textbf{unimodular random rooted graph} is a probability distribution $\rho$ on the space $\mathcal{X}$ (with respect to the Borel $\sigma$-algebra) which satisfies

\begin{displaymath} 
\int\sum_{w\in V(\Gamma)}f(\Gamma,v,w)d\rho(\Gamma,v)
=\int\sum_{w\in V(\Gamma)}f(\Gamma,w,v)d\rho(\Gamma,v)
\end{displaymath}

\noindent for all Borel functions $f:\widetilde{\mathcal{X}}\longrightarrow [0,\infty]$, where $\widetilde{\mathcal{X}}$ denotes the space of isomorphism classes of locally finite connected graphs with an ordered pair of distinguished vertices, and the natural topology thereon.


\begin{defi}
Let $\rho$ be an infinite unimodular random rooted graph. We say that the ASM is $\rho$-critical, with critical exponent $\delta$, if it is critical, with critical exponent $\delta$ (in the sense od Definition \ref{DefCrit}), for $\rho$-almost every rooted graph.
\end{defi}

\begin{rem}\rm
Note that the classical setup for studying the ASM, that is, a sequence of finite graphs exhausting $\mathbb{Z}^d$ fits into our, more general setup and corresponds to the case where the measure $\rho$ is the atom supported by $\mathbb{Z}^d$.
\end{rem}

It is an important open question (see \cite{AlLyons07}) whether all unimodular random rooted graphs on $\mathcal{X}$ can be obtained as limits of finite graphs in the following sense introduced by Benjamini and Schramm in \cite{RandWeakLimit}. Given a sequence $\{\Gamma_n\}_{n\geq 1}$ of finite unrooted graphs, $\rho$ is \textbf{the random weak limit} of $\{\Gamma_n\}_{n\geq 1}$ if the sequence $\{\rho_n\}_{n\geq 1}$ converges weakly to $\rho$ where, for every $n$, $\rho_n$ is the probability distribution on $\mathcal{X}$ induced by choosing a root in $\Gamma_n$ uniformly at random. It is an easy observation that any random weak limit of finite graphs is unimodular.\\
\indent All examples of unimodular random rooted graphs that we consider in this paper are constructed as random weak limits of sequences of finite graphs.

\begin{defi} \label{defCritRWL}
Given a sequence $\{\Gamma_n\}_{n\geq 1}$ of finite unrooted graphs with random weak limit $\rho$, we will say that the ASM on the sequence $\{\Gamma_n\}_{n\geq 1}$ is critical in the random weak limit (with critical exponent $\delta$) if it is $\rho$-critical (with critical exponent $\delta$).
\end{defi}

\begin{rem} \label{remsubgraphsH_n} \rm
In concrete situations, provided with a sequence $\{(\Gamma_n,v_n)\}_{n\geq 1}$ of finite rooted graphs converging in $\mathcal{X}$ to an infinite rooted graph $(\Gamma,v)$, it is sometimes convenient to think of the sequence $\{H_n\}_{n\geq 1}$ approximating $(\Gamma,v)$ (see Definition \ref{DefCrit}), as a sequence of subgraphs of the finite graphs $(\Gamma_n,v_n)$ rather than subgraphs of the limit graph $(\Gamma,v)$. By definition of convergence in $\mathcal{X}$, one can always choose the exhaustion $\{H_n\}_{n\geq 1}$ of $(\Gamma,v)$ so that, for each $n$, $\Gamma_n$ contains a subgraph isomorphic to $H_n$ and containing the root $v_n$. In such a case, we may write $\lim_{n\to\infty}\mathbb{P}_{\mu_n}(Mav_{H_n}(\cdot,v_n)=M)$ instead of $\lim_{n\to\infty}\mathbb{P}_{\mu_n}(Mav_{H_n}(\cdot,v)=M)$.
\end{rem}

\subsection{ASM on separable graphs} \label{CT}

\indent For $k\in \mathbb{N}^\ast$, a graph $\Gamma=(V,E)$ is \textbf{k-connected} if $|V|>k$ and $\Gamma\backslash X$ is connected for every subset $X\subset V$ with $|X|<k$. A connected graph $\Gamma$ is \textbf{separable} if it can be disconnected by removing a single vertex. Such a vertex is called a \textbf{cut vertex}. Note that non-separability of a connected graph is the same as $2$-connectedness. The largest $2$-connected components of a separable graph are called \textbf{blocks}. Any cut vertex belongs to at least two different blocks.\\
\indent Separable graphs belong to a wider class of tree-like graphs. Computations of certain critical values for percolation and Ising model for such graphs can be found in the Ph.D. thesis of Spakulova \cite{Kozakova}. The study of the ASM on separable graphs is also simplified thanks to its tree-like structure, and in particular by the fact that the critical group of such a graph is a direct product of the critical groups of its blocks \cite{BachHarpNagn97}. In this paper, we will need more precise information about recurrent configurations (see Lemma \ref{ConfigIndlemma} below).

\begin{rem}\rm
From now until the end of Section \ref{ASM}, we will assume that $|P|=1$. Indeed, the results of the two forthcoming subsections will be applied in Sections \ref{AvOnBasilica}, \ref{SectionIAM} and \ref{LastSection} to graphs for which we will be able to choose one-element dissipative sets satisfying our Convention \ref{convDissipative}. The choice of the unique dissipative vertex will be explained in Convention \ref{RemOnCPinf} below.
\end{rem}

\indent Consider a finite separable graph $\Gamma$ with blocks $C_1,\dots,C_s$. Fixing one of the vertices (denote it $p$ and think it to be the dissipative vertex), induces the following \textbf{partial order on the vertices} of $\Gamma$. For $w,w'\in V$, we put $w'\succeq w$ if and only if $w$ lies on any path in $\Gamma$ joining $w'$ to $p$. For any $1\leq i\leq s$, let $p_i$ be the smallest element of $V(C_i)$ in this order. Then the following holds:

\begin{lem} \label{ConfigIndlemma}
Given $\Gamma$ a finite separable graph with blocks $C_1,\dots,C_s$ and a dissipative vertex $p$, a configuration $c$ on $\Gamma$ is recurrent if and only if for all $1\leq i\leq s$, the subconfiguration $c^i:V_0(C_i)\longrightarrow\mathbb{N}$ defined by $c^i(v):=c(v)-$\emph{outdeg}$_{C_i}(v)$ is recurrent on the subgraph $C_i$ with $p_i$ considered as the dissipative vertex. (Here, for a subgraph $H$ of $\Gamma$ and a vertex $v$ of $H$, outdeg$_{H}(v)$ stands for the number of edges connecting $v$ to the complement of $H$ in $\Gamma$.)
\end{lem}

\begin{proof}
Let $c$ be a configuration on $\Gamma$. Suppose that $c$ is recurrent, take a block $C_i$ and let $v\in V(C_i)$. By Theorem \ref{BA}, there exists a burning sequence for $c$ on $\Gamma$ $pv_1\dots v_{|V|-1}$. Since $c$ is stable, any vertex $v\in V$ must have some of its neighbours fired before being fired itself. Since every path joining $v$ to $p$ contains the vertex $p_i$, $v$ cannot be fired before $p_i$ does. On the other hand, once $p_i$ is fired, then every vertex of $C_i$ can be fired in the order provided by the sequence $pv_1\dots v_{|V|-1}$. In particular, there is a subsequence of $pv_1\dots v_{|V|-1}$ which is a burning sequence for $c^i$ on $C_i$ with $p_i$ set as the unique dissipative vertex.\\
\indent Conversely, if for each block $C_i$ the subconfiguration $c^i$ is recurrent, then one can fire vertices of $\Gamma$ as follows: after firing the vertex $p$, fire vertices belonging to the blocks containing $p$ according to the burning sequences provided by the Burning Algorithm applied consecutively to each of these blocks. Then, repeat the previous operation with the blocks sharing a vertex with the already fired blocks. Since there is a burning sequence for each block of $\Gamma$, all vertices of $\Gamma$ are eventually fired.
\end{proof}

The following definition and observation will be crucial in our study of avalanches further on.

\begin{defi}
A block-path of length $k$ in a separable graph $\Gamma$ is a sequence of $k$ distinct blocks of $\Gamma$ such that two consecutive blocks intersect.
\end{defi}

Given $w,w'\in V$, there is a unique block-path $\mathcal{C}_1\dots\mathcal{C}_r$ of minimal length such that $w\in \mathcal{C}_1$ and $w'\in \mathcal{C}_r$ (where possibly $\mathcal{C}_1\equiv\mathcal{C}_r$). We say then that $\mathcal{C}_1\dots\mathcal{C}_r$ \textbf{joins} $w$ to $w'$. Similarly, given $w\in V$ and $C$ a block of $\Gamma$, there is a unique block-path $\mathcal{C}_1\dots\mathcal{C}_r=C$ of minimal length such that $w\in \mathcal{C}_1$. We say then that $\mathcal{C}_1\dots\mathcal{C}_r$ joins $w$ to $C$.

\begin{prop} \label{RemOnAv}
Given a finite separable graph $\Gamma$ with a dissipative vertex $p$, and given a vertex $v\in V_0$, let $\mathcal{CP}_v:=\mathcal{C}_1\dots\mathcal{C}_r$ be the block-path joining $v$ to $p$. Then, the avalanche triggered by adding an extra chip on $v$ to some recurrent configuration $c$ depends only on the subconfigurations of $c$ on the blocks constituting $\mathcal{CP}_v$.
\end{prop}

\begin{proof}
If $c(v)<\deg(v)-1$, then the avalanche is trivial. If $c(v)=\deg(v)-1$, then $v$ becomes unstable after adding an extra chip, and a non-trivial avalanche is initiated. Consider the block-path $\mathcal{CP}_v$ joining $v$ to $p$, let $w$ be a separating vertex belonging to some block of $\mathcal{CP}_v$, and consider the subgraph $D(w)$ of $\Gamma$ induced by the set $\{v\in V_0|v\succeq w\}$ of all descendants of $w$. Since $c$ is recurrent, we can conclude by Theorem \ref{BA} and the proof of Lemma \ref{ConfigIndlemma}, that each time $w$ is fired, every successor $v\succ w$ is fired exactly once, and as a result the subconfiguration on $D(w)$ remains unchanged. This happens independently of the recurrent subconfiguration on $D(w)$. The statement follows.
\end{proof}

\subsection{ASM on cacti} \label{SubsectAvOnCT}

In this paper, we will be interested in a particular class of separable graphs called \lq\lq cacti\rq\rq.

\begin{defi} \label{DefCactus}
A separable graph $\Gamma$, possibly with loops, is a cactus if its blocks are either cycles (possibly of length $2$), or single edges.
\end{defi}

\indent The ASM on cacti is addressed in \cite{Mat09} where the identity of the critical group as well as inverses are explicitly realized in terms of configurations. Here we will be rather interested in finding the asymptotic of avalanches on finite approximations of infinite cacti; see Theorem \ref{ThmCP} below. In particular, we will be interested in the behaviour of avalanches in the random weak limit for a sequence of finite cacti. (Note that the limit of a sequence of finite rooted cacti in local convergence is again a cactus.) Our results indicate that the answer depends on such invariant of the infinite graph as the number of ends. More results in this direction are to be found in the forthcoming paper \cite{Mat09}.\\

\subsubsection{ASM on cycles} \label{Cycles}

As the building blocks of a cactus graph are cycles, we will start by recalling and stating some easy facts about the ASM on cycles, \cite{RuSen}, \cite{Mat09}, which will be useful later.\\
\indent Let $C$ be the cycle of length $|C|$ and let $V(C)=\{p,v_{1},v_{2},\dots,v_{|C|-1}\}$, where $p$ is the unique dissipative site and other vertices are numbered in the counterclockwise direction.

\begin{prop} \label{PropRecOnCycle}
1) There are exactly $|C|$ recurrent configurations  $c_0,\dots,c_{|C|-1}$ on $C$. They are given by

\begin{displaymath}
c_{j}(v_{i}) = \left\{ \begin{array}{ll}
0 & \textrm{if $i=j$,}\\
1 & \textrm{otherwise,}
\end{array} \right.
\end{displaymath}

\noindent and $c_{0}(v_{i})=1$ for $i,j=1,\dots,|C|-1$.

2) \emph{\cite{Mat09}} Let $\eta$ be a configuration on $C$ and let $c_j$ be a recurrent configuration. Then,

\begin{displaymath}
\lcr c_j+\eta\rcr=c_{\left[j-\sum_{k=1}^{|C|-1}\eta(v_{k})k\right]_{\mod |C|}},
\end{displaymath}

\noindent where $\lcr \cdot +\cdot\rcr$ denotes the result of adding configurations coordinatewise and then stabilizing.
\end{prop}

\begin{cor} \label{CorAddRecOnCycles}
If $\eta=t\cdot\delta_{v_k}$ for some $1\leq k\leq |C|-1$ and $t\geq 1$ (i.e. $\eta(v_k)=t$ and $\eta(v_i)=0$ for $i\neq k$), then

\begin{equation}
\lcr c_j+t\cdot\delta_{v_k}\rcr=c_{[j-tk]_{\mod |C|}}.
\end{equation}
\end{cor}

We now turn to avalanches on $C$. Note that the mass of any avalanche on $C$ is trivially bounded from above by $|C|-1$. Fix a vertex $v_{i_0}\in V(C)$ on which an extra chip is added. By symmetry, we can suppose without loss of generality that $2i_0\leq |C|$. As above, let $\mu$ denote the uniform distribution over the set of recurrent configurations.

\begin{prop} \emph{\cite{Mat09}} \label{MassAvCycle}
In the notations above,

\begin{displaymath}\mathbb{P}_\mu(Mav_{C}(\cdot,v_{i_0})=M)=
\left\{\begin{array}{cc}
 0 & \textrm{if $0<M< i_0$},\\
 \frac{1}{|C|} & \textrm{if $i_0\leq M\leq |C|-1-i_0$},\\
 \frac{2}{|C|} & \textrm{if $|C|-i_0\leq M<|C|-1$}.
\end{array}\right.
\end{displaymath}
\noindent Moreover, $\mathbb{P}_\mu(Mav_{C}(\cdot,v_{i_0})=0)=\mathbb{P}_\mu(Mav_{C}(\cdot,v_{i_0})=|C|-1)=\frac{1}{|C|}$.
\end{prop}

\begin{proof}
Since there are $|C|$ different recurrent configurations on $C$, there are at most $|C|$ distinct avalanches. The mass of an avalanche is zero (respectively $|C|-1$) if and only if the configuration on which we add the extra chip is $c_{i_0}$ (respectively $c_0$). We thus have $\mathbb{P}_\mu(Mav_{C}(\cdot,v_{i_0})=0)=\mathbb{P}_\mu(Mav_{C}(\cdot,v_{i_0})=|C|-1)=\frac{1}{|C|}$.\\
\indent Let $c_j$ be a recurrent configuration. If $i_0>j$, then the mass of the avalanche is given by $Mav_{C}(c_j,v_{i_0})=|C|-1-j,$ whereas
if $i_0<j$, it is given by $Mav_{C}(c_j,v_{i_0})=j-1$. Thus, if we fix $0<M<|C|-1$, there are at most two avalanches of mass $M$, more precisely:

\begin{itemize}
\item if $0< M< i_0$, then there is no recurrent configuration providing an avalanche of mass $M$;
\item if $i_0\leq M\leq |C|-1-i_0$, then there is one configuration providing an avalanche of mass $M$, which is $c_{M+1}$;
\item if $|C|-i_0\leq M<|C|-1$, then there are two configurations providing an avalanche of mass $M$, which are $c_{M+1}$ and $c_{|C|-1-M}$.
\end{itemize}
\end{proof}

\subsubsection{Avalanches on cacti} \label{SubsecCP}

\indent Our analysis of the dynamics of avalanches on infinite graphs associates with self-similar groups will be based on Theorem \ref{ThmCP} below, a general result about avalanches on finite approximations of infinite one-ended cacti.\\
\indent Let $(\Gamma,v)$ be an infinite one-ended cactus rooted at $v$. Note that there exists a unique block-path $\mathcal{CP}_{v}=\mathcal{C}_1\mathcal{C}_2\dots$ of infinite length in $\Gamma$ starting at $v$ (i.e., $v\in\mathcal{C}_1$ but $v\notin\mathcal{C}_2$). Using the notations from Subsection \ref{CT}, for each $i\geq 1$, $p_i$ denotes the cut vertex between $\mathcal{C}_i$ and $\mathcal{C}_{i+1}$. Removing $p_i$ disconnects $\Gamma$ into several connected components (one of which is infinite). Let $D(p_i)$ denotes the subgraph of $\Gamma$ consisting of the union of all finite components, together with $p_i$. Denote by $d_i$ the number of vertices in $D(p_i)$; we thus have an increasing sequence of positive integers $\{d_i\}_{i\geq 1}$.\\
\indent We can choose an exhaustion $\{H_n\}_{n\geq 1}$ of $(\Gamma,v)$ so that, for any $n\geq 1$, the internal boundary of $H_n$ consists of a unique vertex $p^{(n)}$ (see Convention \ref{convDissipative}); this vertex is a cut vertex in $\Gamma$ between two consecutive blocks of $\mathcal{CP}_{v}$.

\begin{conv} \label{RemOnCPinf} \rm \emph{(Choice of dissipative vertex in one-ended cacti)}\\
Given an infinite one-ended cactus $(\Gamma,v)$, let $\{H_n\}_{n\geq 1}$ be an exhaustion of $(\Gamma,v)$ such that, for each $n$, $v\in V(H_n)$ and the internal boundary of $H_n$ consists of a unique vertex $p^{(n)}$; set $p^{(n)}$ to be the unique dissipative vertex in $H_n$.
\end{conv}

\begin{thm} \label{ThmCP}
Let $(\Gamma,v)$ be an infinite one-ended cactus rooted at $v$. Let $\{H_n\}_{n\geq 1}$ be an exhaustion of $(\Gamma,v)$ as in Convention \ref{RemOnCPinf} and, for any $n\geq 1$, let $p^{(n)}$ be the dissipative vertex in $H_n$. Denote by $\mathcal{CP}^n_{v}=\mathcal{C}_1\dots \mathcal{C}_{r_n}\subset\mathcal{CP}_{v}$ the finite block-path in $H_n$ joining vertex $v$ to $p^{(n)}$. Suppose that $\sum_{j=1,|\mathcal{C}_j|>2}^{r_n}\frac{1}{|\mathcal{C}_j|}$ converges as $r_n\to\infty$. Then, for any integer $M$ large enough that occurs as the mass of an avalanche, we have

\begin{displaymath}
\frac{L}{2\cdot|\mathcal{C}_{i_M}|\cdot |\mathcal{C}_{i_M+1}|}\leq\lim_{n\to\infty}\mathbb{P}_{\mu_n}(Mav_{H_n}(\cdot,v)=M)\leq \frac{2}{|\mathcal{C}_{i_M}|\cdot |\mathcal{C}_{i_M+1}|}
\end{displaymath}

\noindent where $0<L\leq 1$, and the index $i_M$ is uniquely determined by the condition $d_{i_M-1}\leq M<d_{i_M}$.
\end{thm}

\begin{proof}
Consider the subgraph $H_n$ of $\Gamma$ for some fixed $n\geq 1$ and let $c$ be a recurrent configuration on $H_n$. If $s$ is the number of blocks constituting $H_n$, $c$ can be decomposed into $s$ subconfigurations $c^1,\dots,c^s$ where $c^i$ is a recurrent configuration on the block $C_i$ (see Lemma \ref{ConfigIndlemma}). If $c(v)=\deg(v)-1$, then upon adding an extra chip on $v$, an avalanche starts on $\mathcal{C}_1$ which possibly extends to further blocks of $\mathcal{CP}^n_{v}$. Since the order of firings does not matter, we can suppose that one starts stabilizing the subconfiguration on $\mathcal{C}_{j+1}$ only when the subconfiguration on $\mathcal{C}_j$ is already stable. Recall that, by Proposition \ref{RemOnAv}, it is enough to keep track of the subconfigurations of $c$ on the blocks of $\mathcal{CP}^n_{v}$. For any $1\leq j\leq r_n$, we say that the avalanche reaches the block $\mathcal{C}_j$ if $p_{j-1}$ is fired during the avalanche. Note that if $\mathcal{C}_j$ is a single edge, then $c^j(p_{j-1})=\deg(p_{j-1})-1$. Once an avalanche has reached $\mathcal{C}_j$ and if $\mathcal{C}_j$ is not a single edge, then the subavalanche on $\mathcal{C}_j$ has two \lq\lq branches\rq\rq , each of them propagating in direction of $p_j$ along a path joining $p_{j-1}$ to $p_j$. Since the subconfiguration $c^j$ on $\mathcal{C}_j$ is recurrent, there is at most one vertex $w\in V(\mathcal{C}_j)\backslash\{p_j\}$ such that $c^j(w)=\deg(w)-2$ (see Proposition \ref{PropRecOnCycle}). Hence, at least one of the branches of the subavalanche extends to $p_j$ so that at least one chip reaches $p_j$. Then, if $p_j$ is not fired, we say that the avalanche stops on $\mathcal{C}_j$.\\
\indent With every recurrent configuration $c$ on $H_n$, one associates a sequence of positive integers $\{t_j(c)\}_{j=0}^{r_n-1}$, where $t_j(c)$ is the number of chips that have reached $p_j$ during the avalanche triggered by adding an extra chip to $c$. By convention, fix $t_0(c)\equiv 1$. Recall that, on a cycle $C_j$, there are $|C_j|$ recurrent configurations  which are $c^j_0,\dots,c^j_{|C_j|-1}$. For $1\leq j\leq r_n-1$, three following situations may occur:

\begin{enumerate}[(S1)]
\item $t_j(c)-t_{j-1}(c)=1$: this occurs if and only if $c^j=c_0^j$ and $\lcr c_0^j+t_{j-1}(c) \cdot\delta_{p_{j-1}}\rcr\neq c_0^j$;
\item $t_j(c)-t_{j-1}(c)=0$: this occurs if and only if either $c^j=c^j_0$ and $\lcr c_0^j+t_{j-1}(c) \cdot\delta_{p_{j-1}}\rcr= c_0^j$, or $c^j=c_k^j$ for some $0<k<|\mathcal{C}_j|$ and $\lcr c_k^j+t_{j-1}(c)\cdot\delta_{p_{j-1}}\rcr\neq c_0^j$;
\item $t_j(c)-t_{j-1}(c)=-1$: this occurs if and only if $c^j=c_k^j$ for some $0<k<|\mathcal{C}_j|$ and $\lcr c_k^j+t_{j-1}(c)\cdot\delta_{p_{j-1}}\rcr= c_0^j$.
\end{enumerate}

\noindent The difference $|t_j(c)-t_{j-1}(c)|$ cannot be greater than one, since the total amount of chips in a recurrent configuration on a cycle $C_j$ is either $|C_j|-1$ or $|C_j|-2$ (see Proposition \ref{PropRecOnCycle} and Lemma \ref{ConfigIndlemma}). Finally, note that if the block $\mathcal{C}_j$ is a single edge, then $t_j(c)=t_{j-1}(c)$. \\
\indent We consider now avalanches of some fixed mass $M$. Since we are interested in the asymptotic behaviour of avalanches as $n$ tends to infinity and since we have supposed that $r_n$ tends to infinity (as $n\to\infty$), we can suppose without loss of generality that $M<d_{r_n-1}$; all these avalanches reach some block $\mathcal{C}_{i_M}$, $1\leq i_M<r_n$ and stop on it (i.e. vertex $p_{i_M-1}$ is fired but not vertex $p_{i_M}$). Note that an avalanche cannot stop on a cycle of length two.\\
\indent Let us now find bounds on the number of recurrent configurations on $H_n$ producing avalanches of mass $M$. Let $c$ be such that the avalanche triggered by adding an extra chip to $c$ on $v$ is of mass $M$. Then, its corresponding sequence $\{t_j(c)\}_{j=0}^{r_n-1}$ satisfies

\begin{itemize}
 \item $t_j(c)\geq 1$ for all $0\leq j\leq i_M-1$;
 \item $t_{i_M}(c)= 1$;
 \item $t_j(c)=0$ for all $i_M< j\leq r_n-1$.
\end{itemize}

We have to distinguish two cases. Suppose that there exists $1\leq j_0<i_M$ such that $\mathcal{C}_{j_0}$ is a cycle of length two, and suppose that $j_0$ is the smallest such index. Consider the sequence $\{t_j\}_{j=0}^{r_n-1}$ defined by $t_j=1$ if $j<j_0$, $t_j=2$ if $j_0\leq j< i_M$, $t_{i_M}=1$ and $t_j=0$ if $j> i_M$. If there is no index $j_0$ such that $\mathcal{C}_{j_0}$ is a cycle of length two, then define $\{t_j\}_{j=0}^{r_n-1}$ by $t_j=1$ if $j\leq i_M$ and $t_j=0$ if $j> i_M$. We count the number of recurrent configurations $c$ whose associated sequence $\{t_j(c)\}_{j=0}^{r_n-1}$ coincides with $\{t_j\}_{j=0}^{r_n-1}$. In the former case, it follows from Corollary \ref{CorAddRecOnCycles} that, for each $j<i_M$ such that $\mathcal{C}_j$ is not a single edge nor a cycle of length two, there are at least $|\mathcal{C}_j|-2$ recurrent subconfigurations on $\mathcal{C}_j$ satisfying the right-hand side of (S2). If $\mathcal{C}_j$ is a cycle of length two, and $j_0<j<i_M$, then both recurrent subconfigurations on $\mathcal{C}_j$ satisfy the right-hand side of (S2). The subconfiguration on $\mathcal{C}_{j_0}$ must be $c^{j_0}=c_0^{j_0}$ (see (S1)) whereas the subconfigurations on $\mathcal{C}_{i_M}$, $\mathcal{C}_{i_M+1}$ are uniquely determined by (S3). In the latter case, for each $j<i_M$ such that $\mathcal{C}_j$ is not a single edge, there are at least $|\mathcal{C}_j|-2$ recurrent subconfigurations on $\mathcal{C}_j$ satisfying the right-hand side of (S2). Consider now the subavalanche on the cycle $\mathcal{C}_{i_M}$, denoting its mass by $m$ (so that $d(p_{i_M-1},p_{i_M})\leq m<|\mathcal{C}_{i_M}|$). By Proposition \ref{MassAvCycle} and its proof, at least one but at most two subconfigurations on $\mathcal{C}_{i_M}$ provoke subavalanches of such mass. The subconfiguration on $\mathcal{C}_{i_M+1}$ is uniquely determined by (S3). Finally, in both cases, the configurations on the remaining blocks of $H_n$ can be chosen freely since they do not influence the avalanche (see Proposition \ref{RemOnAv}). Thus, the number $N$ of recurrent configurations on $H_n$ producing an avalanche of mass $M$ is at least
\begin{displaymath}
N\geq 2^R\cdot\prod_{\underset{|\mathcal{C}_j|>2}{j=1}}^{i_M-1}(|\mathcal{C}_j|-2)\cdot\prod_{\underset{ C_j\neq \mathcal{C}_1,\dots,\mathcal{C}_{i_M+1}}{C_j\subset\Gamma_n}} |C_j|,
\end{displaymath}
\noindent where $R={|\{\mathcal{C}_j;\hspace{1mm}\text{$\mathcal{C}_j$ is a cycle, $|\mathcal{C}_j|=2$, $j_0<j<i_M$}\}|}$ and the latter product runs over blocks of $H_n$ which are not single edges. Since the total number of recurrent configurations on $H_n$ (which is the number of spanning trees of $H_n$) is equal to the product of the lengths of the cycles in $H_n$, the probability of observing an avalanche of mass $M$ on $H_n$ upon adding an extra chip on $v$ is bounded from below by
\begin{equation} \label{eqLowerBound}
\mathbb{P}_{\mu_n}(Mav_{H_n}(\cdot,v)=M)\geq N\prod_{C_j\subset\Gamma_n} |C_j|^{-1}
\geq\frac{1}{2\cdot|\mathcal{C}_{i_M}|\cdot |\mathcal{C}_{i_M+1}|}\prod_{\underset{|\mathcal{C}_j|>2}{j=1}}^{i_M-1}\left(1-\frac{2}{|\mathcal{C}_j|}\right),
\end{equation}
\noindent where the former product runs over blocks of $H_n$ which are not single edges. The upper-bound
\begin{equation} \label{eqUpperBound}
\mathbb{P}_{\mu_n}(Mav_{H_n}(\cdot,v)=M)\leq\frac{2}{|\mathcal{C}_{i_M}|\cdot|\mathcal{C}_{i_M+1}|}
\end{equation}
\noindent follows from the fact that in the former case, the subconfigurations on $\mathcal{C}_{i_M}$ and $\mathcal{C}_{i_M+1}$ are uniquely determined by (S3) whereas in the latter case, there are at most two subconfigurations on $\mathcal{C}_{i_M}$ producing a subavalanche on $\mathcal{C}_{i_M}$ of mass $m$ and the subconfiguration on $\mathcal{C}_{i_M+1}$ is uniquely determined.\\
\indent The product $\prod_{j=1,|\mathcal{C}_j|>2}^{i_M-1}\left(1-2/|\mathcal{C}_j|\right)$ converges as $i_M\to\infty$ to a limit $L>0$ if and only if, the series $\sum_{j\geq 1,|\mathcal{C}_j|>2}\frac{2}{|\mathcal{C}_j|}$ converges (see for instance \cite{Tannery}). In such a case, it is bounded by $L$ from below for every $i_M\geq 2$. This completes the proof.

\end{proof}

\section{Actions on Rooted Trees and their Schreier Graphs} \label{SelfSimilarGroups}

\setcounter{defi}{0}

Let $\{q_n\}_{n\geq 0}$ be a sequence of positive integers and let $T$ be a rooted tree such that all vertices of the $n$-th level of $T$ (i.e. vertices situated at distance $n$ from the root) have $q_n$ children; $T$ is called \textbf{spherically homogenous} and $\{q_n\}_{n\geq 0}$ is the spherical index of $T$. For any $n\geq 1$, let $X_n$ be a $q_n$-letters alphabet. Then, any vertex of the $n$-th level of $T$ can be regarded as an element of $\prod_{i=1}^nX_i=:L_n$ (the root is viewed as the empty word.) Also, write $X^{\omega}:=\prod_{n\geq1}X_n$, which is the set of infinite words $\xi$ such that, for any $n\geq 1$, the $n$-th letter of $\xi$ belongs to $X_n$. The set $X^{\omega}$ can be identified with the boundary $\partial T$ of the tree, which is defined as the set of infinite geodesic rays starting at the root of $T$. The cylindrical sets $\bigcup_{n\geq 0}\{w\prod_{i>n}X_i|w\in L_n\}$ generate the $\sigma$-algebra of Borel subsets of the space $X^{\omega}$. We shall denote by $\lambda$ the uniform measure on $X^\omega$.\\
\indent Consider the group $Aut(T)$ of all automorphisms of $T$, i.e., the group of all bijections of the set of vertices of $T$ preserving the root and the incidence relation; the levels of the tree are thus preserved by any automorphism of $T$. A group $G\leq Aut(T)$ is said to be \textbf{spherically transitive} if it acts transitively on each level of the tree.\\
\indent For $G<Aut(T)$ we define the following subgroups of $G$: the \textbf{stabilizer} of a vertex $v\in T$ in $G$ by $Stab_G(v)=\{g\in G| g(v)=v\}$; the \textbf{stabilizer of the $n$-th level} of the tree in $G$ by
$Stab_G(L_n)=\bigcap_{v\in L_n}Stab_G(v)$; finally, the \textbf{stabilizer of a boundary point} $\xi\in X^{\omega}$ in $G$ by $Stab_G(\xi)=\{g\in G| g(\xi)=\xi\}$. Suppose that $G$ is spherically transitive; then, the following properties hold:
\begin{itemize}
\item The subgroups $Stab_G(v)$, for $|v|=n$, are all conjugate and of index $\prod_{i=1}^nq_i$.
\item $\bigcap_{\xi \in \partial T}Stab_G(\xi)$ is trivial.
\item Denote by $\xi_n$ the prefix of $\xi$ of length $n$. Then $Stab_G(\xi) = \bigcap_{n\in \mathbb{N}} Stab_G(\xi_n)$.
\item $Stab_G(\xi)$ has infinite index in $G$. 
\end{itemize}

\indent Consider a finitely generated group $G$ with a set $S$ of generators such that $id \not \in S$ and $S = S^{-1}$, and suppose that $G$ acts on a set $M$. Then, one can consider a graph $\Gamma(G,S,M)$ with the set of vertices $M$, and two vertices $m,m'$ joined by an oriented edge labeled by $s$ if there exists $s\in S$ such that $s(m)=m'$. If the action of $G$ on $M$ is transitive, then $\Gamma(G,S,M)$ is the \textbf{Schreier graph} $\Gamma(G,S,Stab_G(m))$ of the group $G$ with respect to the subgroup $Stab_G(m)$ for some (any) $m\in M$. If the action of $G$ on $M$ is not transitive, and $m\in M$, then we denote by $\Gamma(G,S,m)$ the Schreier graph of the action on the $G$-orbit of $m$, and we call such a graph an \textbf{orbital Schreier graph}. In what follows, we will often forget about labels. Also, since $S = S^{-1}$, our graphs are graphs in the sense of Serre \cite{Serre}.\\
\indent Suppose now that $G$ acts spherically transitively on a spherically homogeneous rooted tree $T$ with spherical index $\{q_n\}_{n\geq 0}$. Then, the \textbf{$n$-th Schreier graph} of $G$ is by definition
$\Gamma_n:=\Gamma(G,S,L_n)=\Gamma(G,S,P_n)$ where $P_n$ denotes the subgroup stabilizing some word $w\in L_n$. For each $n\geq 1$, let $\pi_{n+1}:\Gamma(G,S,L_{n+1})\longrightarrow \Gamma(G,S,L_n)$
be the map defined on the vertex set of $\Gamma(G,S,L_{n+1})$ by $\pi_{n+1}(x_1\ldots x_nx_{n+1}) = x_1\ldots x_n$. Since $P_{n+1}\leq P_n$, $\pi_{n+1}$ induces a surjective morphism between $\Gamma(G,S,L_{n+1})$ and
$\Gamma(G,S,L_n)$. This morphism is a graph covering of degree $q_n$.\\
\indent We also consider the action of $G$ on $\partial T\equiv X^\omega$ and the orbital Schreier graphs $\Gamma_{\xi}:=\Gamma(G,S,G\cdot\xi)=\Gamma(G,S,P_\xi)$ where $P_\xi$ denotes the stabilizer of $\xi$ for the action of $G$ on $X^\omega$. Recall that, given a ray $\xi$, we denote by $\xi_n$ the prefix of $\xi$ of length $n$, and that $P_{\xi} = \bigcap_n P_n$. It follows that the infinite Schreier graph $(\Gamma_{\xi},\xi)$ rooted at $\xi$ is the limit of  finite Schreier graphs $(\Gamma_n,\xi_n)$ rooted at $\xi_n$, as $n\rightarrow\infty$, in the compact metric space $(\mathcal{X},Dist)$ (of rooted isomorphism classes) of rooted connected graphs with uniformly bounded degrees (see Subsection \ref{ASMonGamma_n}). Orbital Schreier graphs are interesting infinite graphs that contain information about the group and its action on the tree. The random weak limit of the sequence $\{\Gamma_n\}_{n\geq 1}$, concentrated on the classes of rooted-isomorphism of the orbital Schreier graphs $\{(\Gamma_{\xi},\xi)\}_{\xi\in\partial T}$, is often a continuous measure, see e.g. \cite{DanDonMat09}.

\begin{prop} \label{PropRegGraphsSchreier}
Let $\{\Gamma_n\}_{n\geq 1}$ be a covering sequence of finite $2k$-regular graphs ($k\in\mathbb{N}^\ast$). Then, there exists a rooted tree $T$, and a group $G$ of automorphisms of $T$ such that the $\Gamma_n$'s can be realized as Schreier graphs (with respect to an appropriate set of generators) of the action of $G$ on $T$.
\end{prop}

\begin{proof}
For any $n\geq 1$, let $q_n$ be the degree of the covering $\pi_{n+1}:\Gamma_{n+1}\longrightarrow \Gamma_n$. One associates a \emph{tree of preimages} $T$ with the covering sequence $\{\Gamma_n\}_{n\geq 1}$ as follows: $T$ is an infinite rooted tree with vertex set $\bigcup_{n\geq 1}V(\Gamma_n)\cup \{\ast\}$ such that the $n$-th level of $T$ is $V(\Gamma_n)$ and every vertex $v$ of the $n$-th level has $q_n$ children corresponding to the fibre of $v$ in $\Gamma_{n+1}$ (by convention, the root $\ast$ of $T$ has $|V(\Gamma_1)|$ children.) For any $n\geq 1$, we denote by $T_{[n]}$ the rooted subtree of $T$ of height $n$.\\
\indent We proceed by induction on $n$. Consider the graph $\Gamma_1$. By a theorem of Petersen (see for instance \cite{Diestel}), every $2k$-regular graph has a $2$-factor, that is, a $2$-regular spanning subgraph. Denote by $F^1_1,\dots,F^1_k$ the decomposition of $\Gamma_1$ into $2$-factors. Any $F^1_i$ is a collection $C^i_{1},\dots,C^i_{t_i}$ of disjoint cycles. Assign an arbitrary orientation to each of them, so that each $2$-factor $F^1_i$ determines a unique permutation $\sigma^1_i$ of the vertex set of $\Gamma_1$. For $1\leq i\leq k$, label the edges of the cycles $C^i_{1},\dots,C^i_{t_i}$ in $\Gamma_1$ by $\sigma^1_i$. Consider the subgroup $G_1$ of automorphisms of $T_{[1]}$ generated by the set of permutations $\{\sigma^1_1,\dots,\sigma^1_k\}$. The Schreier graph $\Gamma(G_1,\{\sigma^1_1,\dots,\sigma^1_k\},V(\Gamma_1))$ coincides with $\Gamma_1$ labeled as above.\\
\indent Suppose that there is a subgroup $G_n=<\sigma^n_1,\dots,\sigma^n_k>$ of automorphisms of $T_{[n]}$ such that for every $1\leq m\leq n$, $\Gamma(G_n,\{\sigma^n_1,\dots,\sigma^n_k\},V(\Gamma_m))$ coincides with $\Gamma_m$. Every automorphism $\sigma^n_i$ corresponds to a $2$-factor $F_i^n$ of $\Gamma_n$ (i.e. $F_i^n$ with its edges labeled by $\sigma^n_i$ coincides with $\Gamma(<\sigma^n_i>,\{\sigma^n_i\},V(\Gamma_n))$.) We construct the group $G_{n+1}$ by extending every automorphism $\sigma^n_i$ to an automorphism $\sigma^{n+1}_i$ of $T_{[n+1]}$. Consider the $2$-factor $F^n_i=\bigcup_{s=1}^{t_i} C^i_s$. For any $v\in V(\Gamma_n)$, number its children in $T$ by $v_1,\dots,v_{q_n}$; consider the (unique) cycle $C^i_s$ containing $v$ together with its fibre $C^i_{s1},\dots,C^i_{sr_s}$ ($1\leq r_s\leq q_n$) in $\Gamma_{n+1}$. The orientation of $C^i_s$ induces an orientation on each cycle of the fibre. For $1\leq l \leq r_s$, consider $C^i_{sl}$ and a child $v_j\in V(C^i_{sl})$ of $v$. If the neighbour (with respect to the induced orientation) of $v_j$ in $C^i_{sl}$ is a child $w_{j'}$ of $w\in V(\Gamma_n)$, then let the automorphism $\sigma^{n+1}_i$ transpose vertices $v_j$ and $v_{j'}$ in $T_{[n+1]}$. Then, consider $w_{j'}$ together with its next neighbour $u_{j''}$ in $C^i_{sl}$ and let $\sigma^{n+1}_i$ transpose vertices $w_{j'}$ and $w_{j''}$. Continue like this along $C^i_{sl}$ until $v_j$ is reached again. We thus obtain a set $\{\sigma^{n+1}_1,\dots,\sigma^{n+1}_k\}$ of automorphisms of $T_{[n+1]}$ such that the restriction of every $\sigma^{n+1}_i$ to $T_{[n]}$ is $\sigma^n_i$. For every $1\leq i\leq k$, label the edges of $\Gamma_{n+1}$ belonging to the fibre of $F^n_i$ by $\sigma^{n+1}_i$. By construction, the subgroup $G_{n+1}$ of $Aut(T_{[n+1]})$ generated by these automorphisms is such that $\Gamma(G_{n+1},\{\sigma^{n+1}_1,\dots,\sigma^{n+1}_k\},V(\Gamma_{n+1}))$ coincides with $\Gamma_{n+1}$ labeled as above.\\
\indent For $i=1,\dots,k$, consider the automorphisms of $T$ defined by $\sigma_i:=\lim_{n\to\infty}\sigma^n_i$, and let $G=<\sigma_1,\dots,\sigma_k>$ be the subgroup of $Aut(T)$ generated by these elements. As, for any $n\geq 1$, $\sigma^n_i$ is the restriction of $\sigma_i$ to $T_{[n]}$, replace in each $\Gamma_n$ the labels $\sigma^n_i$ by $\sigma_i$ for $1\leq i\leq k$. Then, for any $n\geq 1$, the Schreier graph $\Gamma(G,\{\sigma_1,\dots,\sigma_k\},V(\Gamma_n))$ of the action of $G$ on the $n$-th level of $T$ coincides with $\Gamma_n$ newly labeled.
\end{proof}

It follows from a result of Nekrashevych (see Theorem \ref{ThmNekrashevych} below) that if $\{\Gamma_n\}_{n\geq 1}$ is a covering sequence of finite $2k$-regular cacti, and only then, the corresponding group of automorphisms of the tree of preimages (see Proposition \ref{PropRegGraphsSchreier}) is an \emph{iterated monodromy group} of a \emph{post-critically finite backward iteration of topological polynomials}. A post-critically finite backward iteration is a sequence $f_1,f_2,\dots$ of complex polynomials
(or orientation preserving branched coverings of planes) such that there exists a
finite set $\mathcal{P}$ with all critical values of $f_1\circ f_2\circ\dots\circ f_n$ belonging to $\mathcal{P}$ for every $n$. The iterated monodromy group of such a sequence is the automorphism group of
the tree of preimages $T_t=\bigsqcup_{n\geq 0}(f_1\circ f_2\circ\dots\circ f_n)^{-1}(t)$ induced by the monodromy action of the fundamental group $\pi_1(\mathbb{C}\backslash \mathcal{P},t)$, where $t$
is an arbitrary basepoint.

\begin{thm}[Nekrashevych \cite{Nekr09}] \label{ThmNekrashevych}
An automorphism group $G$ of a rooted tree $T$ is an iterated monodromy
group of a post-critically finite backward iteration of polynomials if and only if there exists a generating set of $G$ with respect to which the Schreier graphs of the action of $G$ on $T$ are cacti.
\end{thm}

Suppose now that the rooted tree $T$ is $q$-regular (i.e. $q_n=q$ for any $n\geq 0$.) Then, given a finite alphabet $X=\{0,1,\ldots,q-1\}$, any vertex of the $n$-th level of $T$ can be regarded as an element of $X^n$, the set of words of length $n$ in the alphabet $X$ ($X^0$ consists of the empty word), whereas the boundary $\partial T$ of $T$ is identified with $X^\omega$, the set of infinite words in $X$; write $X^{\ast}=\bigcup_{n\geq 0}X^n$.\\
\indent Given $g\in Aut(T)$ and $v\in X^*$, define $g|_v\in Aut(T)$, called  the \textbf{restriction of the action of $g$ to the subtree rooted at $v$}, by $g(vw)=g(v)g|_v(w)$ for all $w\in X^*$. For any vertex $v$ of the tree, the subtree of $T$ rooted at $v$ is isomorphic to $T$. Therefore, every automorphism $g\in Aut(T)$ induces a permutation of the vertices of the first level of the tree and $q$ restrictions, $g|_0,...,g|_{q-1},$ to the subtrees rooted at the vertices of the first level. It can be written as $g=\tau_g(g|_0,\ldots,g|_{q-1})$, where $\tau_g\in S_q$ describes the action of $g$ on the first level of the tree. In fact, $Aut(T)$ is isomorphic to the wreath product $S_q\wr Aut(T)$ where $S_q$ denotes the symmetric group on $q$ letters, and thus $Aut(T)\cong \wr_{i=1}^{\infty} S_q$.\\
\indent For a subgroup $G<Aut(T)$, the natural question whether restricting the action to a subtree isomorphic to $T$ preserves $G$, motivates the following definition. It was forged around 2000, see e.g. \cite{GNS}, though self-similar groups were known before -- this class of groups contains many exotic examples of groups, including groups of intermediate growth, non-elementary amenable groups, amenable but not subexponentially amenable groups.

\begin{defi}
The action of a group $G$ by automorphisms on a $q$-regular rooted tree $T$ is self-similar if $g|_v\in G$, $\forall v\in X^*, \forall g\in G$.
\end{defi}

Consequently, if $G<Aut(T)$ is self-similar, an automorphism $g\in G$ can be represented as $g=\tau_g(g|_0,\ldots,g|_{q-1})$, where $\tau_g\in S_q$ describes the action of $g$ on the first level of the tree, and $g|_i\in G$ is the restriction of the action of $g$ to the subtree $T_i$ rooted at the $i$-th vertex of the first level. So, if $x\in X$ and $w$ is a finite word in $X$, we have $g(xw)=\tau_g(x)g|_x(w)$.\\
\indent Self-similar groups can be also characterized as \textbf{automata groups}, i.e., groups generated by states of an invertible automaton (see e.g. \cite{GNS}). An \textbf{automaton} over the alphabet $X$ with the set of states $\mathcal{S}$ is defined by the transition map $\mu: \mathcal{S}\times X \rightarrow \mathcal{S}$ and the output map $\nu: \mathcal{S}\times X \rightarrow X$. It is \textbf{invertible} if, for all $s\in\mathcal{S}$, the transformation $\nu(s, \cdot):X\rightarrow X$ is a permutation of $X$. It can be represented by its \emph{Moore diagram} where vertices correspond to states and for every state $s\in \mathcal{S} $ and every letter $x\in X$, an oriented edge connects $s$ with $\mu(s,x)$ labeled by $x|\nu(s,x)$. A natural action on the words over $X$ is induced, so that the maps $\mu$ and $\nu$ can be extended to $\mathcal{S}\times X^{\ast}$: $\mu(s,xw) =\mu(\mu(s,x),w)$, $\nu(s,xw) = \nu(s,x)\nu(\mu(s,x),w)$, where we set $\mu(s,\emptyset) = s$ and $\nu(s,\emptyset) = \emptyset$. If we fix an initial state $s$ in an automaton $\mathcal{A}$, then the transformation $\nu(s,\cdot)$ on the set $X^{\ast}$ is thus defined; it is denoted by $\mathcal{A}_s$. The image of a word $x_1x_2\ldots$ under $\mathcal{A}_s$ can be easily found using the Moore diagram: consider the directed path starting at the state $s$ with consecutive labels $x_1|y_1$, $x_2|y_2,...$; the image of the word $x_1x_2\ldots$ under the transformation $\mathcal{A}_s$ is then $y_1y_2\ldots$. More generally, given an invertible automaton $\mathcal{A}=(\mathcal{S},X,\mu,\nu)$, one can consider the group generated by the transformations $\mathcal{A}_s$, for $s\in \mathcal{S}$; this group is called the \textbf{automaton group} generated by $\mathcal{A}$ and is denoted by $G(\mathcal{A})$.\\
\indent To a group with a self-similar action that is \textbf{contracting}, (which means the existence of a finite set $\mathcal{N}\subset G$ such that for every $g\in G$ there exists $k\in \mathbb{N}$ such that $g|_v\in \mathcal{N}$, for all words $v$ of length greater or equal to $k$), Nekrashevych associates its \emph{limit space} $\mathcal J(G)$, often a fractal. Rescaled finite Schreier graphs form a sequence of finite approximations to the compact $\mathcal J(G)$. Orbital Schreier graphs $\Gamma_{\xi}$ on the other hand describe the local structure of the limit space.\\
\indent An important class of self-similar groups is formed by iterated monodromy groups of partial self-coverings of path connected and locally path connected topological spaces (e.g. of complex rational functions.) If the covering is expanding, its Julia set is homeomorphic to the limit space of its iterated monodromy group. Details about this very interesting subject can be found in \cite{NekBook}.

\section{Invariance property of avalanches of the ASM on cacti} \label{SubsecErgodicityAval}

\setcounter{defi}{0}

In this section, we return to studying avalanches on cacti. Our aim here is to show that Theorem \ref{ThmCP} can be applied not only to individual limits in the space $\mathcal{X}$ of rooted graphs but also in the random weak limit. More precisely, we show:

\begin{prop}\label{PropErgodicity}
Let $\{\Gamma_n\}_{n\geq 1}$ be a covering sequence of finite $2k$-regular cacti ($k\in\mathbb{N}^\ast$) such that the conditions of Theorem \ref{ThmCP} are satisfied in the random weak limit $\rho$. Then, asymptotically in M, the probability distribution $\lim_{n\to\infty}\mathbb{P}_{\mu_n}(Mav_{H_n}(\cdot,v)=M)$ (where $\{H_n\}_{n\geq 1}$ is an exhaustion of $(\Gamma,v)$ satisfying Convention \ref{RemOnCPinf}) is $\rho$-almost everywhere the same. In particular, the critical exponent is almost surely constant.
\end{prop}

The following lemma was explained to us by G. Elek:

\begin{lem} \label{LemPhiMeasurable}
Let $G\leq Aut(T)$ be a finitely generated spherically transitive group of automorphisms of a rooted tree $T$. Recall that $\lambda$ denotes the uniform measure on the boundary $\partial T$ of $T$ and consider the application $\phi:\partial T\longrightarrow \mathcal{X}$, $\phi(\xi):=(\Gamma_\xi,\xi)$, mapping a point $\xi\in\partial T$ to the (rooted isomorphism class of the) orbital Schreier graph $\Gamma_\xi$ rooted at $\xi$. Then $\phi$ is measurable and the image of $\lambda$ under $\phi$ is the random weak limit of the sequence $\{\Gamma_n\}_{n\geq 1}$ of finite Schreier graphs of the action of $G$ on the levels of $T$.
\end{lem}

\begin{proof}
The $\sigma$-algebra on $\mathcal{X}$ is generated by cylindrical sets of the form $C_{(H,w)}:=\{(\Gamma,v)|B_{\Gamma}(v,r)\simeq(H,w)\}$ where $r\in\mathbb{N}$ and $(H,w)$ is a finite rooted graph. We say that a vertex $v$ of $\Gamma$ has $r$-type $(H,w)$ if the ball of radius $r$ centered in $v$ is isomorphic to $(H,w)$.\\
\indent Fix $r\in\mathbb{N}$; for any $\xi\in\partial T$, there exists a smallest integer $n(\xi)$ such that the balls $B_{\Gamma_\xi}(\xi,r)$ and $B_{\Gamma_n}(\xi_n,r)$ are isomorphic for all $n\geq n(\xi)$. For any $n\geq 1$, given a finite rooted graph $(H,w)$, define the set $A_n:=\{v\in V(\Gamma_n)|\textrm{$v=\xi_{n(\xi)}$ for some $\xi\in\partial T$ and $v$ has $r$-type $(H,w)$}\}$. Also, define $B_n:=\bigcup_{m>n}\{w\in V(\Gamma_m)|\textrm{$w=\xi_{n(\xi)}$ for some $\xi\in\partial T$}\}$. Then,

\begin{displaymath}
\phi^{-1}(C_{(H,w)})=\bigcup_{n\geq 1}\left(\bigcup_{A_n}vX^\omega\backslash\left(\bigcup_{B_n}wX^\omega\cap\bigcup_{A_n}vX^\omega\right)\right),
\end{displaymath}

\noindent so that $\phi^{-1}(C_{(H,w)})$ is a Borel set, and thus $\phi$ is measurable.\\
\indent Note that the integer-valued function $\xi\mapsto n(\xi)$ is measurable; hence, for any $\epsilon>0$, there exists $n_\epsilon$ such that $\lambda(\{\xi\in\partial T|n(\xi)>n_\epsilon\})<\epsilon$. We claim that $\lambda(\phi^{-1}(C_{(H,w)}))=\lim_{n\to\infty}\frac{1}{|V(\Gamma_n)|}|\{v\in V(\Gamma_n)|\textrm{$v$ has $r$-type $(H,w)$}\}|$. Indeed, given $\epsilon>0$, we say that a vertex $v\in V(\Gamma_n)$ is $\epsilon$-bad if $\lambda(\{\textrm{$\xi\in vX^\omega|$ $\xi$ and $v$ have different $r$-types}\})>\epsilon/n$. We have

\begin{displaymath}
\lambda(\{\xi\in\partial T|n(\xi)>n_\epsilon\})=\sum_{v\in V(\Gamma_{n_\epsilon})}\lambda(\{\xi\in vX^\omega| n(\xi)>n_\epsilon\})<\epsilon.
\end{displaymath}

\noindent It is easy to check that the proportion of terms in the previous sum which are greater than $\sqrt{\epsilon}/n_\epsilon$ must be less than $\sqrt{\epsilon}$. Since, for any $\epsilon>0$ and $v\in V(\Gamma_{n_\epsilon})$, $\{\textrm{$\xi\in vX^\omega|$ $\xi$ and $v$ have different $r$-types}\}\subset\{\xi\in vX^\omega| n(\xi)>n_\epsilon\}$, it follows that the proportion of vertices in $\Gamma_{n_\epsilon}$ which are $\sqrt{\epsilon}$-bad is smaller than $\sqrt{\epsilon}$. This shows that the difference $\lambda(\phi^{-1}(C_{(H,w)}))-\frac{1}{|V(\Gamma_n)|}|\{v\in V(\Gamma_n)|\textrm{$v$ has $r$-type $(H,w)$}\}|$ can be made arbitrarily small by taking $n$ large enough.
\end{proof}

\begin{proof}[Proof of Proposition \ref{PropErgodicity}]
Observe that the conditions of Theorem \ref{ThmCP} are all measurable; in particular, the subset $\mathcal{C}\subset\mathcal{X}$ constituted by one-ended cacti is measurable. Using notations from Subsection \ref{SubsecCP}, for any $M\in\mathbb{N}$, let $X_M:\mathcal{C}\longrightarrow\mathbb{R}_+$ be the function mapping a one-ended cactus $(\Gamma,v)$ to $1/(|\mathcal{C}_{i_M}|\cdot|\mathcal{C}_{i_M+1}|)$ if the integer $M$ occurs as the mass of an avalanche on $(\Gamma,v)$, and to $0$ otherwise. The function $X_M$ is measurable as, for any fixed $M\in\mathbb{N}$ and $a,b\in\mathbb{R}$, the event $\{(\Gamma,v)| a\leq1/(|\mathcal{C}_{i_M}|\cdot|\mathcal{C}_{i_M+1}|)< b\}$ is a cylinder event. For any function $g:\mathbb{N}\longrightarrow\mathbb{R}_+$, consider the event

\begin{displaymath}
E_g:=\bigcup_{M_0\geq 1}\bigcap_{M\geq M_0}\{(\Gamma,v)| \textrm{$X_M(\Gamma,v)=g(M)$ or $X_M(\Gamma,v)=0$}\}.
\end{displaymath}

\noindent Our aim is to show that $E_g$ is of $\rho$-measure $0$ or $1$, and this will be done by using an ergodicity argument.\\
\indent It follows from Proposition \ref{PropRegGraphsSchreier}, that the sequence $\{\Gamma_n\}_{n\geq 1}$ can be realized as Schreier graphs (with respect to an appropriate set of generators) of an action of a group $G$ of automorphisms of a rooted tree $T$. Since the graphs we consider are connected, the action of $G$ on $T$ is spherically transitive, and hence the action of $G$ on the boundary $\partial T$ of $T$ is ergodic with respect to the uniform measure $\lambda$ (see for instance Proposition 6.5. in \cite{GNS}.)\\
\indent By Lemma \ref{LemPhiMeasurable}, the application $\phi: \partial T\longrightarrow \mathcal{X}$, $\phi(\xi):=(\Gamma_\xi,\xi)$ is measurable, and the random weak limit $\rho$ of the sequence $\{\Gamma_n\}_{n\geq 1}$ is the image under $\phi$ of the uniform measure $\lambda$. Recall that a measure $\mu$ on a standard Borel space $(X,\mathcal{B})$ with an equivalence relation $\mathcal{R}$ is $\mathcal{R}$-ergodic, if every Borel $\mathcal{R}$-invariant subset of $X$ is of $\mu$-measure $0$ or $1$. Consider the equivalence relation $\mathcal{R}$ on $\mathcal{X}$, the \emph{change of root}, that identifies different rootings of a graph. 
One easily checks that the random weak limit $\rho=\phi(\lambda)$ is $\mathcal{R}$-ergodic.\\
\indent We verify that the event $E_g$ is $\mathcal{R}$-invariant: let $(\Gamma,v)$ and $(\Gamma',v')$ be one-ended cacti and suppose that $(\Gamma,v)$ and $(\Gamma',v')$ are $\mathcal{R}$-equivalent. Let $\mathcal{CP}_v=\mathcal{C}_1\mathcal{C}_2\dots$ be the unique block-path of infinite length in $(\Gamma,v)$ starting at $v$ (respectively $\mathcal{CP}_{v'}=\mathcal{C}_1'\mathcal{C}_2'\dots$ in $(\Gamma',v')$ starting at $v'$) and recall that $p_i$ (respectively $p_i'$) denotes the cut vertex between $\mathcal{C}_i$ and $\mathcal{C}_{i+1}$ (respectively between $\mathcal{C}_i'$ and $\mathcal{C}_{i+1}'$) (see Subsection \ref{SubsecCP}). Since $(\Gamma,v)$ and $(\Gamma',v')$ are one-ended and isomorphic as unrooted graphs then, up to some initial segment, $\mathcal{CP}_v$ and $\mathcal{CP}_{v'}$ are isomorphic (i.e. there exist $k,l\geq 1$ such that $\mathcal{C}_k\mathcal{C}_{k+1}\dots$ and $\mathcal{C}_l'\mathcal{C}_{l+1}'\dots$ are isomorphic.) Moreover, the subgraphs $D(p_{k+i})\subset(\Gamma,v)$ and $D(p_{l+i}')\subset(\Gamma',v')$ are isomorphic for any $i\geq 0$. It follows that $X_M(\Gamma,v)=X_M(\Gamma',v')$ for any $M$ sufficiently large.\\
\indent Thus, by ergodicity of $\rho$, the event $E_g$ has probability $0$ or $1$. It follows then from Theorem \ref{ThmCP}, that the asymptotical behaviour (in $M$) of the distribution $\lim_{n\to\infty}\mathbb{P}_{\mu_n}(Mav_{H_n}(\cdot,v)=M)$ is $\rho$-almost everywhere the same.
\end{proof}

\section{The Basilica Group and its Schreier Graphs} \label{Basilica}

\setcounter{defi}{0}

The Basilica group $\mathcal{B}$ is an automorphism group of the rooted binary tree which is generated by two automorphisms $a$ and $b$ having the following self-similar structure:

\begin{equation}
a=e(b,id) \ \ \ \ \ \ \ b=(0 \ 1)(a,id),
\end{equation}

\noindent where $id$ denotes the trivial automorphism of the tree, whereas $e$ is the identity permutation in $S_2$. In other words, $a$ fixes the first level, then acts as $b$ on the subtree rooted at $0$ and as the identity on the subtree rooted at $1$, whereas $b$  permutes the vertices of the first level, then acts as $a$ on the subtree rooted at $0$ and as the identity on the subtree rooted at $1$. It can be easily checked that the action of $\mathcal{B}$ on the binary tree is spherically transitive.\\
\indent The group $\mathcal{B}$ was introduced by Grigorchuk and \.{Z}uk \cite{grizuk} as the group generated by the three-state automaton represented in Figure \ref{AUTOMATBASILICA}. It can also be described as the iterated monodromy group $IMG(z^2-1)$ of the complex polynomial $z^2-1$ \cite{NekBook} (see Figure \ref{JULIABASILICA}).


\begin{figure}[H]
\begin{center}
\begin{picture}(300,130)
\thicklines \setvertexdiam{20} \setprofcurve{15}\setloopdiam{20}
\letstate A=(75,110)
\letstate C=(200,70)
\letstate B=(75,20)
\drawstate(A){$a$} \drawstate(B){$b$} \drawstate(C){$id$}
\drawcurvededge(A,C){$1|1$}
\setprofcurve{-20}\drawcurvededge[r](B,C){$1|0$}
\drawloop[r](C){$0|0,1|1$} \drawcurvededge[r](A,B){$0|0$}
\drawcurvededge[r](B,A){$0|1$}
\end{picture}
\end{center}
\caption{The automaton generating the Basilica group.}
\label{AUTOMATBASILICA}
\end{figure}
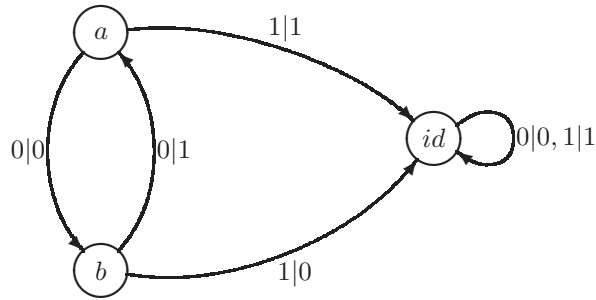


\begin{figure}[H] 
\begin{center}
\includegraphics[scale=0.3]{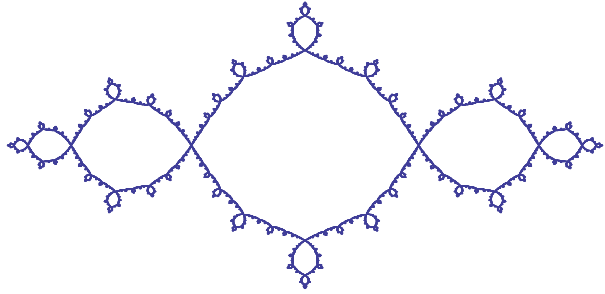}
\end{center}
\caption{The Julia set $\mathcal{J}(z^2-1)$.}
\label{JULIABASILICA}
\end{figure}

For each $n\geq 1$, we denote by $\Gamma_n\equiv \Gamma(\mathcal{B},\{a,b\},\{0,1\}^n)$ the Schreier graph of the action of the Basilica group $\mathcal{B}$ on the $n$-th level of the binary tree. These graphs, appropriately rescaled, form an approximating sequence of the Basilica Julia set $\mathcal J(z^2-1)$ (this is used for example by Rogers and Teplyaev in \cite{RogTep09} for defining laplacians on the Julia set). The graphs $\{\Gamma_n\}_{n\geq 1}$ can be constructed recursively as follows:

\begin{prop}\label{rulesBas}\emph{\cite{DanDonMat09}}
The Schreier graph $\Gamma_{n+1}$ is obtained from $\Gamma_n$ by applying to all subgraphs of $\Gamma_n$ given by single edges the rules represented in Figure \ref{SUBSTRULES}.
\end{prop}


\begin{figure}[H]
\begin{center}
\begin{picture}(400,125)
\put(112,117){SR1}\put(192,117){SR2}\put(292,117){SR3}
\letvertex A=(120,100)\letvertex B=(100,20)\letvertex C=(140,20)
\letvertex D=(180,100)\letvertex E=(220,100)\letvertex F=(180,20)\letvertex G=(220,20)
\letvertex H=(280,100)\letvertex I=(320,100)\letvertex L=(260,10)\letvertex M=(300,20)\letvertex N=(340,10)
\put(117,60){$\Downarrow$}\put(197,60){$\Downarrow$}\put(297,60){$\Downarrow$}
\put(117,92){$1u$}\put(97,11){$11u$}\put(137,11){$01u$}
\put(177,92){$u$}\put(217,92){$v$}\put(177,11){$0u$}\put(217,11){$0v$}
\put(277,92){$0u$}\put(317,92){$0v$}
\put(257,1){$00u$}\put(296,10){$10v$}\put(337,1){$00v$} \put(327,97)

\drawvertex(A){$\bullet$}\drawvertex(B){$\bullet$}
\drawvertex(C){$\bullet$}\drawvertex(D){$\bullet$}
\drawvertex(E){$\bullet$}\drawvertex(F){$\bullet$}
\drawvertex(G){$\bullet$}\drawvertex(H){$\bullet$}
\drawvertex(I){$\bullet$}\drawvertex(L){$\bullet$}
\drawvertex(M){$\bullet$}\drawvertex(N){$\bullet$}
\drawundirectedloop(A){$a$}\drawundirectedloop[l](B){$a$}
\drawundirectedcurvededge(B,C){$b$}\drawundirectedcurvededge(C,B){$b$}
\drawundirectededge(D,E){$b$} \drawundirectededge(F,G){$a$}
\drawundirectededge(H,I){$a$}
\drawundirectedcurvededge(L,M){$b$}\drawundirectedloop(M){$a$}
\drawundirectedcurvededge(M,N){$b$}
\end{picture}
\end{center}
\end{figure}
with
\begin{figure}[H]
\begin{center}
\begin{picture}(200,40)
\letvertex A=(70,25)\letvertex B=(130,25)
\put(67,16){$0$}\put(127,16){$1$}\put(30,21){$\Gamma_1$}
\drawvertex(A){$\bullet$}\drawvertex(B){$\bullet$}
\drawundirectedloop[l](A){$a$}\drawundirectedloop[r](B){$a$}
\drawundirectedcurvededge(A,B){$b$}\drawundirectedcurvededge(B,A){$b$}
\end{picture}
\end{center}
\caption{Rewriting rules for construction of the Basilica Schreier graphs and the Schreier graph $\Gamma_1$.}
\label{SUBSTRULES}
\end{figure}


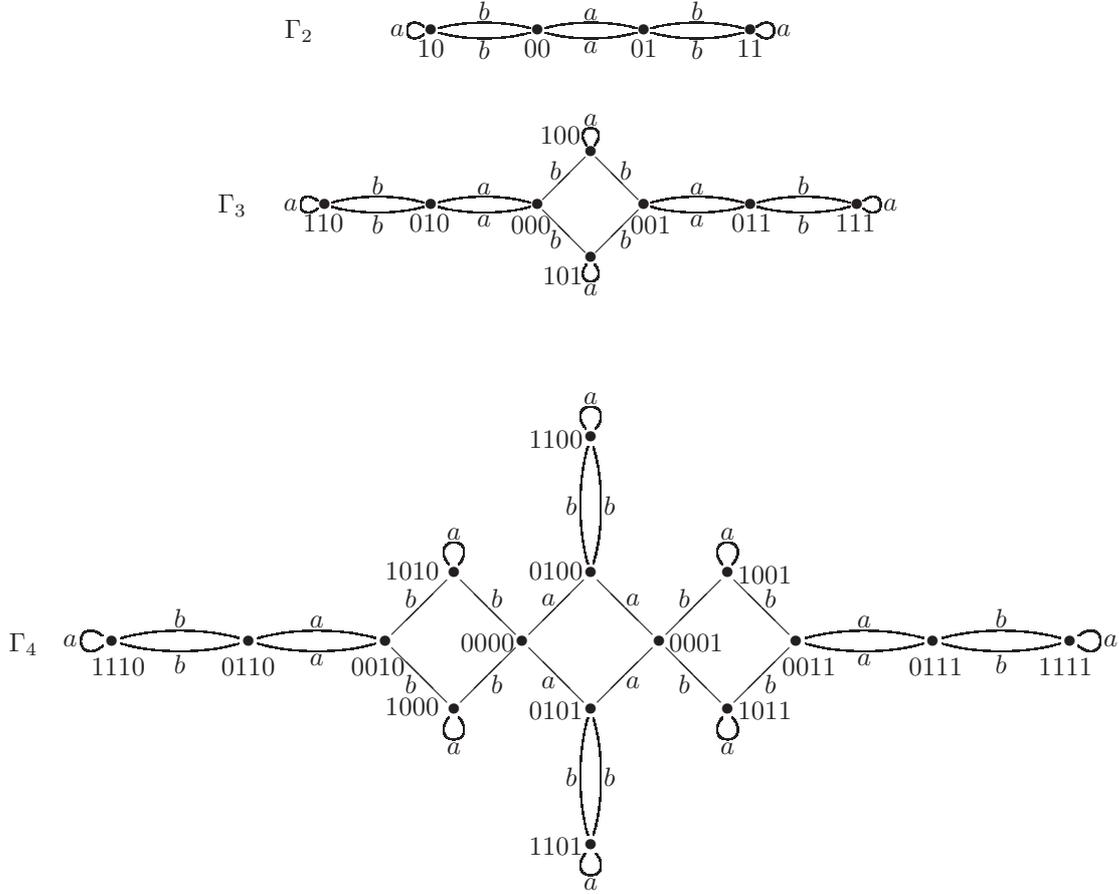
\begin{figure}[H]
\begin{center}
\begin{picture}(300,30)
\letvertex C=(90,15)\letvertex D=(130,15)
\letvertex E=(170,15)\letvertex F=(210,15)
\put(85,5){10}\put(125,5){00}\put(165,5){01}\put(205,5){11}
\drawvertex(C){$\bullet$}\drawvertex(D){$\bullet$}
\drawvertex(E){$\bullet$}\drawvertex(F){$\bullet$}
\drawundirectedcurvededge(C,D){$b$}
\drawundirectedcurvededge(D,C){$b$}
\drawundirectedcurvededge(D,E){$a$}
\drawundirectedcurvededge(E,D){$a$}
\drawundirectedcurvededge(E,F){$b$}
\drawundirectedcurvededge(F,E){$b$}
\put(35,12){$\Gamma_2$}
\drawundirectedloop[l](C){$a$} \drawundirectedloop[r](F){$a$}
\end{picture}
\begin{picture}(300,60)
\letvertex A=(50,10)\letvertex B=(90,10)\letvertex C=(130,10)\letvertex D=(150,30)
\letvertex E=(150,-10)\letvertex F=(170,10)\letvertex G=(210,10)\letvertex H=(250,10)
\put(42,0){110} \put(82,0){010}\put(120,-1){000}
\put(131,33){100}\put(132,-20){101}\put(165,-1){001}
\put(203,0){011}\put(242,0){111}
\drawvertex(A){$\bullet$}\drawvertex(B){$\bullet$}
\drawvertex(C){$\bullet$}\drawvertex(D){$\bullet$}
\drawvertex(E){$\bullet$}\drawvertex(F){$\bullet$}
\drawvertex(G){$\bullet$}\drawvertex(H){$\bullet$}
\drawundirectededge(C,D){$b$} \drawundirectededge(E,C){$b$}
\drawundirectededge(F,E){$b$} \drawundirectededge(D,F){$b$}
\drawundirectedcurvededge(A,B){$b$}
\drawundirectedcurvededge(B,A){$b$}
\drawundirectedcurvededge(B,C){$a$}
\drawundirectedcurvededge(C,B){$a$}
\drawundirectedcurvededge(F,G){$a$}
\drawundirectedcurvededge(G,F){$a$}
\drawundirectedcurvededge(G,H){$b$}
\drawundirectedcurvededge(H,G){$b$} \put(10,7){$\Gamma_3$}
\drawundirectedloop[l](A){$a$}\drawundirectedloop(D){$a$}\drawundirectedloop[b](E){$a$}
\drawundirectedloop[r](H){$a$}
\end{picture}
\unitlength=0,45mm
\begin{picture}(300,200)
\put(-20,77){$\Gamma_4$}
\letvertex A=(10,80)
\letvertex B=(50,80)\letvertex C=(90,80)\letvertex D=(110,100)\letvertex E=(110,60)
\letvertex F=(130,80)\letvertex G=(150,100)\letvertex H=(150,60)\letvertex I=(170,80)
\letvertex J=(150,140)\letvertex K=(150,20)
\letvertex L=(190,100)
\letvertex M=(190,60)\letvertex N=(210,80)
\letvertex O=(250,80)\letvertex P=(290,80)
\put(4,70){1110} \put(42,70){0110}\put(80,70){0010}
\put(90,98){1010}\put(90,58){1000}\put(112,77){0000}
\put(132,98){0100}\put(132,57){0101} \put(173,77){0001}
\put(132,137){1100}\put(132,17){1101}
\put(193,97){1001}\put(193,57){1011}\put(206,70){0011}
\put(243,70){0111}\put(281,70){1111}
\drawvertex(A){$\bullet$}\drawvertex(B){$\bullet$}
\drawvertex(C){$\bullet$}\drawvertex(D){$\bullet$}
\drawvertex(E){$\bullet$}\drawvertex(F){$\bullet$}
\drawvertex(G){$\bullet$}\drawvertex(H){$\bullet$}
\drawvertex(I){$\bullet$}\drawvertex(L){$\bullet$}
\drawvertex(M){$\bullet$}\drawvertex(N){$\bullet$}
\drawvertex(O){$\bullet$}\drawvertex(P){$\bullet$}
\drawvertex(J){$\bullet$}\drawvertex(K){$\bullet$}
\drawundirectedcurvededge(A,B){$b$}\drawundirectedcurvededge(B,A){$b$}
\drawundirectedcurvededge(B,C){$a$}\drawundirectedcurvededge(C,B){$a$}
\drawundirectededge(C,D){$b$} \drawundirectededge(D,F){$b$}
\drawundirectededge(F,E){$b$} \drawundirectededge(E,C){$b$}
\drawundirectededge(F,G){$a$} \drawundirectededge(G,I){$a$}
\drawundirectededge(I,H){$a$} \drawundirectededge(H,F){$a$}
\drawundirectededge(I,L){$b$} \drawundirectededge(L,N){$b$}
\drawundirectededge(N,M){$b$} \drawundirectededge(M,I){$b$}
\drawundirectedcurvededge(G,J){$b$}\drawundirectedcurvededge(J,G){$b$}
\drawundirectedcurvededge(H,K){$b$}\drawundirectedcurvededge(K,H){$b$}
\drawundirectedcurvededge(N,O){$a$}\drawundirectedcurvededge(O,N){$a$}
\drawundirectedcurvededge(O,P){$b$}\drawundirectedcurvededge(P,O){$b$}
\drawundirectedloop[l](A){$a$}\drawundirectedloop(D){$a$}\drawundirectedloop[b](E){$a$}
\drawundirectedloop(J){$a$}\drawundirectedloop[b](K){$a$}\drawundirectedloop(L){$a$}\drawundirectedloop[b](M){$a$}
\drawundirectedloop[r](P){$a$}
\end{picture}
\end{center}
\caption{Basilica Schreier graphs $\Gamma_n$, $2\leq n\leq 4$.}
\label{FINITESCHREIER}
\end{figure}

It follows that, for each $n\geq 1$, $\Gamma_n$ is a $4$-regular cactus such that removing any cut vertex disconnects $\Gamma_n$ into exactly two components. Let us call the unique cycle of $\Gamma_n$ containing vertices $0^n$ and $0^{n-1}1$ the \textbf{central cycle} of $\Gamma_n$. Given any vertex $v\in V(\Gamma_n)$, there is a unique block-path (see Subsection \ref{CT}) $\mathcal{CP}_{v}=\mathcal{C}_1\dots \mathcal{C}_{r}$ joining $v$ to the central cycle of $\Gamma_n$.

\begin{defi}\label{DefDeco}{\bf{\lq\lq Decoration of a vertex\rq\rq .}}
1) Let $v\in V(\Gamma_n)\backslash\{0^n\}$ be a cut vertex. Denote by $U_1$ and $U_2$ the two connected components obtained by removing $v$, so that moreover $0^n\in U_1$. The decoration $\mathcal{D}(v)$ of $v$ is the subgraph induced by the vertex set $V(U_2)\cup \{v\}$.\\
2) Let $v$ be a vertex with a loop. Then $\mathcal{D}(v)$ is the subgraph induced by $\{v\}$.\\
3) If $v=0^n$, then $\mathcal{D}(0^n)$ is the subgraph induced by $V(U_i)\cup \{0^n\}$ where $0^{n-1}1\notin U_i$.\\
\indent A decoration of a given vertex $v\in V(\Gamma_n)$ is called a k-decoration (or a decoration of height k) if it is isomorphic to the decoration of the vertex $0^k$ for some $1\leq k\leq n$.
\end{defi}

The following proposition collects some of the properties of the graph $\Gamma_n$.

\begin{prop}\label{PropertiesGamma_n}
For any $n\geq 1$, consider the Schreier graph $\Gamma_n$. Then, the following hold:

\begin{enumerate}
\item \emph{\cite{DanDonMat09}} Every decoration in $\Gamma_n$ is a $k$-decoration for some $1\leq k\leq n$.
\item
\begin{displaymath}
|\mathcal{D}(0^n)|=\left\{\begin{array}{cc}
\frac{1}{3}(2^n+2) & \textrm{if $n$ is even,}\\
\frac{1}{3}(2^n+1) & \textrm{if $n$ is odd.}
\end{array}\right.
\end{displaymath}

\item The lengths of the cycles constituting $\Gamma_n$ are all powers of two; the number $\nu_k$ of cycles of length $2^k$ ($k\geq 1$) is

    \begin{displaymath}
    \nu_k=\left\{\begin{array}{ccc}
    3\cdot 2^{n-2k-1} & \textrm{for} & 1\leq k\leq \frac{n}{2}-1,\\
    3 & \textrm{for} & k=\frac{n}{2},
    \end{array}\right.
    \end{displaymath}

\noindent if $n$ is even, and

    \begin{displaymath}
    \nu_k=\left\{\begin{array}{ccc}
    3\cdot 2^{n-2k-1} & \textrm{for} & 1\leq k\leq \lfloor\frac{n}{2}\rfloor-1,\\
    4 & \textrm{for} &  k=\lfloor\frac{n}{2}\rfloor,\\
    1 & \textrm{for} & k=\lceil\frac{n}{2}\rceil,
    \end{array}\right.
    \end{displaymath}

\noindent if $n$ is odd.
\end{enumerate}
\end{prop}

The proofs of statements \emph{2.} and \emph{3.} are straightforward when using the substitutional rules described in Proposition \ref{rulesBas} and induction on $n$.\\
\indent The structure of the critical group $K(\Gamma_n)$ follows now immediately (see Subsection \ref{CT}).

\begin{prop} \label{Isomorphic}
If $n$ is even, then $K(\Gamma_n)$ is isomorphic to

\begin{displaymath}
\prod_{k=1}^{\frac{n}{2}-1}\left(\mathbb{Z}/2^k\mathbb{Z}\right)^{3\cdot 2^{n-2k-1}}
\times \left(\mathbb{Z}/2^{\frac{n}{2}}\mathbb{Z}\right)^3,
\end{displaymath}

\noindent and if $n$ is odd, then $K(\Gamma_n)$ is isomorphic to

\begin{displaymath}
\prod_{k=1}^{\lfloor\frac{n}{2}\rfloor-1}\left(\mathbb{Z}/2^k\mathbb{Z}\right)^{3\cdot 2^{n-2k-1}}
\times \left(\mathbb{Z}/2^{\lfloor\frac{n}{2}\rfloor}\mathbb{Z}\right)^4
\times \left(\mathbb{Z}/2^{\lceil\frac{n}{2}\rceil}\mathbb{Z}\right).
\end{displaymath}
\end{prop}

Note that since the lengths of the cycles in $\Gamma_n$ are all powers of two, the latter decomposition corresponds to the decomposition of $K(\Gamma_n)$ into invariant factors.\\
\indent Given a ray $\xi\in \{0,1\}^\omega$, the sequence $\{(\Gamma_n,\xi_n)\}_{n\geq 1}$  of finite Schreier graphs, rooted at the $n$-th prefix $\xi_n$ of $\xi$, converges in $(\mathcal{X},Dist)$ to the infinite orbital Schreier graph $(\Gamma_\xi,\xi)\equiv (\Gamma(\mathcal{B},\{a,b\},\mathcal{B}\cdot\xi),\xi)$.  The following results classify all rays $\xi\in\{0,1\}^\omega$ with respect to the number of ends of the corresponding limit graph (can be equal to $4$, $2$ or, almost surely, to $1$), as well as gives information about  different types of isomorphisms of infinite orbital Schreier graphs.

\begin{thm}\emph{~\cite{DanDonMat09}} \label{ClassifBas}
Set $E_i = \{\xi \in \{0,1\}^{\omega} \ | \ \textrm{the infinite Schreier graph}\ \Gamma_{\xi} \
\mbox{has }i \mbox{ ends}\}$. Then,
\begin{enumerate}
\item  $E_4 = \{w0^{\omega}, w(01)^{\omega}\ |\ \ w\in
\{0,1\}^{\ast}\}$;
\item $E_1 = \{\alpha_1\beta_1\alpha_2\beta_2\ldots, \
\alpha_i,\beta_j\in \{0,1\}\ | \  \textrm{$\{\alpha_i\}_{i\geq 1}$
and $\{\beta_j\}_{j\geq 1}$ both contain infinitely many }1's\}$;
\item $E_2 = \{0,1\}^{\omega} \setminus \left(E_1\sqcup
E_4\right)$.
\end{enumerate}
\end{thm}

\begin{cor}\emph{~\cite{DanDonMat09}} \label{CorClassifBas}
\begin{enumerate}
\item There exists only one class of isomorphism of 4-ended (unrooted) infinite Schreier graphs. It contains a single orbit.
\item There exist uncountably many classes of isomorphism of 2-ended (unrooted) infinite Schreier graphs. Each of these classes contains exactly two orbits.
\item There exist uncountably many classes of isomorphism of 1-ended (unrooted) infinite Schreier graphs. The isomorphism class of $\Gamma_{1^\omega}$ is a single orbit, and every other class contains uncountably many orbits.
\end{enumerate}
\end{cor}

Recall that $\phi:\{0,1\}^\omega\longrightarrow \mathcal{X}$, $\phi(\xi):=(\Gamma_\xi,\xi)$, is the application mapping an infinite binary sequence $\xi$ to the (rooted isomorphism class of the) orbital Schreier graph $\Gamma_\xi$ rooted at $\xi$, and that the random weak limit of the sequence of finite Schreier graphs $\{\Gamma_n\}_{n\geq 1}$ is the image under $\phi$ of $\lambda$, the uniform measure on $\{0,1\}^\omega$ (see Lemma \ref{LemPhiMeasurable}).

\begin{prop}\emph{~\cite{DanDonMat09}}\label{unifmeasure}
The random weak limit of the sequence of finite Schreier graphs $\{\Gamma_n\}_{n\geq 1}$ is concentrated on $1$-ended graphs.
\end{prop}

We first describe the limit graphs with four and two ends (proofs can be found in \cite{DanDonMat09}).
Given $\xi\in E_4$, any orbital Schreier graph $\Gamma_\xi$ is isomorphic to the four-ended graph $\Gamma_{(4)}$ constructed as follows (see Figure \ref{FOURENDEDGRAPH}): take two copies $\mathcal{R}_1$ and $\mathcal{R}_2$ of the double ray whose vertices are naturally identified with the integers. Let these two double rays intersect at vertex $0$. For every $k\geq 0$, define the subset of $\mathbb{Z}$

\begin{displaymath}
A_k:=\left\{n\in\mathbb{Z}|n\equiv 2^k\mod 2^{k+1}\right\}.
\end{displaymath}

\noindent Attach to each vertex of $A_k$ in $\mathcal{R}_1$ (respectively in $\mathcal{R}_2$) a $(2k+1)$-decoration (respectively a $(2k+2)$-decoration) by its unique vertex of degree $2$.\\
\indent For any $\xi\in E_2$, $\xi$ can be written as $\xi=\alpha_1\beta_1\alpha_2\beta_2\dots$ where exactly one of the sequences $\{\alpha_i\}_{i\geq 1}$ or $\{\beta_i\}_{i\geq 1}$ has finitely many $1$'s. If $\{\alpha_i\}_{i\geq 1}$ has finitely many $1$'s, the graph $\Gamma_\xi$ is isomorphic to the following graph $\Gamma(\xi)$: consider the subsets of $\mathbb{Z}$

\begin{displaymath}
A'_0:= 2\mathbb{Z} \ \ \mbox{ and } \ \ \ A'_k:=\{n\in \mathbb{Z} \
| \ n\equiv 2^k-1-\sum_{i=1}^k2^i\beta_{i+1} \mod 2^{k+1}\} \
\mbox{ for each }k\geq 1.
\end{displaymath}

\noindent Construct $\Gamma(\xi)$ as a double ray with integer vertices with, for each $k\geq 0$, a $(2k+2)$-decoration attached by its unique vertex of degree 2 to every vertex corresponding to an integer in $A'_k$.\\
\indent In the case where $\{\beta_i\}_{i\geq 1}$ has finitely many $1$'s, the graph $\Gamma(\xi)$ is defined similarly, replacing $\beta$ by $\alpha$ in the definition of $A_k'$ and by attaching $(2k+1)$-decorations instead of $(2k+2)$-decorations (see Figure \ref{TWOENDEDGRAPH}).

\begin{cor}
The two-ended orbital Schreier graphs $\Gamma_\xi$, $\xi\in E_2$, form an uncountable family of non-isomorphic graphs which are not quasi-isometric to the one-dimensional lattice.
\end{cor}


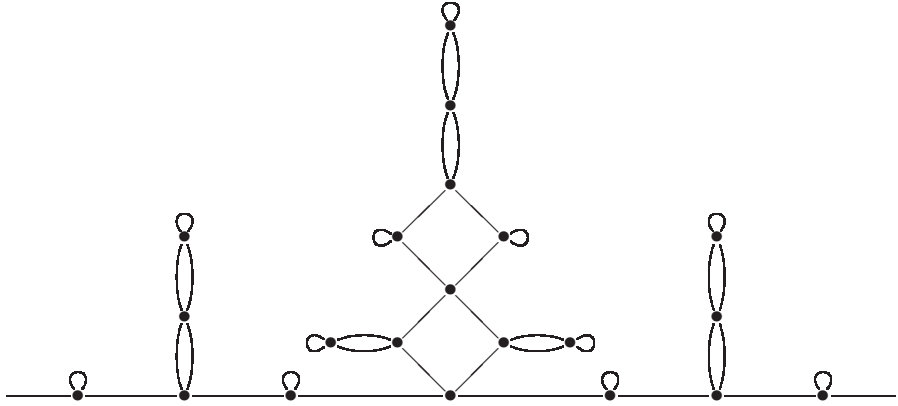
\begin{figure}[H]
\begin{center}
\begin{picture}(350,195)
\letvertex A=(5,40)\letvertex B=(35,40)\letvertex C=(75,40)\letvertex D=(75,70)
\letvertex E=(75,100)\letvertex F=(115,40)\letvertex G=(175,40)\letvertex H=(155,60)
\letvertex I=(130,60)\letvertex L=(175,80)\letvertex M=(195,60)\letvertex N=(220,60)
\letvertex O=(155,100)\letvertex R=(175,150)\letvertex S=(175,180)\letvertex T=(235,40)
\letvertex U=(275,40)\letvertex V=(275,70)\letvertex P=(195,100)\letvertex Q=(175,120)
\letvertex X=(315,40)\letvertex W=(275,100)\letvertex Y=(345,40)
\drawvertex(B){$\bullet$}\drawvertex(C){$\bullet$}
\drawvertex(D){$\bullet$}
\drawvertex(E){$\bullet$}\drawvertex(F){$\bullet$}
\drawvertex(G){$\bullet$}\drawvertex(H){$\bullet$}
\drawvertex(I){$\bullet$}\drawvertex(L){$\bullet$}
\drawvertex(M){$\bullet$}\drawvertex(N){$\bullet$}
\drawvertex(O){$\bullet$}\drawvertex(P){$\bullet$}
\drawvertex(Q){$\bullet$}\drawvertex(R){$\bullet$}
\drawvertex(S){$\bullet$}\drawvertex(T){$\bullet$}
\drawvertex(U){$\bullet$}\drawvertex(V){$\bullet$}
\drawvertex(W){$\bullet$}\drawvertex(X){$\bullet$}
\drawundirectededge(A,B){}\drawundirectededge(B,C){}
\drawundirectedcurvededge(C,D){}\drawundirectedcurvededge(D,C){}
\drawundirectedcurvededge(E,D){}\drawundirectedcurvededge(D,E){}
\drawundirectededge(C,F){} \drawundirectededge(F,G){}
\drawundirectededge(G,H){} \drawundirectededge(H,L){}
\drawundirectededge(L,M){} \drawundirectededge(M,G){}
\drawundirectededge(L,O){} \drawundirectededge(O,Q){}
\drawundirectededge(Q,P){} \drawundirectededge(P,L){}
\drawundirectededge(G,T){} \drawundirectededge(T,U){}
\drawundirectededge(U,X){} \drawundirectededge(X,Y){}
\drawundirectedcurvededge(I,H){}\drawundirectedcurvededge(H,I){}
\drawundirectedcurvededge(M,N){}\drawundirectedcurvededge(N,M){}
\drawundirectedcurvededge(Q,R){}\drawundirectedcurvededge(R,Q){}
\drawundirectedcurvededge(R,S){}\drawundirectedcurvededge(S,R){}
\drawundirectedcurvededge(U,V){}\drawundirectedcurvededge(V,U){}
\drawundirectedcurvededge(V,W){}\drawundirectedcurvededge(W,V){}
\drawundirectedloop(B){}\drawundirectedloop(E){}\drawundirectedloop(S){}
\drawundirectedloop[l](O){}\drawundirectedloop[r](P){}\drawundirectedloop[l](I){}\drawundirectedloop[r](N){}
\drawundirectedloop(T){}\drawundirectedloop(W){}\drawundirectedloop(X){}\drawundirectedloop(F){}
\end{picture}
\end{center}
\caption{A finite part of $\Gamma_{\xi}$, $\xi\in E_2$.}
\label{TWOENDEDGRAPH}
\end{figure}

\indent We now proceed to the case of one-ended limit graphs. Let $\xi\in\{0,1\}^\omega$. Recall that for a finite Schreier graph $\Gamma_n$ with a root $\xi_n$, $\mathcal{CP}_{\xi_n}$ denotes the unique block-path joining $\xi_n$ to the central cycle in $\Gamma_n$. In the case of one-ended limit graphs, we have:

\begin{lem}\emph{\cite{DanDonMat09}} \label{LemOnCPinfBasilica}
Let $\xi\in E_1$. Then, the limit

\begin{equation}
(\mathcal{CP}_\xi,\xi):=\lim_{n\to\infty}(\mathcal{CP}_{\xi_n},\xi_n)
\end{equation}

\noindent is well-defined and the graph $\mathcal{CP}_\xi$ is isomorphic to the unique block-path of infinite length in $(\Gamma_\xi,\xi)$ starting at $\xi$.
\end{lem}

\begin{rem}\rm
It follows from Theorem \ref{ThmCP} that for understanding the asymptotic of avalanches it is enough to keep track of the sizes of blocks $\mathcal{C}_1,\mathcal{C}_2,\dots$ constituting the block-path $\mathcal{CP}_\xi$.
\end{rem}

We will need the following technical lemmas:

\begin{lem} \emph{\cite{DanDonMat09}} \label{PropAlg}
An element $\xi\in\{0,1\}^\omega$, $\xi\neq w1^\omega$ for any $w\in \{0,1\}^\ast$, belongs to $E_1$ if and only if there exists a unique triple $(l,\{m_k\}_{k\geq 0},\{t_k\}_{k\geq 0})$ where $l\geq 1$ and $m_0\geq 0$ are integers and $m_0$ is even; $t_0=0$; and $\{m_k\}_{k\geq 1}$, $\{t_k\}_{k\geq 1}$ are sequences of strictly positive integers and the $m_k$'s are even, such that $\xi$ can be written as
\begin{equation} \label{CanFormOfXi}
\xi=0^{l-1}1(0x^0_{1}0x^0_{2}\dots 0x^0_{\frac{m_0}{2}})1^{t_1}(0x^1_{1}0x^1_{2}\dots 0x^1_{\frac{m_1}{2}})1^{t_2}\dots.
\end{equation}
with $x_i^j\in\{0,1\}$ for all $i,j$.\\
\indent If $\xi=w1^\omega$ for some $w\in \{0,1\}^\ast$, then there exists a unique triple $(l,\{m_k\}_{k=0}^{k_0},\{t_k\}_{k=0}^{k_0})$ where $l\geq 1$ and $m_0\geq 0$ are integers and $m_0$ is even;
$t_0=0$; and $\{m_k\}_{k=1}^{k_0}$, $\{t_k\}_{k=1}^{k_0}$ are finite sequences of strictly positive integers and the $m_k$'s are even, such that $\xi$ can be written as
$$\xi=0^{l-1}1(0x^0_{1}0x^0_{2}\dots 0x^0_{\frac{m_0}{2}})1^{t_1}(0x^1_{1}0x^1_{2}\dots 0x^1_{\frac{m_1}{2}})1^{t_2}\dots
1^{t_{k_0}}(0x^{k_0}_{1}0x^{k_0}_{2}\dots 0x^{k_0}_{\frac{m_{k_0}}{2}})1^\omega.$$
\end{lem}

\begin{lem}\emph{\cite{DanDonMat09}}\label{remonai}
Let $\xi\in E_1$ and define a sequence of integers $a_i=a_i^\xi$, $i\geq 1$, as follows: if $\xi=1^\omega$, then $a_i:=i$ for all $i\geq 1$. If $\xi\neq 1^\omega$, then Lemma \ref{PropAlg} provides a triple $(l,\{m_k\},\{t_k\})$ associated with $\xi$. For all $j\geq 1$, $0\leq s< t_j$, let $a_{T_{j-1}+s+1}:=l+M_{j-1}+T_{j-1}+s$, where $M_j:=\sum_{k=0}^jm_k$ and $T_j:=\sum_{k=0}^jt_k$. Then,

\begin{itemize}
\item The sequence $\{a_i\}_{i\geq 1}$ is increasing. More precisely,
\begin{equation} \label{eqDiffAi}
a_{i+1}-a_i=\left\{\begin{array}{cc}
m_j+1 & \textrm{if there exists $j>0$ such that $i=T_j$,}\\
1 & \textrm{otherwise.}
\end{array}\right.
\end{equation}
\item For all $i\geq 1$, the size of $\mathcal{C}_i$ in $\mathcal{CP}_\xi\subset \Gamma_\xi$ is equal to $2^{\lceil a_i/2\rceil}$.
\end{itemize}
\end{lem}

The description from Lemma \ref{PropAlg} allows to classify  the words $\xi\in E_1$ giving rise to isomorphic orbital Schreier graphs $\Gamma_\xi$ (see Theorem 5.4 in \cite{DanDonMat09}).

\begin{prop}
The orbital one-ended Schreier graphs $\Gamma_\xi$, $\xi\in E_1$, form an uncountable family of $4$-regular graphs of quadratic growth (for a proof of this fact, see \cite{Bond07}).
\end{prop}


\begin{figure}[H]
\begin{center}
\begin{picture}(400,290)
\letvertex I=(160,150)\letvertex N=(170,180)
\letvertex O=(200,190)\letvertex R=(230,180)\letvertex S=(240,150)\letvertex T=(230,120)
\letvertex U=(200,110)\letvertex V=(170,120)\letvertex P=(200,220)\letvertex Q=(200,250)
\letvertex Z=(200,80)\letvertex J=(200,50)\letvertex K=(260,170)\letvertex X=(280,150)
\letvertex W=(260,130)\letvertex g=(260,200)\letvertex h=(260,100)\letvertex c=(300,170)
\letvertex Y=(300,130)\letvertex d=(320,150)\letvertex e=(360,150)\letvertex f=(400,150)
\put(396,138){$1^{\omega}$}
\letvertex II=(60,200)\letvertex NN=(60,100)
\letvertex A=(120,260)\letvertex B=(120,240)
\letvertex C=(120,220)\letvertex D=(105,205)\letvertex E=(135,205)\letvertex F=(120,190)
\letvertex G=(90,180)\letvertex H=(77,193)\letvertex i=(80,150)\letvertex l=(90,120)
\letvertex m=(77,107)\letvertex n=(105,95)\letvertex o=(120,110)\letvertex p=(120,80)
\letvertex q=(120,60)\letvertex r=(120,40)\letvertex s=(135,95)\letvertex t=(150,120)
\letvertex u=(163,107)\letvertex v=(150,180)\letvertex z=(163,193)
\drawvertex(I){$\bullet$}\drawvertex(N){$\bullet$}\drawvertex(O){$\bullet$}\drawvertex(P){$\bullet$}
\drawvertex(J){$\bullet$}\drawvertex(K){$\bullet$}\drawvertex(Q){$\bullet$}\drawvertex(R){$\bullet$}
\drawvertex(S){$\bullet$}\drawvertex(T){$\bullet$}\drawvertex(U){$\bullet$}\drawvertex(V){$\bullet$}
\drawvertex(W){$\bullet$}\drawvertex(X){$\bullet$}\drawvertex(Y){$\bullet$}\drawvertex(Z){$\bullet$}
\drawvertex(g){$\bullet$}\drawvertex(h){$\bullet$}\drawvertex(c){$\bullet$}\drawvertex(f){$\bullet$}
\drawvertex(d){$\bullet$}\drawvertex(e){$\bullet$}\drawvertex(A){$\bullet$}\drawvertex(n){$\bullet$}
\drawvertex(B){$\bullet$}\drawvertex(o){$\bullet$}\drawvertex(C){$\bullet$}\drawvertex(p){$\bullet$}
\drawvertex(D){$\bullet$}\drawvertex(q){$\bullet$}\drawvertex(E){$\bullet$}\drawvertex(r){$\bullet$}
\drawvertex(F){$\bullet$}\drawvertex(s){$\bullet$}\drawvertex(G){$\bullet$}\drawvertex(t){$\bullet$}
\drawvertex(H){$\bullet$}\drawvertex(u){$\bullet$}\drawvertex(i){$\bullet$}\drawvertex(v){$\bullet$}
\drawvertex(l){$\bullet$}\drawvertex(z){$\bullet$}\drawvertex(m){$\bullet$}
\drawundirectededge(II,i){} \drawundirectededge(NN,i){}
\drawundirectededge(I,N){} \drawundirectededge(N,O){}
\drawundirectededge(O,R){} \drawundirectededge(R,S){}
\drawundirectededge(S,T){} \drawundirectededge(T,U){}
\drawundirectededge(U,V){} \drawundirectededge(V,I){}
\drawundirectedcurvededge(O,P){}\drawundirectedcurvededge(P,O){}
\drawundirectedcurvededge(Q,P){}\drawundirectedcurvededge(P,Q){}
\drawundirectedcurvededge(U,Z){}\drawundirectedcurvededge(Z,U){}
\drawundirectedcurvededge(Z,J){}\drawundirectedcurvededge(J,Z){}
\drawundirectededge(S,K){} \drawundirectededge(K,X){}
\drawundirectededge(X,W){} \drawundirectededge(W,S){}
\drawundirectededge(X,c){} \drawundirectededge(c,d){}
\drawundirectededge(d,Y){} \drawundirectededge(Y,X){}
\drawundirectedcurvededge(d,e){}\drawundirectedcurvededge(e,d){}
\drawundirectedcurvededge(e,f){}\drawundirectedcurvededge(f,e){}
\drawundirectedcurvededge(K,g){}\drawundirectedcurvededge(g,K){}
\drawundirectedcurvededge(W,h){}\drawundirectedcurvededge(h,W){}
\drawundirectedloop(Q){}\drawundirectedloop[b](J){}
\drawundirectedloop(g){}\drawundirectedloop[b](h){}\drawundirectedloop(c){}\drawundirectedloop[b](Y){}
\drawundirectedloop[r](f){}\drawundirectedloop[b](V){}\drawundirectedloop(N){}\drawundirectedloop[b](T){}
\drawundirectedloop(R){}
\drawundirectededge(C,D){}\drawundirectededge(D,F){}\drawundirectededge(F,E){}\drawundirectededge(E,C){}\drawundirectededge(F,G){}
\drawundirectededge(G,i){}\drawundirectededge(i,l){}\drawundirectededge(l,o){}\drawundirectededge(o,t){}
\drawundirectededge(t,I){}\drawundirectededge(I,v){}\drawundirectededge(v,F){}\drawundirectededge(o,n){}
\drawundirectededge(n,p){}\drawundirectededge(p,s){}\drawundirectededge(s,o){}
\drawundirectedcurvededge(A,B){}\drawundirectedcurvededge(B,A){}\drawundirectedcurvededge(B,C){}\drawundirectedcurvededge(C,B){}
\drawundirectedcurvededge(G,H){}\drawundirectedcurvededge(H,G){}\drawundirectedcurvededge(l,m){}\drawundirectedcurvededge(m,l){}
\drawundirectedcurvededge(t,u){}\drawundirectedcurvededge(u,t){}\drawundirectedcurvededge(v,z){}\drawundirectedcurvededge(z,v){}
\drawundirectedcurvededge(p,q){}\drawundirectedcurvededge(q,p){}\drawundirectedcurvededge(q,r){}\drawundirectedcurvededge(r,q){}
\drawundirectedloop(A){}\drawundirectedloop[l](D){}\drawundirectedloop(H){}
\drawundirectedloop[r](E){}\drawundirectedloop[l](n){}\drawundirectedloop[r](s){}\drawundirectedloop[b](r){}\drawundirectedloop[b](m){}\drawundirectedloop(z){}\drawundirectedloop[b](u){}
\end{picture}
\end{center}
\caption{A finite part of $\Gamma_{1^{\omega}}$.}
\label{ONEENDEDGRAPH}
\end{figure}
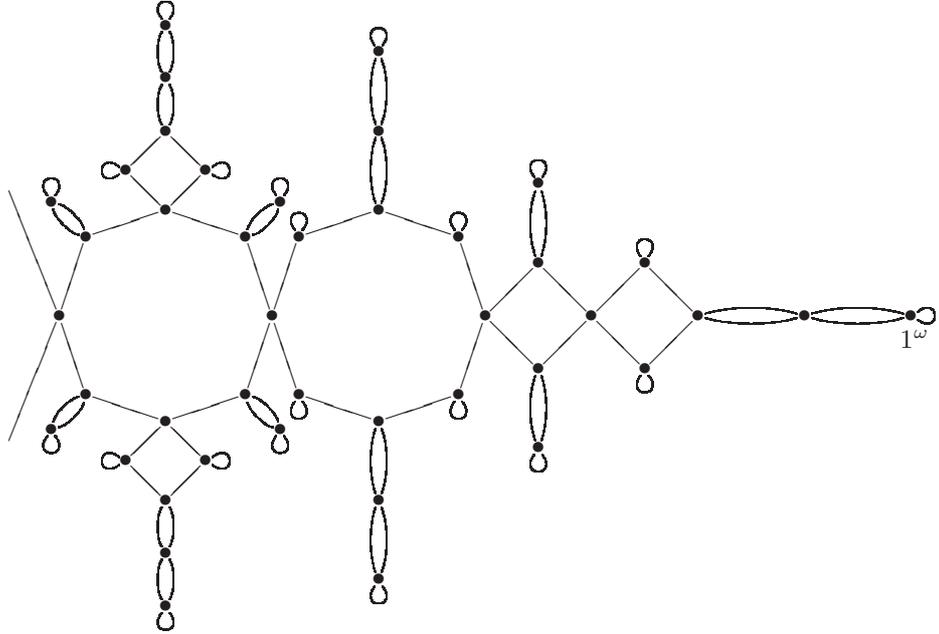

\begin{rem}\rm
It follows also from Theorem 5.4 in \cite{DanDonMat09}, that any two non-isomorphic Schreier graphs $\Gamma_\xi$ and $\Gamma_\eta$ for $\xi,\eta\in E_1$ are not quasi-isometric. Also, none of them is
quasi-isometric to $\mathbb{Z}^2$.\\
\indent Indeed, let $\Gamma_\xi\not\simeq\Gamma_\eta$ and suppose that the sequences $\{a_i^\xi\}$ and $\{a_i^\eta\}$ do not coincide eventually (i.e. there do not exist $i_0,j_0$ such that $a^\xi_{i_0+k}=a^\eta_{j_0+k}$ for all $k\geq 0$.) Since, under a quasi-isometry, $\mathcal{CP}_\xi$ must be mapped to $\mathcal{CP}_\eta$ and since the length of the $i$-th cycle of $\mathcal{CP}_\xi$ (respectively $\mathcal{CP}_\eta$) is $2^{\lceil a^\xi_i/2\rceil}$ (respectively $2^{\lceil a^\eta_i/2\rceil}$), we get a contradiction. On the other hand, if we suppose that the sequences $\{a_i^\xi\}$ and $\{a_i^\eta\}$ do eventually coincide, then condition (c) in Theorem 5.4 of \cite{DanDonMat09} is not satisfied which means that the difference of the distances between successive cut vertices of $\mathcal{CP}_\xi$, respectively $\mathcal{CP}_\eta$, diverges.\\
\indent To see that $\mathbb{Z}^2$ is not quasi-isometric to any orbital Schreier graph $\Gamma_\xi$ for $\xi\in E_1$, note that any quasi-isometry between infinite graphs maps a bi-infinite self-avoiding path to a bi-infinite self-avoiding path. However, there is no bi-infinite self-avoiding path in $\Gamma_\xi$, for any $\xi\in E_1$.
\end{rem}

\section{Avalanches on Basilica Schreier Graphs} \label{AvOnBasilica}

In this section, we study avalanches of the ASM on finite approximations of the infinite orbital Schreier graphs $(\Gamma_\xi,\xi)$, $\xi\in \{0,1\}^\omega$, of the Basilica group.\\
\indent Given $\xi\in\{0,1\}^\omega$ and given an exhaustion $\{H_n\}_{n\geq 1}$ of $(\Gamma_\xi,\xi)$ (see Convention \ref{convDissipative}), we look at the probability distribution, as $n\rightarrow\infty$, of the random variable $Mav_{H_n}(\cdot,\xi)$ giving the mass of an avalanche triggered by adding a chip on the root $\xi$ to a recurrent configuration on $H_n$ chosen uniformly at random. Recall from Section \ref{Basilica} that, for almost every infinite binary sequence $\xi$, the orbital Schreier graph $\Gamma_\xi$ has $1$ end (these boundary points are partitioned into uncountably many uncountable classes of isomorphic $\Gamma_\xi$'s); that there also exist an uncountable infinity of $\xi$'s that give rise to orbital Schreier graphs with $2$ ends (partitioned into countable isomorphism classes); and a countable number of $\xi$'s with a $4$-ended $\Gamma_\xi$ (all isomorphic as unrooted graphs). We examine separately the asymptotic distribution of the mass of avalanches depending on the
number of ends in the orbital infinite graph $\Gamma_\xi$. The four-ended and two-ended graphs
are shown to be non-critical (Theorems \ref{Main4Ends} and \ref{Main2ends}). However, almost every one-ended graph, and therefore also almost every orbital Schreier graph of the Basilica group is critical with the critical exponent equal to $1$ (Theorem \ref{Main}).

\subsection{Limit graph with four ends} \label{4ends}

Recall that all orbital Schreier graphs $\Gamma_\xi$'s that have $4$ ends are isomorphic to the graph $\Gamma_{(4)}$ described in Section \ref{Basilica}. Therefore it is enough to examine one such $\Gamma_\xi$, and we will consider $\xi=0^{\omega}$.\\
\indent Let $\{(\Gamma_n,0^n)\}_{n\geq 1}$ be the sequence of finite rooted Schreier graphs converging in $\mathcal{X}$ to $(\Gamma_{0^\omega},0^\omega)$. For any $n\geq 1$, we fix in $\Gamma_n$ four dissipative vertices as follows: consider the vertices $0^{n-1}1$ and $0^{n-2}10$; for each of them, its neighbours which are situated on a path from it to $0^n$ are dissipative. The infinite graph $(\Gamma_{0^\omega},0^\omega)$ is exhausted by the subgraphs $H_n$ that are isomorphic, for each $n$, to the connected component of $0^n$ in $\Gamma_n$ remaining when removing the above four vertices, together with these four dissipative vertices (see Remark \ref{remsubgraphsH_n}). As $n$ tends to infinity, both cycles in $\Gamma_n$ containing $0^n$ grow and split in the limit, sending vertices $0^{n-1}1$ and $0^{n-2}10$ to infinity and giving in the limit the four infinite paths in $\Gamma_{0^\omega}$ intersecting at $0^\omega$ (see Figure \ref{FOURENDEDGRAPH} and \cite{DanDonMat09}). Consequently, our choice of subgraphs $H_n$ and of dissipative vertices corresponds to our Convention \ref{convDissipative}.\\
\indent It is further convenient to merge in $H_n$ all four dissipative vertices into a single dissipative vertex $p$. The graph $\bar{H}_n$ obtained in this way is still separable but is not a cactus anymore. More precisely, all blocks of $\bar{H}_n$ but one are cycles denoted by $C_1,\dots,C_{s}$. Denote the exceptional block by $B$; it consists of vertices $0^n$ and $p$, and of four disjoint paths, $\mathcal{P}_1$ to $\mathcal{P}_4$, where $|\mathcal{P}_1|=|\mathcal{P}_2|=2^{\lceil\frac{n}{2}\rceil-1}-1$ whereas  $|\mathcal{P}_3|=|\mathcal{P}_4|=2^{\lceil\frac{n-1}{2}\rceil-1}-1$ (see Figure \ref{FOURENDEDGRAPH}). Note that considering the ASM on the graph $H_n$ is equivalent to consider it on $\bar{H}_n$. Indeed, merging all dissipative vertices into a single dissipative vertex does not affect neither the structure of chip configurations (as they are defined on non-dissipative vertices only) nor the firing rules (as dissipative vertices are never fired during the stabilization process) and hence avalanches. Also, $\mathcal{R}_{H_n}\equiv\mathcal{R}_{\bar{H}_n}$ since performing the Burning Algorithm on $H_n$ is equivalent to perform it on the graph $\bar{H}_n$ (as the graphs spanned by the sets of vertices $V_0(H_n)$ and $V_0(\bar{H}_n)$ are isomorphic.)


\begin{figure}[H]
\begin{center}
\begin{picture}(320,320)
\setloopdiam{3} \put(153,153){$0^n$}
\letvertex A=(150,150)\letvertex B=(175,150)\letvertex C=(200,150)\letvertex D=(225,150)\letvertex E=(250,150)\letvertex F=(275,150)\letvertex G=(295,150)
\put(297,150){\dots} \letvertex Gg=(307,150) \letvertex Ggg=(330,150)
\drawundirectededge(Gg,Ggg){} \put(330,155){$p_2\equiv p$} \drawvertex(Ggg){$\bullet$} \put(310,135){$\mathcal{P}_2$}
\put(149,295){$\vdots$} \letvertex Hh=(150,305) \letvertex Hhh=(150,325)
\drawundirectededge(Hh,Hhh){} \put(153,325){$p_1\equiv p$} \drawvertex(Hhh){$\bullet$} \put(130,300){$\mathcal{P}_1$}
\put(45,149){\dots} \letvertex Ii=(42,150) \letvertex Iii=(20,150)
\drawundirectededge(Ii,Iii){} \put(20,155){$p_4\equiv p$} \drawvertex(Iii){$\bullet$} \put(40,135){$\mathcal{P}_4$}
\put(149,45){$\vdots$} \letvertex Ii=(150,43) \letvertex Iii=(150,23)
\drawundirectededge(Ii,Iii){} \put(153,22){$p_3\equiv p$} \drawvertex(Iii){$\bullet$} \put(130,40){$\mathcal{P}_3$}
\letvertex H=(150,125)\letvertex I=(150,100)\letvertex L=(150,75)\letvertex M=(150,55)\letvertex N=(125,150)\letvertex O=(100,150)
\letvertex P=(75,150)\letvertex Q=(55,150)\letvertex R=(150,175)\letvertex S=(150,200)\letvertex T=(150,225)\letvertex U=(150,250)\letvertex V=(150,275)
\letvertex Z=(150,295)\letvertex J=(110,160)\letvertex K=(100,170)\letvertex X=(100,185)\letvertex W=(75,165)\letvertex Y=(100,200)
\letvertex a=(165,125)\letvertex b=(160,110)\letvertex c=(160,90)\letvertex d=(170,100)\letvertex e=(185,100)\letvertex f=(200,100)\letvertex g=(165,75)
\letvertex h=(200,135)\letvertex i=(200,120)\letvertex l=(230,140)\letvertex m=(240,140)\letvertex n=(260,140)\letvertex o=(270,140)\letvertex p=(250,130)
\letvertex q=(240,120)\letvertex r=(260,120)\letvertex s=(250,110)\letvertex t=(250,95)\letvertex u=(250,80)\letvertex v=(125,165)
\letvertex z=(90,160)\letvertex j=(135,200)\letvertex k=(120,200)\letvertex x=(140,270)\letvertex w=(140,240)\letvertex y=(140,260)
\letvertex aa=(140,230)\letvertex bb=(130,250)\letvertex cc=(120,260)\letvertex dd=(120,240)\letvertex ee=(110,250)
\letvertex ff=(95,250)\letvertex gg=(80,250)
\drawvertex(A){$\bullet$}\drawvertex(B){$\bullet$}\drawvertex(C){$\bullet$}\drawvertex(D){$\bullet$}\drawvertex(E){$\bullet$}\drawvertex(F){$\bullet$}
\drawvertex(H){$\bullet$}\drawvertex(I){$\bullet$}\drawvertex(L){$\bullet$}\drawvertex(N){$\bullet$}
\drawvertex(O){$\bullet$}\drawvertex(P){$\bullet$}\drawvertex(R){$\bullet$}\drawvertex(S){$\bullet$}\drawvertex(T){$\bullet$}
\drawvertex(U){$\bullet$}\drawvertex(V){$\bullet$}\drawvertex(X){$\bullet$}\drawvertex(Y){$\bullet$}\drawvertex(W){$\bullet$}\drawvertex(J){$\bullet$}
\drawvertex(K){$\bullet$}\drawvertex(a){$\bullet$}\drawvertex(b){$\bullet$}\drawvertex(c){$\bullet$}\drawvertex(d){$\bullet$}\drawvertex(e){$\bullet$}
\drawvertex(f){$\bullet$}\drawvertex(g){$\bullet$}\drawvertex(h){$\bullet$}\drawvertex(i){$\bullet$}\drawvertex(l){$\bullet$}\drawvertex(m){$\bullet$}
\drawvertex(n){$\bullet$}\drawvertex(o){$\bullet$}\drawvertex(p){$\bullet$}\drawvertex(q){$\bullet$}\drawvertex(r){$\bullet$}\drawvertex(s){$\bullet$}
\drawvertex(t){$\bullet$}\drawvertex(u){$\bullet$}\drawvertex(v){$\bullet$}\drawvertex(z){$\bullet$}\drawvertex(x){$\bullet$}\drawvertex(y){$\bullet$}
\drawvertex(w){$\bullet$}\drawvertex(j){$\bullet$}\drawvertex(k){$\bullet$}\drawvertex(aa){$\bullet$}\drawvertex(bb){$\bullet$}\drawvertex(cc){$\bullet$}
\drawvertex(dd){$\bullet$}\drawvertex(ee){$\bullet$}\drawvertex(ff){$\bullet$}\drawvertex(gg){$\bullet$}
\drawundirectededge(A,B){}\drawundirectededge(B,C){}\drawundirectededge(C,D){}\drawundirectededge(D,E){}\drawundirectededge(E,F){}
\drawundirectededge(F,G){}\drawundirectededge(A,H){}\drawundirectededge(H,I){}\drawundirectededge(I,L){}\drawundirectededge(L,M){}
\drawundirectededge(A,N){}\drawundirectededge(N,O){}\drawundirectededge(O,P){}\drawundirectededge(P,Q){}\drawundirectededge(A,R){}
\drawundirectededge(R,S){}\drawundirectededge(S,T){}\drawundirectededge(T,U){}\drawundirectededge(U,V){}\drawundirectededge(V,Z){}
\drawundirectededge(I,b){}\drawundirectededge(b,d){}\drawundirectededge(d,c){}\drawundirectededge(c,I){}\drawundirectededge(E,n){}
\drawundirectededge(n,p){}\drawundirectededge(p,m){}\drawundirectededge(m,E){}\drawundirectededge(p,r){}\drawundirectededge(r,s){}
\drawundirectededge(s,q){}\drawundirectededge(q,p){}\drawundirectededge(O,z){}\drawundirectededge(z,K){}\drawundirectededge(K,J){}
\drawundirectededge(J,O){}\drawundirectededge(U,w){}\drawundirectededge(w,bb){}\drawundirectededge(bb,y){}\drawundirectededge(y,U){}
\drawundirectededge(bb,dd){}\drawundirectededge(dd,ee){}\drawundirectededge(ee,cc){}\drawundirectededge(cc,bb){}
\drawundirectedcurvededge(a,H){}\drawundirectedcurvededge(H,a){}\drawundirectedcurvededge(d,e){}\drawundirectedcurvededge(e,d){}
\drawundirectedcurvededge(e,f){}\drawundirectedcurvededge(f,e){}\drawundirectedcurvededge(L,g){}\drawundirectedcurvededge(g,L){}
\drawundirectedcurvededge(C,h){}\drawundirectedcurvededge(h,C){}\drawundirectedcurvededge(h,i){}\drawundirectedcurvededge(i,h){}
\drawundirectedcurvededge(l,m){}\drawundirectedcurvededge(m,l){}\drawundirectedcurvededge(n,o){}\drawundirectedcurvededge(o,n){}
\drawundirectedcurvededge(s,t){}\drawundirectedcurvededge(t,s){}\drawundirectedcurvededge(t,u){}\drawundirectedcurvededge(u,t){}
\drawundirectedcurvededge(N,v){}\drawundirectedcurvededge(v,N){}\drawundirectedcurvededge(K,X){}\drawundirectedcurvededge(X,K){}
\drawundirectedcurvededge(X,Y){}\drawundirectedcurvededge(Y,X){}\drawundirectedcurvededge(P,W){}\drawundirectedcurvededge(W,P){}
\drawundirectedcurvededge(S,j){}\drawundirectedcurvededge(j,S){}\drawundirectedcurvededge(j,k){}\drawundirectedcurvededge(k,j){}
\drawundirectedcurvededge(gg,ff){}\drawundirectedcurvededge(ff,gg){}\drawundirectedcurvededge(ff,ee){}\drawundirectedcurvededge(ee,ff){}
\drawundirectedcurvededge(x,y){}\drawundirectedcurvededge(y,x){}\drawundirectedcurvededge(w,aa){}\drawundirectedcurvededge(aa,w){}
\drawundirectedloop[r](a){}\drawundirectedloop(b){}\drawundirectedloop[b](c){}\drawundirectedloop[r](f){}\drawundirectedloop[r](g){}\drawundirectedloop[b](B){}\drawundirectedloop[b](i){}
\drawundirectedloop[b](D){}\drawundirectedloop[l](l){}\drawundirectedloop[r](o){}\drawundirectedloop[l](q){}\drawundirectedloop[r](r){}\drawundirectedloop[b](u){}\drawundirectedloop[b](F){}
\drawundirectedloop[l](V){}\drawundirectedloop[l](gg){}\drawundirectedloop(cc){}\drawundirectedloop[b](dd){}\drawundirectedloop(x){}\drawundirectedloop[b](aa){}\drawundirectedloop[l](T){}
\drawundirectedloop[l](k){}\drawundirectedloop[l](R){}\drawundirectedloop(W){}\drawundirectedloop(Y){}\drawundirectedloop[l](V){}
\drawundirectedloop[r](J){}\drawundirectedloop(v){}\drawundirectedloop[l](z){}
\end{picture}
\end{center}
\caption{The graph $H_n$ as a subgraph of $\Gamma_n$.}
\label{FOURENDEDGRAPH}
\end{figure}
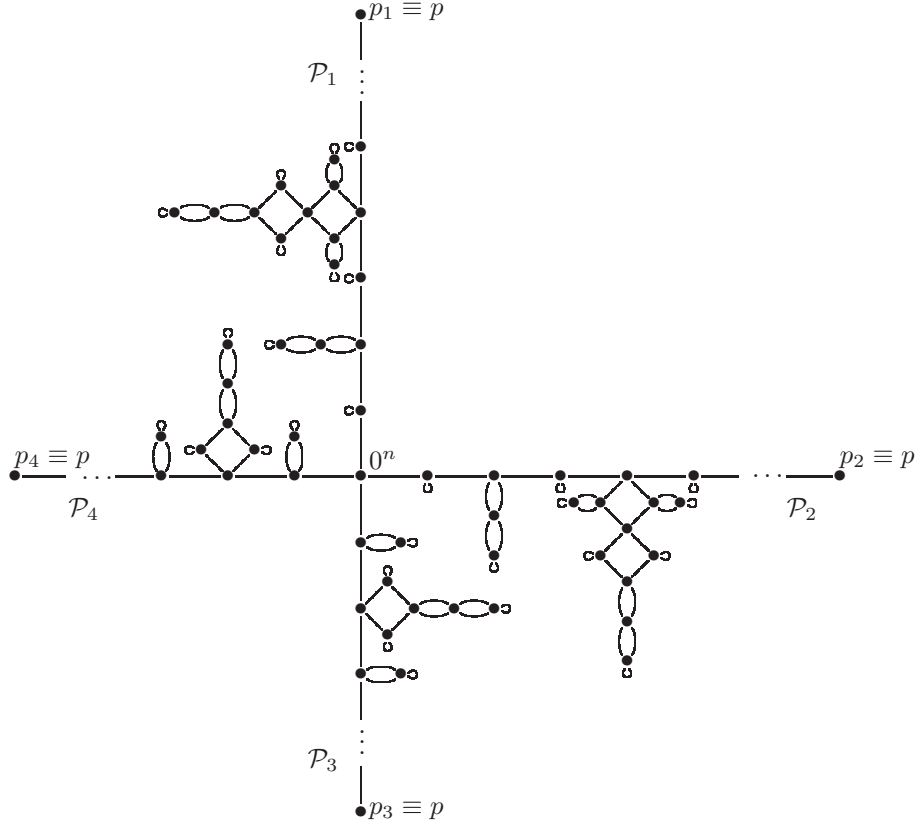

The description of recurrent configurations on $\bar{H}_n$ (and hence on $H_n$) follows now directly from Lemma \ref{ConfigIndlemma}, Proposition \ref{PropRecOnCycle} and from Theorem \ref{BA}. Given a block $C_i$ of $\bar{H}_n$, denote its vertices by $p_i,v^i_1,\dots,v^i_{|C_i|-1}$ (recall from Subsection \ref{CT} that $p_i$ denotes the smallest element of $V(C_i)$ in the order $\succeq$.)

\begin{prop} \label{recconf4ends}
A chip configuration $c:V_0(\bar{H}_n)\longrightarrow\mathbb{N}$ on $\bar{H}_n$ is recurrent if and only if it has the form

\begin{displaymath}
c=c^1_{j_1}+c^2_{j_2}+\dots+c^{s}_{j_{s}}+c^B
\end{displaymath}

\noindent where for every $1\leq i\leq s$, $j_i\in\{0,1,\dots,|C_i|-1\}$. If $j_i\neq 0$, then $c^i_{j_i}:V_0(\bar{H}_n)\longrightarrow \mathbb{N}$ is given by
\begin{displaymath}
c^i_{j_i}(w) = \left\{ \begin{array}{ll}
2 & \textrm{if $w=v^i_{j_i}$},\\
3 & \textrm{if $w=v^i_k$ for $k=1,\dots,|C_i|-1$, $k\neq j_i$},\\
0 & \textrm{otherwise,}
\end{array} \right.
\end{displaymath}
\noindent  whereas if $j_i=0$,
\begin{displaymath}
c^i_{0}(w) = \left\{ \begin{array}{ll}
3 & \textrm{if $w=v^i_k$ for $k=1,\dots,|C_i|-1$,}\\
0 & \textrm{otherwise.}
\end{array} \right.
\end{displaymath}
The subconfiguration $c^B:V_0(\bar{H}_n)\longrightarrow \mathbb{N}$ satisfies

\begin{enumerate}
\item $2\leq c^B(v)\leq 3$ for every $v\in V_0(B)\backslash\{0^n\}$;
\item for $1\leq i\leq 4$, $c^B(v)=2$ for at most one vertex $v\in V(\mathcal{P}_i)\backslash\{0^n\}$ with the additional condition that at least one path $\mathcal{P}_i$ is such that $c(v)=3$ for every $v\in V(\mathcal{P}_i)\backslash\{0^n\}$;
\item $\big|\{1\leq i\leq 4|\exists v\in V(\mathcal{P}_i)\backslash\{0^n\} \hspace{1mm}|\hspace{1mm}c^B(v)=2\}\big|\leq c^B(0^n)\leq 3$;
\item $c^B(w)=0$ for all $w\notin V_0(B)$.
\end{enumerate}
\end{prop}

The following result shows that the ASM on the sequence $\{H_n\}_{n\geq 1}$ approximating the infinite orbital Schreier graph $(\Gamma_{0^\omega},0^\omega)$ is non-critical in the sense of Definition \ref{DefCrit}:

\begin{thm} \label{Main4Ends}
Consider the infinite orbital Schreier graph $(\Gamma_{0^\omega},0^\omega)$. Then, there exist constants $C_1,C_2>0$ such that

\begin{equation}\label{Bounds4ends}
C_1\cdot2^{-\frac{3n}{2}}\leq\mathbb{P}_{\mu_n}(Mav_{H_n}(\cdot,0^n)=M)\leq C_2\cdot 2^{-n}.
\end{equation}
\end{thm}

\begin{proof}
Consider the graph $\bar{H}_n$ for $n>5$. Given any recurrent configuration $c$, it follows from Proposition \ref{RemOnAv} that the avalanche triggered by adding an extra chip on $0^n$ to $c$ does only depend on the subconfiguration $c^B$ of $c$ on the block $B$ of $\bar{H}_n$.\\
\indent Given an integer $M>0$, we count the number of recurrent configurations on $\bar{H}_n$ producing an avalanche of mass $M$. We first compute the total number of recurrent configurations on the block $B$, which is the number $\kappa(B)$ of spanning trees of $B$. Recall that  $|\mathcal{P}_1|=|\mathcal{P}_2|=2^{\lceil\frac{n}{2}\rceil-1}-1$ and $|\mathcal{P}_3|=|\mathcal{P}_4|=2^{\lceil\frac{n-1}{2}\rceil-1}-1$. As it does not influence the final result, we omit the additive constant for technical convenience, and we get

\begin{displaymath}
\kappa(B)=2|\mathcal{P}_1|^2|\mathcal{P}_3|+2|\mathcal{P}_3|^2|\mathcal{P}_1|=
\left\{\begin{array}{cc}
2^{3n/2-1} & \textrm{if $n$ is even,}\\
3\cdot2^{(3n-5)/2} & \textrm{if $n$ is odd.}
\end{array}\right.
\end{displaymath}

Given a recurrent configuration $c^B$ on the block $B$, denote by $l_i(c^B)$ the distance between $0^n$ and the vertex situated on $\mathcal{P}_i$ with only $2$ chips on it (see \emph{2.} in Proposition \ref{recconf4ends}). If there is no such vertex on some of the paths $\mathcal{P}_i$, then set $l_i(c^B)=|\mathcal{P}_i|$.
We look now at how the mass of avalanches triggered by adding an extra chip on $0^n$ depends on the $l_i$'s: if at least one of the $l_i$'s grows (respectively decreases), then the mass grows (respectively decreases). Thus, in order to keep the mass $M$ of the avalanche unchanged while modifying the values of the $l_i$'s, we must let some of them grow as well as some of them decrease. Suppose without loss of generality that decorations of odd heights are attached to the paths $\mathcal{P}_1$ and $\mathcal{P}_2$, whereas decorations of even heights are attached to the paths $\mathcal{P}_3$ and $\mathcal{P}_4$. It follows (see proof of Proposition \ref{RemOnAv}) that an increase of $l_1$ (respectively $l_3$) must be thus compensated by a decrease of $l_2$ (respectively $l_4$) whereas an increase of $l_2$ (respectively $l_4$) must be compensated by a decrease of $l_1$ (respectively $l_3$).\\
\indent Observe now that if $c_1$ and $c_2$ are two recurrent configurations on $\bar{H}_n$ such that $l_1(c_1^B)+l_2(c_1^B)\neq l_1(c_2^B)+l_2(c_2^B)$ (or similarly $l_3(c_1^B)+l_4(c_1^B)\neq l_3(c_2^B)+l_4(c_2^B)$), then the masses of the avalanches triggered respectively by $c_1$ and $c_2$ are different.\\
\indent It follows from the previous observation, that avalanches which are less likely to occur are those of small mass. We derive the lower-bound in (\ref{Bounds4ends}) by counting the number of recurrent configurations on $B$ leading to avalanches on $\bar{H}_n$ of minimal mass. There are exactly two such recurrent configurations $c_{min}^B$ and $d_{min}^B$; $c_{min}^B$ satisfies $l_1(c_{min}^B)=l_2(c_{min}^B)=l_3(c_{min}^B)=1$ whereas $d_{min}^B$ satisfies $l_1(d_{min}^B)=l_2(d_{min}^B)=l_4(d_{min}^B)=1$. Normalizing by $\kappa(B)$ yields the lower bound in (\ref{Bounds4ends}).\\
\indent On the other hand, the most likely avalanches arise from recurrent configurations $c$ with $c^B$ on $B$ satisfying $l_1(c^B)+l_2(c^B)=|\mathcal{P}_1|+1$. There are not more than $2|\mathcal{P}_2|+2(|\mathcal{P}_1|-2)\leq 2^{\lceil\frac{n}{2}\rceil+1}$ such recurrent configurations on $B$. Normalizing by $\kappa(B)$ yields the upper-bound in (\ref{Bounds4ends}).
\end{proof}

\begin{rem}\rm
A careful computation yields approximate values for the constants $C_1\approx 3.77$ and $C_2\approx 5.65$.
\end{rem}

\subsection{Limit graphs with two ends}

The Basilica group provides us with an uncountable family of two-ended graphs not quasi-isometric to $\mathbb{Z}$ (see Section \ref{Basilica}). We prove in this subsection that the ASM on sequences of finite graphs approximating these infinite graphs does not exhibit a critical behaviour with respect to the mass of avalanches, in the sense of Definition \ref{DefCrit}.\\
\indent Some particular cases of sequences of cacti approximating a $2$-ended graph were already studied by Ali and Dhar in \cite{AliDhar95} where they considered graphs obtained from $\mathbb{Z}$ by replacing even edges by cycles of fixed length $L$. (Note that if $L=2$, the corresponding graphs are essentially the Schreier graphs associated with the self-similar action on the binary rooted tree of the so-called Grigorchuk group, the first example of a group of intermediate growth). Ali and Dhar have found that the ASM on these sequences of \emph{decorated chains} is not critical; in particular, they have shown that $\mathbb{P}_{\mu_n}(Mav_{\Gamma_n}(\cdot,v_n)=M)=f(\frac{M}{n})n^{-1}$ where $f$ denotes some scaling function. The behaviour of avalanches with respect to their mass is thus similar to what one obtains on a sequence of growing cycles $C_n$ of length $n$ approximating the lattice $\mathbb{Z}$ where $\mathbb{P}_{\mu_n}(Mav_{C_n}(\cdot,v_n)=M)\sim 1/n$ (see Subsection \ref{Cycles}).\\
\indent Let $\xi\in E_2$, let $\Gamma_\xi$ be the corresponding two-ended orbital Schreier graph and let $\{(\Gamma_n,\xi_n)\}_{n\geq 1}$ be the sequence of finite rooted Schreier graphs converging in $\mathcal{X}$ to $(\Gamma_\xi,\xi)$. Recall that $\mathcal{CP}_{\xi_n}$ denotes the block-path in $\Gamma_n$ joining the vertex $\xi_n$ to the central cycle of $\Gamma_n$.\\
\indent For any $n\geq 1$, we fix in $\Gamma_n$ two dissipative vertices $p_1$ and $p_2$; these are the two neighbours of $0^n$ such that any path joining $\xi_n$ to $0^n$ contains one of them. The infinite graph $(\Gamma_\xi,\xi)$ is exhausted by the subgraphs $H_n$ that are isomorphic, for each $n$, to the connected component of $\xi_n$ in $\Gamma_n$ remaining when removing both above vertices, together with them (see Remark \ref{remsubgraphsH_n}). As $n$ tends to infinity, the length of the central cycle in $\Gamma_n$ grows and splits in the limit, sending vertex $0^n$ to infinity and giving the bi-infinite path in $\Gamma_\xi$ (see Figure \ref{TWOENDEDGRAPH} and \cite{DanDonMat09}). Consequently, our choice of subgraphs $H_n$ and of dissipative vertices corresponds to our Convention \ref{convDissipative}.\\
\indent The recurrent configurations on $H_n$ are given by Lemma \ref{ConfigIndlemma} and Proposition \ref{PropRecOnCycle} (as in Subsection \ref{4ends}, we may merge both dissipative vertices $p_1$ and $p_2$ into a single one, $p$; the resulting graph $\bar{H}_n$ is still separable and $\mathcal{R}_{H_n}\equiv\mathcal{R}_{\bar{H}_n}$.)\\
\indent As in the case of $\Gamma_{0^\omega}$, the ASM on the sequence $\{H_n\}_{n\geq 1}$ approximating the infinite orbital Schreier graph $(\Gamma_\xi,\xi)$ has non-critical behaviour if $\xi\in E_2$:

\begin{thm} \label{Main2ends}
Let $\xi\in E_2$ and consider the two-ended orbital rooted Schreier graph $(\Gamma_\xi,\xi)$.
Then, the probability distribution of the mass of an avalanche on $H_n$ satisfies

\begin{displaymath}
\mathbb{P}_{\mu_n}(Mav_{H_n}(\cdot, \xi_n)=M)\sim 2^{-\frac{n}{2}}.
\end{displaymath}
\end{thm}

\begin{proof}
Let $\xi\in E_2$. Observe that there exists a subsequence $\{n_i\}_{i\geq 1}$ of $\mathbb{N}$ such that, for every $i\geq1$, the vertex separating the penultimate cycle of the block-path $\mathcal{CP}_{\xi_{n_i}}$ from the last cycle of $\mathcal{CP}_{\xi_{n_i}}$ is different from $0^n$. Let $n\geq 1$ belong to such a subsequence and consider the graph $\Gamma_n$. The root $\xi_n$ belongs to some $k$-decoration attached to the central cycle of $\Gamma_n$ by some vertex $v\neq p_i$, $i=1,2$. Note that if we choose $n$ large enough, $k$ does not depend on $n$.\\
\indent Let $c$ be a (randomly chosen) recurrent configuration on $H_n\subset\Gamma_n$. By Proposition \ref{RemOnAv}, the mass of the avalanche triggered by adding to $c$ an extra chip on $\xi_n$ depends only on the subconfigurations of $c$ on $\mathcal{CP}_{\xi_n}$.\\
\indent As the avalanche propagates along the $k$-decoration attached at $v$, a certain amount of chips migrates in the direction of the central cycle of $\Gamma_n$ and finally reaches $v$. If the amount of chips eventually reaching $v$ is greater than one, then necessarily, the avalanche will propagate in both directions on the whole central cycle and the mass of the avalanche will be maximal (denote this mass by $M_{max}$). The same happens if only one chip reaches $v$ but every vertex on the central cycle has three chips on it. On the other hand, if only one chip reaches $v$ and if there is a vertex on the central cycle with only two chips on it, then the avalanche will propagate along the central cycle in such a way that in one direction it will reach one of the dissipative vertices but in the other direction, it will be stopped at the vertex with only two chips. Denote by $P$ the probability that at least two chips reach $v$ during an avalanche. Similarly, denote by $\tilde{P}$ the probability that the mass $M$ of the avalanche is greater than the cardinality of the decoration attached to $v$ (which, by Proposition \ref{PropertiesGamma_n}, is equal to $1/3(2^k+1)$ or $1/3(2^k+2)$ depending on the parity of $k$). Note that neither $P$ nor $\tilde{P}$ depend on $n$.\\
\indent Observe that, by Proposition \ref{MassAvCycle} and its proof, there are at most two subconfigurations on the central cycle producing avalanches of the same mass. Collecting together all previous observations, we have, for $M$ sufficiently large,

\begin{displaymath}
\mathbb{P}_{\mu_n}(Mav_{H_n}(\cdot, \xi_n)=M)=\left\{\begin{array}{cc}
\tilde{P}\cdot\frac{\alpha}{|C|-2} & \textrm{if $M<M_{\max}$,}\\
P+\frac{1-P}{|C|-2} & \textrm{if $M=M_{max}$.}
\end{array}\right.
\end{displaymath}

\noindent where $\alpha\in\{1,2\}$ and $|C|$ denotes the length of the central cycle. Since $|C|\sim 2^{\frac{n}{2}}$, the result follows.
\end{proof}

\subsection{Limit graphs with one end} \label{SubSectOneEnd}

Recall from Section \ref{Basilica}, that $E_1\subset\{0,1\}^\omega$ denotes the subset of full measure consisting of such rays $\xi$ that the infinite orbital Schreier graph $\Gamma_\xi$ has one end. For $\xi\in E_1$, consider the sequence $\{\xi_n\}_{n\geq 1}$ of vertices of the ray belonging to the consecutive levels of the tree, and the rooted finite Schreier graphs $\{(\Gamma_n,\xi_n)\}_{n\geq 1}$ converging to $(\Gamma_\xi,\xi)$. Let $\mathcal{CP}_{\xi_n}=\mathcal{C}_1\dots \mathcal{C}_{r_n}$ be the unique block-path in $\Gamma_n$ joining $\xi_n$ to the central cycle of $\Gamma_n$. By Lemma \ref{LemOnCPinfBasilica}, $(\mathcal{CP}_\xi,\xi)=\lim_{n\to\infty}(\mathcal{CP}_{\xi_n},\xi_n)$ is a well-defined block-path isomorphic to the unique block-path of infinite length in $(\Gamma_\xi,\xi)$ starting at $\xi$. Recall that there exists a subsequence $\{n_i\}_{i\geq 1}$ of $\mathbb{N}$ such that for every $i\geq1$, the vertex separating the penultimate cycle of the block-path $\mathcal{CP}_{\xi_{n_i}}$ from the last cycle of $\mathcal{CP}_{\xi_{n_i}}$ is different from $0^n$. Let $n\geq 1$ belong to such a subsequence. For any $n\geq 1$, we set $p^{(n)}:=0^n$ in $\Gamma_n$ to be dissipative. The infinite graph $(\Gamma_\xi,\xi)$ is exhausted by the subgraphs $H_n$ that are isomorphic, for each $n$, to the connected component of $\xi_n$ in $\Gamma_n$ remaining when removing vertex $0^n$, together with $0^n$ (see Remark \ref{remsubgraphsH_n}). Our choice of subgraphs $H_n$ corresponds to Convention \ref{RemOnCPinf}. The following statement is the main result of this section:

\begin{thm} \label{Main}
For almost every $\xi\in E_1$ (with respect to the uniform measure $\lambda$ on  $\{0,1\}^\omega$), we have

\begin{equation}
\lim_{n\to\infty}\mathbb{P}_{\mu_n}(Mav_{H_n}(\cdot,\xi_n)=M)\sim M^{-1}.
\end{equation}
\end{thm}

\noindent As an immediate consequence of Theorem \ref{Main} and Proposition \ref{unifmeasure}, we have:

\begin{cor}
The ASM on the sequence $\{\Gamma_n\}_{n\geq 1}$ of Schreier graphs of the Basilica group is critical in the random weak limit, with critical exponent equal to $1$.
\end{cor}

\indent Given $\xi\in E_1$, let $(l,\{m_k\},\{t_k\})$ be the triple provided by Lemma \ref{PropAlg} and let $\{a_i\}_{i\geq 1}$ be the sequence associated with $\xi$ as defined in Lemma \ref{remonai}, so that the size of the $i$-th block of $\mathcal{CP}_\xi$ is $2^{\lceil\frac{a_i}{2}\rceil}$.\\
\indent In order to prove Theorem \ref{Main}, we will need the following lemma:

\begin{lem} \label{LemExpDecay}
Choose $\xi\in E_1$ uniformly at random. Then, there is almost surely only a finite number of indices $j$ such that the corresponding terms of the sequence $\{m_k\}_{k\geq 0}$ associated with $\xi$ satisfy $m_j\geq 2j$.
\end{lem}

\begin{proof}
With any $\xi\in E_1$ is associated a triple $(l,\{m_k\},\{t_k\})$ given by Lemma \ref{PropAlg}. For any $j\geq 1$, define the event $A_j:=\{\xi\in E_1| m_j\geq 2j\}$. By definition of the sequence $\{m_k\}$ (see (\ref{CanFormOfXi}) in Lemma \ref{PropAlg}), for all $r>0$, we have $\mathbb{P}(m_j\geq 2r)\leq 2^{-r}$. Thus, $\mathbb{P}(A_j)\leq 2^{-j}$ and, by Borel-Cantelli Lemma, $\mathbb{P}(\limsup_{j\to\infty}A_j)=0$.
\end{proof}
We turn now to the proof of Theorem \ref{Main}:

\begin{proof}
Choose $\xi\in E_1$ uniformly at random. For any $n\geq 1$, consider the finite Schreier graph $(\Gamma_n,\xi_n)$, the block-path $\mathcal{CP}_{\xi_n}$ and the sequence $\{a_i\}_{i\geq 1}$ associated with $\xi$ (see Lemma \ref{remonai}). For further convenience, we interpolate the sequence $\{a_i\}_{i\geq 1}$ by an increasing continuous function $a:[0,+\infty)\longrightarrow [0,+\infty)$ such that $a(0)=0$.\\
\indent Recalling that $|\mathcal{C}_{i}|=2^{\lceil a(i)/2\rceil}$ for every $i\geq 1$, the series

\begin{displaymath}
\sum_{i=1}^\infty\frac{1}{|\mathcal{C}_{i}|}=\sum_{i=1}^\infty 2^{-\lceil\frac{a(i)}{2}\rceil}
\end{displaymath}

\noindent converges, and it follows from Theorem \ref{ThmCP} that

\begin{equation} \label{eq1MainProof}
\frac{L}{2\cdot|\mathcal{C}_{i_M}|\cdot|\mathcal{C}_{i_M+1}|} \leq\mathbb{P}_{\mu_n}(Mav_{H_n}(\cdot,\xi_n)=M)\leq \frac{2}{|\mathcal{C}_{i_M}|\cdot|\mathcal{C}_{i_M+1}|}
\end{equation}

\noindent where $\mathcal{C}_{i_M}$ denotes the block on which each avalanche of mass $M$ stops, and $0<L\leq 1$ is a constant depending on the sequence $\{a_i\}_{i\geq 1}\equiv \{a(i)\}_{i\geq 1}$. From (\ref{eq1MainProof}), we get

\begin{equation} \label{eq2MainProof}
\frac{L}{4}\cdot 2^{-\frac{a(i_M)+a(i_M+1)}{2}} \leq\mathbb{P}_{\mu_n}(Mav_{H_n}(\cdot,\xi_n)=M)\leq 2\cdot 2^{-\frac{a(i_M)+a(i_M+1)}{2}}.
\end{equation}

On the other hand, the mass of an avalanche which stops on $\mathcal{C}_{i_M}$ is bounded by

\begin{equation} \label{ineqCard}
|\mathcal{D}(0^{a(i_M-1)+1})|<M<|\mathcal{D}(0^{a(i_M)+1})|,
\end{equation}

\noindent where $|\mathcal{D}(0^{a(i_M)+1})|$ is the number of vertices in the decoration of vertex $0^{a(i_M)+1}$ in $\Gamma_{a(i_M)+1}$. By Proposition \ref{PropertiesGamma_n}, (\ref{ineqCard}) implies

\begin{displaymath}
\frac{1}{3}\big(2^{a(i_M-1)+1}+1\big)<M<\frac{1}{3}\big(2^{a(i_M)+1}+2\big).
\end{displaymath}

\noindent These inequalities can be rewritten as

\begin{displaymath} \left\{\begin{array}{c}
a(i_M-1)<\log(3M-1)-1,\\
a(i_M)>\log(3M-2)-1\\
\end{array}\right.
\end{displaymath}

\noindent where $\log(\cdot)\equiv\log_2(\cdot)$. Since $a$ is increasing, one may write

\begin{displaymath} \left\{\begin{array}{c}
i_M<a^{-1}(\log(3M-1)-1)+1,\\
i_M>a^{-1}(\log(3M-2)-1).\\
\end{array}\right.
\end{displaymath}

\noindent The difference $a^{-1}(\log(3M-1)-1)+1-a^{-1}(\log(3M-2)-1)$ tends to $1$ as $M\to\infty$. We can then assume that $i_M=\lfloor a^{-1}(\log(3M))\rfloor$ for $M$ sufficiently large.\\
\indent We show that, almost surely, $a(i_M+1)/a(i_M)$ tends to $1$ as $M\to\infty$. Recall that (see Lemma \ref{remonai}), for all $j\geq 1$, $0\leq s< t_j$, $a(T_{j-1}+s+1)=l+M_{j-1}+T_{j-1}+s$, where $M_j:=\sum_{k=0}^jm_k$ and $T_j:=\sum_{k=0}^jt_k$. Writing $i:=T_{j-1}+s+1$, $a(i)=l+M_{j-1}+i-1$; we consider $j\equiv j(i)$ as a (non-decreasing) function of $i$ (corresponding to the number of terms in the sum $M_{j-1}$.) Note that $j(i)\leq i$. By Lemma \ref{remonai},

\begin{displaymath}
a(i+1)-a(i)=\left\{\begin{array}{cc}
m_{j(i)}+1 & \textrm{if $i$ is such that $i=T_{j(i)}$,}\\
1 & \textrm{otherwise.}
\end{array}\right.
\end{displaymath}

\noindent  On the other hand, it follows from Lemma \ref{LemExpDecay} that, almost surely, there exists $j_0\geq 1$ such that $m_j\leq 2j$ for all $j>j_0$. We thus have

\begin{displaymath}
1\leq\frac{a(i+1)}{a(i)}\leq \frac{a(i)+m_{j(i)}+1}{a(i)}\leq 1+\frac{2j(i)}{a(i)}+\frac{1}{a(i)}
\end{displaymath}

\noindent where the last inequality holds almost surely for any $i$ sufficiently large. Clearly, $2j(i)/a(i)\leq 2j(i)/i$. We check that $j(i)/i$, which is non-increasing, tends to $0$ as $i\to\infty$. For the sake of contradiction, suppose that $j(i)/i$ tends (from above) to $C>0$ as $i\to\infty$. 
It is easy to check that, given any finite word $w\in\{0,1\}^\ast$, $w$ appears almost surely as a subword in $\xi\in\{0,1\}^\omega$ situated as far as we want in $\xi$, i.e., given $n_0\geq 1$, $\mathbb{P}\left(\textrm{$\xi=\xi_nw\xi'$, $n\geq n_0$, $\xi'\in \{0,1\}^\omega$}\}\right)=1$. It follows that, almost surely, the sequence $\{t_k\}_{k\geq 0}$ (see Lemma \ref{PropAlg}) is not bounded. Thus, we can find $t_{k_0}$ large enough such that $j(i_0)/i_0<C$ where $i_0=T_{k_0-1}+t_{k_0}$ and $j(i_0)=k_0$, and we get a contradiction. As $i_M$ tends to infinity as $M\to\infty$, we conclude that, almost surely, $a(i_M+1)/a(i_M)$ tends to $1$ as $M\to\infty$.\\
\indent Write $x_M:=a^{-1}(\log(3M))$ so that $i_M=\lfloor x_M\rfloor$. For any $\epsilon>0$, there exists $M_0$ such that for all $M>M_0$, $a(x_M)<a(i_M+1)\leq(1+\epsilon)a(i_M)$. It follows then from (\ref{eq2MainProof}) that for $n$ and $M$ sufficiently large,

\begin{displaymath}
\frac{L}{4}\cdot 2^{-\frac{a(x_M)(2+\epsilon)}{2}} \leq\mathbb{P}_{\mu_n}(Mav_{H_n}(\cdot,\xi_n)=M)\leq 2\cdot 2^{-\frac{a(x_M)}{1+\epsilon}},
\end{displaymath}

\noindent and hence,

\begin{displaymath}
\frac{L}{4}\cdot (3M)^{-\frac{2+\epsilon}{2}}\leq\mathbb{P}_{\mu_n}(Mav_{H_n}(\cdot,\xi_n)=M)
\leq 2\cdot (3M)^{-\frac{1}{1+\epsilon}}.
\end{displaymath}

\noindent We thus conclude that, almost surely, $\lim_{n\to\infty}\mathbb{P}_{\mu_n}(Mav_{H_n}(\cdot,\xi_n)=M)\sim M^{-1}$.
\end{proof}

\section{Schreier Graphs of $IMG(-z^3/2+3z/2)$ -- Examples with the Critical Exponent $>1$} \label{SectionIAM}

In this section, we examine the ASM on Schreier graphs of still another (though similar to
the Basilica) self-similar group, and compute the critical exponent for the mass of avalanches in the random weak limit to be $2\log 2/\log 3 > 1$.

\subsection{Interlaced adding machines}

The adding machine $\mathcal{A}$ is a group of automorphisms of the binary rooted tree generated by an automorphism $a$ defined self-similarly by $a=(0 \ 1)(id,a)$. Thus, the action of $a$ on the $n$-th level of the tree corresponds to adding one to the binary representation of integers modulo $2^n$ (recall that vertices of the $n$-th level are identified with binary words of length $n$.) It follows that, for any $n\geq 1$, the Schreier graph $\Gamma(\mathcal{A},\{a\},\{0,1\}^n)$ is a cycle of length $2^n$. The action of the automorphism $a$ on the boundary of the tree is free and the group generated by $a$ is $\mathbb{Z}$. It follows that the orbital Schreier graphs $\Gamma(\mathcal{A},\{a\},\mathcal{A}\cdot\xi)$, for $\xi\in\{0,1\}^\omega$, are all isomorphic (as unlabeled graphs) to the bi-infinite path. In other words, the random weak limit of the sequence $\{\Gamma(\mathcal{A},\{a\},\{0,1\}^n)\}_{n\geq 1}$ is atomic and supported by a single graph, which is $\mathbb{Z}$. As mentioned in the introduction, it is easy to see that the ASM is not critical in this case.\\
\indent The interlaced adding machines group $\mathcal{I}$ is a spherically transitive group of automorphisms of the ternary rooted tree $T$ generated by two automorphisms $a$ and $b$ with the following self-similar structure:

\begin{displaymath}
\begin{array}{ccc}
a=(0 \ 1)(id,a,id), & & b=(0 \ 2)(id,id,b).
\end{array}
\end{displaymath}

\noindent The group $\mathcal{I}$ is the iterated monodromy group of the complex polynomial $-z^3/2+3z/2$ (see \cite{Nek}), whose Julia set is represented in Figure \ref{JULIAIAM}.

\begin{figure}[H]
\begin{center}
\includegraphics[scale=0.3]{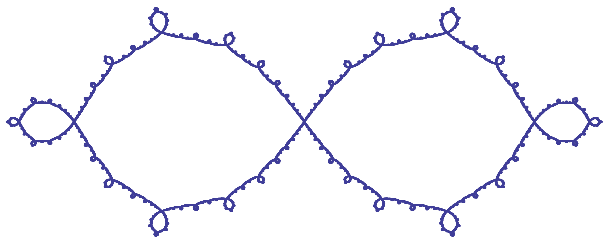}
\end{center}
\begin{center}
\end{center}
\caption{The Julia set $\mathcal{J}(-z^3/2+3z/2)$.}
\label{JULIAIAM}
\end{figure}

One notices that this Julia set looks very much like the Basilica Julia set (see Figure \ref{JULIABASILICA}). The Basilica Schreier graphs and the Schreier graphs $\widetilde{\Gamma}_n:=\Gamma(\mathcal{I},\{a,b\},\{0,1,2\}^n)$ are also very similar.\\
\indent It follows directly from the definition of the group $\mathcal{I}$, that for any $n\geq 1$, the Schreier graph $\widetilde{\Gamma}_n$ is a $4$-regular cactus and has all its edges labeled either by $a$ or by $b$. The number of vertices of $\widetilde{\Gamma}_n$ is $3^n$, so that the covering map $\pi_{n+1}:\widetilde{\Gamma}_{n+1}\longrightarrow\widetilde{\Gamma}_n$ is of degree $3$.\\
\indent By \cite{BondDanNag11}, the Schreier graphs $\widetilde{\Gamma}_\xi:=\Gamma(\mathcal{I},\{a,b\},\mathcal{I}\cdot\xi)$ have either $1$, $2$ or $4$ ends, and the number of ends is one for almost all $\xi$ with respect to the uniform measure on the boundary $\partial T$ of the tree. More precisely, we have a classification in terms of ternary sequences of the orbital Schreier graphs with respect to their number of ends, in the spirit of the Basilica case treated in \cite{DanDonMat09}. Given a word $w\in\{0,1,2\}^\ast$, we say that $w$ is of type $A$ (respectively $B$) if it does not contain the letter $2$ (respectively $1$). Any word (finite or infinite) in $\{0,1,2\}$ can be decomposed into an alternative succession of blocks of type $A$ and $B$.

\begin{thm} \label{ThmSchreierInterlaced}
\begin{enumerate}
\item The orbital Schreier graph $\widetilde{\Gamma}_\xi$ has one end if and only if the number of blocks in the decomposition of $\xi$ into blocks of type $A$ and $B$ is infinite;
\item the orbital Schreier graph $\widetilde{\Gamma}_\xi$ has four ends if and only if $\xi\in\{w0^\omega,w1^\omega,w2^\omega|w\in \{0,1,2\}^\ast\}$;
\item in all other cases, the orbital Schreier graph $\widetilde{\Gamma}_\xi$ has two ends.
\end{enumerate}
\end{thm}

\noindent For $i=1,2,4$, denote by $\widetilde{E}_i:=\{\xi\in\{0,1,2\}^\omega| \ \textrm{the orbital Schreier graph $\widetilde{\Gamma}_\xi$ has $i$ ends}\}$. Moreover, we also have:

\begin{prop} \label{PropUncount1endedIAM}
There exist uncountably many non-isomorphic orbital Schreier graphs with one end.
\end{prop}

\noindent Proposition \ref{PropUncount1endedIAM} follows from Theorem \ref{ThmSchreierInterlaced} together with the following lemma, proved similarly to Proposition 5.6 in \cite{DanDonMat09}.

\begin{lem} \label{PropSchreierInterlaced}
Let $w\in\{0,1,2\}^n$. Then,

\begin{enumerate}
\item the total number of blocks in the decomposition of $w$ into blocks of type $A$ and $B$ equals the number of blocks in the block-path $\mathcal{CP}_{w}$ in $\widetilde{\Gamma}_n$ joining $w$ to $0^n$;
\item the size of the $i$-th block in the block-path $\mathcal{CP}_{w}$ is equal to $2^{\nu_i}$, where $\nu_i$ denotes the length of the prefix of $w$ containing the $i$ first blocks.
\end{enumerate}
\end{lem}

\noindent We will also need the following result obtained by following the method developed in \cite{Bond07}.

\begin{prop}
For almost every $\xi\in \{0,1,2\}^\omega$ (with respect to the uniform measure $\lambda$ on $\xi\in \{0,1,2\}^\omega$), the degree of polynomial growth of $\widetilde{\Gamma}_\xi$ is $\log 3/\log 2$.
\end{prop}

\subsection{Criticality of the ASM on the Schreier graphs of $IMG(-z^3/2+3z/2)$}

In this subsection, we consider avalanches of the ASM on finite approximations of the infinite orbital rooted Schreier graphs $(\widetilde{\Gamma}_\xi,\xi)$, where $\xi\in\widetilde{E}_1$. For any $n\geq 1$, we set the vertex $p^{(n)}:=0^n$ in $\widetilde{\Gamma}_n$ to be dissipative. As in the case of Basilica Schreier graphs,
the infinite graph $(\widetilde{\Gamma}_\xi,\xi)$ is exhausted by the subgraphs $H_n$ that are isomorphic, for each $n$, to the connected component of $\xi_n$ in $\Gamma_n$ remaining when removing vertex $0^n$, together with $0^n$ (see Remark \ref{remsubgraphsH_n}). It follows from Lemma \ref{PropSchreierInterlaced} and Theorem \ref{ThmSchreierInterlaced}, that the number of blocks in the block-path $\mathcal{CP}_{\xi_n}$ joining $\xi_n$ to $0^n$ in $(\widetilde{\Gamma}_n,\xi_n)$ tends to infinity as $n\to\infty$. Consequently, our choice of subgraphs $H_n$ corresponds to Convention \ref{RemOnCPinf}. We will prove the following:

\begin{thm}\label{IAM}
For almost every $\xi\in\widetilde{E}_1$ (with respect to the uniform measure $\lambda$ on $\{0,1,2\}^\omega$), we have

\begin{displaymath}
\lim_{n\to\infty}\mathbb{P}_{\mu_n}(Mav_{H_n}(\cdot,\xi_n)=M)\sim M^{-\frac{2\log 2}{\log 3}}.
\end{displaymath}
\end{thm}

\begin{cor}
We thus exhibit an uncountable family of non-isomorphic $4$-regular, one-ended graphs of superlinear but subquadratic growth, such that the ASM on the sequences of finite graphs approximating them is critical with critical exponent equal to $2\log 2/\log 3>1$.
\end{cor}

\begin{proof}
Let $\xi\in\widetilde{E}_1$ and let $\xi=A_1B_1A_2B_2\dots$ be its decomposition in blocks of type $A$ and $B$ ($A_1$ may be empty). For any $n\geq 1$, consider the Schreier graph $(\widetilde{\Gamma}_n,\xi_n)$, the block-path $\mathcal{CP}_{\xi_n}$ in $\widetilde{\Gamma}_n$ and denote its blocks by $\mathcal{C}_1\mathcal{C}_2\dots\mathcal{C}_{r_n}$, so that $r_n$ is the number of blocks in the above decomposition of the prefix $\xi_n$ of $\xi$. By Lemma \ref{PropSchreierInterlaced}, for any $i\geq 1$, the size of $\mathcal{C}_i$ is given by

\begin{equation}\label{eqnu}
\log_2(|\mathcal{C}_i|)=\nu_i=\left\{\begin{array}{cc}
\sum_{k=1}^{i/2}(|A_k|+|B_k|) & \textrm{if $i$ is even,}\\
\sum_{k=1}^{(i-1)/2}(|A_k|+|B_k|)+|A_{(i+1)/2}| & \textrm{if $i$ is odd,}
\end{array}\right.
\end{equation}

where $|A_k|$ (respectively $|B_k|$) denotes the length of the block $A_k$ (respectively $B_k$). For further convenience, we interpolate the sequence $\{\nu_i\}_{i\geq 1}$ by a continuous, increasing function $\nu:[0,+\infty)\longrightarrow [0,+\infty)$ such that $\nu(0)=0$. As the series $\sum_{i\geq 1}\frac{1}{|\mathcal{C}_i|}$ converges, it follows from Theorem \ref{ThmCP}, that

\begin{equation} \label{eq1Interlaced}
\frac{L}{2\cdot|\mathcal{C}_{i_M}|\cdot|\mathcal{C}_{i_M+1}|} \leq\mathbb{P}_{\mu_n}(Mav_{H_n}(\cdot,\xi_n)=M)\leq \frac{2}{|\mathcal{C}_{i_M}|\cdot|\mathcal{C}_{i_M+1}|}
\end{equation}

\noindent where $\mathcal{C}_{i_M}$ denotes the block on which each avalanche of mass $M$ stops, and $0<L\leq 1$ is a constant depending on the sequence $\{\nu_i\}_{i\geq 1}\equiv\{\nu(i)\}_{i\geq 1}$. From (\ref{eq1Interlaced}), we get

\begin{equation} \label{eq2Interlaced}
\frac{L}{2}\cdot 2^{-(\nu(i_M)+\nu(i_M+1))} \leq\mathbb{P}_{\mu_n}(Mav_{H_n}(\cdot,\xi_n)=M)\leq 2\cdot 2^{-(\nu(i_M)+\nu(i_M+1))}.
\end{equation}

Observe that, for any $n\geq 1$ and $1\leq k\leq n$, the cardinality of a $k$-decoration in $\widetilde{\Gamma}_n$ (see Definition \ref{DefDeco}) is equal to $1/2(3^k+1)$. It follows that the mass of an avalanche which stops on $\mathcal{C}_{i_M}$ is bounded by

\begin{displaymath}
1/2(3^{\nu(i_M-1)}+1)<M<1/2(3^{\nu(i_M)}+1).
\end{displaymath}

\noindent Since $\nu$ is increasing, this leads to

\begin{displaymath}
\left\{\begin{array}{c}
i_M < \nu^{-1}(\log_3(2M-1))+1,\\
i_M > \nu^{-1}(\log_3(2M-1)).
\end{array}\right.
\end{displaymath}

\noindent As $i_M$ is an integer, we have $i_M=\lceil\nu^{-1}(\log_3(2M-1))\rceil$.

\begin{lem}\label{lemnu}
For almost every $\xi\in\tilde{E_1}$, $\lim_{i\to\infty}\frac{\nu(i+1)}{\nu(i)}=1$.
\end{lem}

\begin{proof}[Proof of the lemma:]
By (\ref{eqnu}), $\{\nu(i)\}_{i\geq 1}$ satisfies

\begin{displaymath}
\nu(i+1)-\nu(i)=\left\{\begin{array}{cc}
|A_{i/2+1}| & \textrm{if $i$ is even,}\\
|B_{(i+1)/2}| & \textrm{if $i$ is odd.}
\end{array}\right.
\end{displaymath}

For every $k\geq 1$, define the event $E_k:=\{\xi\in\{0,1,2\}^\omega|\hspace{1mm}|A_k|\geq k\}$ (respectively $\bar{E}_k:=\{\xi\in\{0,1,2\}^\omega|\hspace{1mm}|B_k|\geq k\}$). As $\mathbb{P}(E_k)\leq \left(\frac{2}{3}\right)^k$ and $\sum_{k\geq 1}\left(\frac{2}{3}\right)^k<\infty$, it follows from Borel-Cantelli Lemma, that $\mathbb{P}(\limsup_{k\to\infty}E_k)=0$. Identically, $\mathbb{P}(\limsup_{k\to\infty}\bar{E}_k)=0$. In other words, there almost surely exists $i_0\geq 1$ such that, for all $i>i_0$, $|A_{i/2+1}|<i/2+1$ (respectively $|B_{(i+1)/2}|<(i+1)/2$).\\
\indent We have, for $i$ even,

\begin{displaymath}
1\leq\frac{\nu(i+1)}{\nu(i)}=\frac{\nu(i)+|A_{i/2+1}|}{\nu(i)}\leq 1+\frac{i}{2\nu(i)}+\frac{1}{\nu(i)},
\end{displaymath}

\noindent where the last inequality holds almost surely for any $i$ sufficiently large. The same bound holds for $i$ odd. Using a similar argument than in the proof of Theorem \ref{Main}, we check that $i/\nu(i)$, which is non-increasing, tends to $0$ as $i\to\infty$.
\end{proof}

Let $x_M=\nu^{-1}(\log_3(2M-1))+1$, so that $i_M=\lfloor x_M\rfloor$. By Lemma \ref{lemnu}, for any $\epsilon >0$, there exists $M_0$ such that for any $M>M_0$, $\nu(i_M+1)\leq (1+\epsilon)\nu(i_M)$. From (\ref{eq2Interlaced}), we get

\begin{displaymath}
\frac{L}{2}\cdot 2^{-(2+\epsilon)\nu(i_M)} \leq\mathbb{P}_{\mu_n}(Mav_{H_n}(\cdot,\xi_n)=M)
\leq 2\cdot 2^{-\frac{2+\epsilon}{1+\epsilon}\nu(i_M+1)}.
\end{displaymath}

\noindent As $i_M\leq x_M\leq i_M+2$ and $\nu$ is increasing, we have

\begin{displaymath}
\frac{L}{2}\cdot 2^{-(2+\epsilon)\nu(x_M)} \leq\mathbb{P}_{\mu_n}(Mav_{H_n}(\cdot,\xi_n)=M)
\leq 2\cdot 2^{-\frac{2+\epsilon}{1+\epsilon}\nu(x_M-1)}.
\end{displaymath}

\noindent Using, for the lower bound, the fact that for any $\epsilon' >0$, there is $M'_0$ such that for any $M>M'_0$, $\nu(\nu^{-1}(\log_3(2M-1))+1)\leq (1+\epsilon')\log_3(2M-1)$, we get

\begin{displaymath}
\frac{L}{2}\cdot 2^{-(2+\epsilon)(1+\epsilon')\log_3(2M-1)} \leq\mathbb{P}_{\mu_n}(Mav_{H_n}(\cdot,\xi_n)=M)
\leq 2\cdot 2^{-\frac{2+\epsilon}{1+\epsilon}\log_3(2M-1)}
\end{displaymath}

\noindent which is equivalent to

\begin{displaymath}
\frac{L}{2}\cdot (2M-1)^{-\frac{(2+\epsilon)(1+\epsilon')}{\log_2(3)}} \leq\mathbb{P}_{\mu_n}(Mav_{H_n}(\cdot,\xi_n)=M)
\leq 2\cdot (2M-1)^{-\frac{2+\epsilon}{(1+\epsilon)\log_2(3)}}.
\end{displaymath}

\noindent Thus, almost surely, $\lim_{n\to\infty}\mathbb{P}_{\mu_n}(Mav_{H_n}(\cdot,\xi_n)=M)\sim M^{-\frac{2\log 2}{\log 3}}$.
\end{proof}

\section{Growth of Orbital Schreier Graphs and Critical Exponent of the ASM} \label{LastSection}

In this section we show that, under some conditions, the critical exponent of the ASM on a finite approximation of an infinite one-ended cactus is related to the growth of that graph. Then, we exhibit a family of iterated monodromy groups of quadratic polynomials such that the ASM on the corresponding sequences of Schreier graphs is critical in the random weak limit, with arbitrarily small critical exponent.

\subsection{Degree of polynomial growth of orbital Schreier graphs and critical exponent}

\indent Given a locally finite graph $\Gamma$ and $v\in V(\Gamma)$, we say that $\Gamma$ has polynomial growth of degree $\alpha$ if the quantity

\begin{displaymath}
\alpha:=\limsup_{r\to\infty}\frac{\log(|B_\Gamma(v,r)|)}{\log r}
\end{displaymath}

\noindent is finite. Note that $\alpha$ does not depend on the choice of $v$.\\
\indent Let $(\Gamma,v)$ be an infinite one-ended cactus rooted at $v$. Let $\mathcal{CP}_{v}=\mathcal{C}_1\mathcal{C}_2\dots$ be the unique block-path of infinite length in $\Gamma$ starting at $v$. Recall from Subsection \ref{SubsecCP} that, for each $i\geq 1$, $p_i$ denotes the cut vertex between $\mathcal{C}_i$ and $\mathcal{C}_{i+1}$, and that $D(p_i)$ denotes the subgraph of $\Gamma$ consisting of the union of all finite connected components remaining when removing $p_i$, together with $p_i$. Finally, recall that $d_i$ denotes the number of vertices in $D(p_i)$.

\begin{thm} \label{Growth}
Let $(\Gamma,v)$ be an infinite one-ended cactus rooted at $v$. Let $\{H_n\}_{n\geq 1}$ be an exhaustion of $(\Gamma,v)$ as in Convention \ref{RemOnCPinf} and, for any $n\geq 1$, let $p^{(n)}$ be the dissipative vertex in $H_n$. Denote by $\mathcal{CP}^n_{v}=\mathcal{C}_1\dots \mathcal{C}_{r_n}\subset\mathcal{CP}_{v}$ the finite block-path in $H_n$ joining vertex $v$ to $p^{(n)}$. Suppose that $\sum_{j=1,|\mathcal{C}_j|>2}^{r_n}\frac{1}{|\mathcal{C}_j|}$ converges as $r_n\to\infty$. Suppose moreover that the subgraphs $D(p_i)$, $i\geq 1$, satisfy the following requirements:

\begin{enumerate}
\item there exists a constant $c>0$ such that, for any $i$ sufficiently large, $Diam(D(p_i))\leq c|\mathcal{C}_i|$;
\item $\lim_{i\to\infty}\frac{\log d_{i-1}}{\log d_i}=1$.
\end{enumerate}

\noindent Then, for any $\epsilon>0$, there exists $M_0$ such that, for any $M>M_0$

\begin{displaymath}
C_1\cdot M^{-\frac{2}{\beta'-\epsilon}} \leq\lim_{n\to\infty}\mathbb{P}_{\mu_n}(Mav_{H_n}(\cdot,v)=M)\leq C_2\cdot M^{-\frac{2}{\beta+\epsilon}},
\end{displaymath}

\noindent where $C_1,C_2>0$, $\beta:=\limsup_{i\to\infty}\frac{\log d_i}{\log Diam(D(p_i))}$ and $\beta':=\liminf_{i\to\infty}\frac{\log d_i}{\log Diam(D(p_i))}$.\\
\indent In particular, if $\beta=\beta'$, then the ASM on the sequence $\{H_n\}_{n\geq 1}$ approximating $(\Gamma,v)$ is critical (in the sense of Definition \ref{DefCrit}) with critical exponent equal to $\delta=2/\beta$.
\end{thm}

\begin{cor}
Let $(\Gamma,v)$ be as in Theorem \ref{Growth}. Suppose that $\Gamma$ has polynomial growth and that its degree of growth $\alpha$ is given by the quantity $\lim_{i\to\infty}\frac{\log d_i}{\log Diam(D(p_i))}$. Then, the critical exponent $\delta$ is related to the growth degree $\alpha$ of $\Gamma$ by $\delta=2/\alpha$.
\end{cor}

\begin{proof}
From Theorem \ref{ThmCP}, for any integer $M$ large enough that occurs as the mass of an avalanche, we have

\begin{equation} \label{eq1Growth}
\frac{L}{2\cdot|\mathcal{C}_{i_M}|\cdot|\mathcal{C}_{i_M+1}|} \leq\mathbb{P}_{\mu_n}(Mav_{H_n}(\cdot,v)=M)\leq \frac{2}{|\mathcal{C}_{i_M}|\cdot|\mathcal{C}_{i_M+1}|}
\end{equation}

\noindent where the index $i_M$ is uniquely determined by the condition $d_{i_M-1}\leq M<d_{i_M}$ and $0<L\leq 1$ is a constant. Applying logarithm to these inequalities and normalizing, we get

\begin{equation} \label{eqNormalizing}
\frac{\log d_{i_M-1}}{\log Diam(D(p_{i_M}))}\leq\frac{\log M}{\log Diam(D(p_{i_M}))}<\frac{\log d_{i_M}}{\log Diam(D(p_{i_M}))}.
\end{equation}

\noindent By condition \emph{2.}, we have

\begin{displaymath}
\liminf_{M\to\infty}\frac{\log d_{i_M-1}}{\log Diam(D(p_{i_M}))}=\liminf_{M\to\infty}\frac{\log d_{i_M}}{\log Diam(D(p_{i_M}))}=\beta'.
\end{displaymath}

\noindent On the other hand, condition \emph{1.} implies that for any $M$ sufficiently large,

\begin{displaymath}
\frac{\log M}{\log \left(c|\mathcal{C}_{i_M}|\right)}\leq\frac{\log M}{\log Diam(D(p_{i_M}))}\leq\frac{\log M}{\log\left( \tilde{c}|\mathcal{C}_{i_M}|\right)}
\end{displaymath}

\noindent (the upper bound follows from the fact that, by definition, $\mathcal{C}_{i}\subset D(p_{i})$ for any $i\geq 1$.) Hence, for any $\epsilon>0$, there exists $M_0$ such that for any $M>M_0$,

\begin{displaymath}
\beta'-\epsilon<\frac{\log M}{\log \left(c|\mathcal{C}_{i_M}|\right)}<\beta+\epsilon
\end{displaymath}

\noindent which is equivalent to

\begin{displaymath}
\frac{1}{c}M^{\frac{1}{\beta+\epsilon}}<|\mathcal{C}_{i_M}|<\frac{1}{c}M^{\frac{1}{\beta'-\epsilon}}.
\end{displaymath}

\noindent If we normalize in (\ref{eqNormalizing}) by $\log Diam(D(p_{i_M+1}))$, we obtain similarly

\begin{displaymath}
\frac{1}{c}M^{\frac{1}{\beta+\epsilon}}<|\mathcal{C}_{i_M+1}|<\frac{1}{c}M^{\frac{1}{\beta'-\epsilon}}.
\end{displaymath}

\noindent Thus, we have

\begin{displaymath}
c^2M^{\frac{-2}{\beta'-\epsilon}}<\frac{1}{|\mathcal{C}_{i_M}|\cdot|\mathcal{C}_{i_M+1}|}
<c^2M^{\frac{-2}{\beta+\epsilon}}
\end{displaymath}

\noindent and replacing in (\ref{eq1Growth}), we get the result.

\end{proof}

\subsection{Examples with arbitrarily small critical exponent} \label{k(v)}

We will now consider a particular family of self-similar groups of automorphisms of the binary rooted tree that gives rise to Schreier graphs of bigger and bigger degree and of bigger and bigger polynomial growth. These graphs satisfy the conditions of our Theorem \ref{Growth}, and thus provide examples of criticality with critical exponent arbitrarily close to $0$.\\
\indent The groups we are going to consider are realized as iterated monodromy groups of quadratic polynomials $z^2+c$, where the parameter $c$ is chosen to be the center of one of the secondary $p/q$-components of the Mandelbrot set, so that the critical point $0$ of the polynomial $z^2+c$ belongs to a super-attracting cycle of length $q\geq 2$. The case $q=2$ corresponds
to the Basilica group, see Figure \ref{JULIABASILICA}, and the case $q=3$ is the so-called Douady rabbit, see Figure \ref{JULIADOUADY}.  \\
\indent If the orbit of $0$ under iterations of the polynomial $z^2+c$ is a finite cycle, one can associate to the polynomial a \emph{kneading} automaton $\mathcal{A}_v$, where $v$ is a finite binary word, and the self-similar group $\mathcal{K}(v)$ generated by $\mathcal{A}_v$ is the iterated monodromy group of $z^2+c$ (see Chapters 6.6-6.11 in \cite{NekBook}). The length of the word $v$ is equal to the size of the orbit of $0$ under iterations of the polynomial. For a word $v=x_1x_2\dots x_{k-1}\in\{0,1\}^{k-1}$, $k>1$, the automaton $\mathcal{A}_v$ has $k+1$ states (including the identity state) and its Moore diagram is pictured in Figure \ref{AUTOMATKV} (for $x\in\{0,1\}$, we write $\bar{x}:=1-x$):

\begin{figure}[H]
\begin{center}
\begin{picture}(0,0)%
\includegraphics{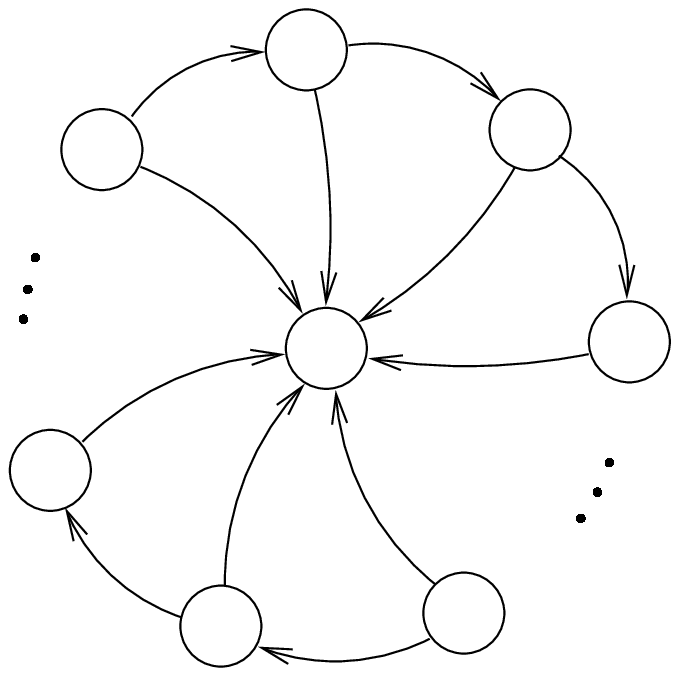}%
\end{picture}%
%
%
\setlength{\unitlength}{4736sp}%
\begingroup\makeatletter\ifx\SetFigFont\undefined%
\gdef\SetFigFont#1#2#3#4#5{%
  \reset@font\fontsize{#1}{#2pt}%
  \fontfamily{#3}\fontseries{#4}\fontshape{#5}%
  \selectfont}%
\fi\endgroup%
\begin{picture}(2842,2837)(1374,-2516)
\put(3891,-1054){\makebox(0,0)[lb]{\smash{{\SetFigFont{11}{13.2}{\rmdefault}{\mddefault}{\updefault}{\color[rgb]{0,0,0}$a_{k-1}$}%
}}}}
\put(1587,-1553){\makebox(0,0)[lb]{\smash{{\SetFigFont{11}{13.2}{\rmdefault}{\mddefault}{\updefault}{\color[rgb]{0,0,0}$a_{i-1}$}%
}}}}
\put(3563,-186){\makebox(0,0)[lb]{\smash{{\SetFigFont{11}{13.2}{\rmdefault}{\mddefault}{\updefault}{\color[rgb]{0,0,0}$a_k$}%
}}}}
\put(2757,-1079){\makebox(0,0)[lb]{\smash{{\SetFigFont{11}{13.2}{\rmdefault}{\mddefault}{\updefault}{\color[rgb]{0,0,0}$id$}%
}}}}
\put(3252,-2140){\makebox(0,0)[lb]{\smash{{\SetFigFont{11}{13.2}{\rmdefault}{\mddefault}{\updefault}{\color[rgb]{0,0,0}$a_{i+1}$}%
}}}}
\put(2332,-2180){\makebox(0,0)[lb]{\smash{{\SetFigFont{11}{13.2}{\rmdefault}{\mddefault}{\updefault}{\color[rgb]{0,0,0}$a_i$}%
}}}}
\put(2674,134){\makebox(0,0)[lb]{\smash{{\SetFigFont{11}{13.2}{\rmdefault}{\mddefault}{\updefault}{\color[rgb]{0,0,0}$a_1$}%
}}}}
\put(1867,-269){\makebox(0,0)[lb]{\smash{{\SetFigFont{11}{13.2}{\rmdefault}{\mddefault}{\updefault}{\color[rgb]{0,0,0}$a_2$}%
}}}}
\put(3237,184){\makebox(0,0)[lb]{\smash{{\SetFigFont{11}{13.2}{\rmdefault}{\mddefault}{\updefault}{\color[rgb]{0,0,0}$0|1$}%
}}}}
\put(2905,-411){\makebox(0,0)[lb]{\smash{{\SetFigFont{11}{13.2}{\rmdefault}{\mddefault}{\updefault}{\color[rgb]{0,0,0}$1|0$}%
}}}}
\put(4025,-510){\makebox(0,0)[lb]{\smash{{\SetFigFont{11}{13.2}{\rmdefault}{\mddefault}{\updefault}{\color[rgb]{0,0,0}$x_{k-1}|x_{k-1}$}%
}}}}
\put(2806,-2456){\makebox(0,0)[lb]{\smash{{\SetFigFont{11}{13.2}{\rmdefault}{\mddefault}{\updefault}{\color[rgb]{0,0,0}$x_i|x_i$}%
}}}}
\put(1949,122){\makebox(0,0)[lb]{\smash{{\SetFigFont{11}{13.2}{\rmdefault}{\mddefault}{\updefault}{\color[rgb]{0,0,0}$x_1|x_1$}%
}}}}
\put(1389,-2065){\makebox(0,0)[lb]{\smash{{\SetFigFont{11}{13.2}{\rmdefault}{\mddefault}{\updefault}{\color[rgb]{0,0,0}$x_{i-1}|x_{i-1}$}%
}}}}
\put(3039,-1665){\makebox(0,0)[lb]{\smash{{\SetFigFont{11}{13.2}{\rmdefault}{\mddefault}{\updefault}{\color[rgb]{0,0,0}$\bar{x}_i|\bar{x}_i$}%
}}}}
\put(2316,-405){\makebox(0,0)[lb]{\smash{{\SetFigFont{11}{13.2}{\rmdefault}{\mddefault}{\updefault}{\color[rgb]{0,0,0}$\bar{x}_1|\bar{x}_1$}%
}}}}
\put(3319,-788){\makebox(0,0)[lb]{\smash{{\SetFigFont{11}{13.2}{\rmdefault}{\mddefault}{\updefault}{\color[rgb]{0,0,0}$\bar{x}_{k-1}|\bar{x}_{k-1}$}%
}}}}
\put(2462,-1908){\makebox(0,0)[lb]{\smash{{\SetFigFont{11}{13.2}{\rmdefault}{\mddefault}{\updefault}{\color[rgb]{0,0,0}$\bar{x}_{i-1}|\bar{x}_{i-1}$}%
}}}}
\end{picture}%
\end{center}
\caption{The automaton $\mathcal{A}_v$ corresponding to the word $v=x_1x_2\dots x_{k-1}$.}
\label{AUTOMATKV}
\end{figure}

Consequently, the generators $\{a_1,\dots,a_k\}$ of the group $\mathcal{K}(v)$ generated by $\mathcal{A}_v$ have the following self-similar structure:

\begin{displaymath}
a_1=(0 \ 1)(a_k,id),
\hspace{2cm}
a_{i+1}=\left\{\begin{array}{cc}
e(a_i,id) & \textrm{if $x_i=0$,} \\
e(id,a_i) & \textrm{if $x_i=1$,}
\end{array}\right.
\hspace{5mm}\textrm{for $i=1,\dots,k-1$.}
\end{displaymath}

We can, for example, consider the family of groups $\mathcal{K}(0^{k-1})$ for $k>1$. The group $\mathcal{K}(0)$ is the Basilica group that we have already studied in Sections \ref{Basilica} and \ref{AvOnBasilica}, whereas $\mathcal{K}(00)$ is the group $IMG(z^2+c)$ where $c\approx-0.1225+0.7448i$. The Julia set of this group, called the \emph{Douady Rabbit}, is represented in Figure \ref{JULIADOUADY}.

\begin{figure}[H]
\begin{center}
\includegraphics[scale=0.23]{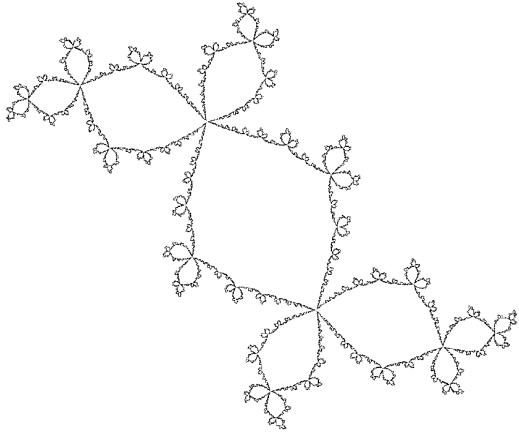}
\end{center}
\caption{The Julia set of the Douady Rabbit.}
\label{JULIADOUADY}
\end{figure}

For any $k>1$, the group $\mathcal{K}(0^{k-1})$ is the iterated monodromy group of a post-critically finite polynomial and, by Theorem \ref{ThmNekrashevych}, the Schreier graphs of the action of $\mathcal{K}(0^{k-1})$ on the levels of the binary rooted tree are cacti. By extending to the groups $\mathcal{K}(0^{k-1})$, for any $k>1$, the analysis done for the Basilica group, we obtain the following description of the finite Schreier graphs:

\begin{prop} \label{PropSchreierKv}
Let $k\geq 2$, and consider the Schreier graphs $\Gamma_n:=\Gamma(\mathcal{K}(0^{k-1}),\{a_1,\dots,a_k\},\{0,1\}^n)$ of the action of $\mathcal{K}(0^{k-1})$ on the levels of the binary rooted tree. Given $\xi_n\in\Gamma_n$, let $\mathcal{CP}_{\xi_n}=\mathcal{C}_1\dots\mathcal{C}_{r_n}$ be the block-path joining $\xi_n$ to the vertex $0^n$ in $\Gamma_n$. Then, $r_n\leq n$ and the sizes of the blocks of $\mathcal{CP}_{\xi_n}$ are given by $|\mathcal{C}_j|=2^{b_j}$, where the sequence $\{b_j\}_{j=1}^{r_n}$ is a non-decreasing sequence of positive integers with no constant segments of length greater than $k$.
\end{prop}

By \cite{BondDanNag11}, almost all orbital Schreier graphs $\Gamma_\xi:=\Gamma(\mathcal{K}(0^{k-1}),\{a_1,\dots,a_k\},\mathcal{K}(0^{k-1})\cdot\xi)$ (with respect to the uniform distribution on the boundary $\partial T$ of the tree) have one end. Denote by $E_1\subset \partial T$ the set of full measure constituted of infinite words $\xi$ such that the corresponding orbital Schreier graph $\Gamma_\xi$ has one end. For $\xi\in E_1$, as in the case of Basilica one-ended Schreier graphs, the limit in $\mathcal{X}$ $\lim_{n\to\infty}(\mathcal{CP}_{\xi_n},\xi_n)$ is isomorphic to $\mathcal{CP}_{\xi}$, the unique block-path of infinite length in $\Gamma_\xi$ starting at $\xi$. Similarly than in Subsection \ref{SubSectOneEnd}, for any $n\geq 1$, we set $p^{(n)}:=0^n$ in $\Gamma_n$ to be dissipative. The infinite graph $(\Gamma_\xi,\xi)$ is exhausted by the subgraphs $H_n$ that are isomorphic, for each $n$, to the connected component of $\xi_n$ in $\Gamma_n$ remaining when removing vertex $0^n$, together with $0^n$ (see Remark \ref{remsubgraphsH_n}). Our choice of subgraphs $H_n$ corresponds to Convention \ref{RemOnCPinf}. It thus follows from Proposition \ref{PropSchreierKv}, that the orbital rooted Schreier graph $(\Gamma_\xi,\xi)$ satisfies the assumptions of Theorem \ref{ThmCP}.\\
\indent On the other hand, the orbital Schreier graphs $\Gamma_\xi$ have polynomial growth, and by applying an algorithm from \cite{Bond07}, we show that the degree of polynomial growth grows with $k$ (essentially this is due to the fact that the graphs are $2k$-regular.)

\begin{prop}
The degree of polynomial growth of the orbital Schreier graphs $\Gamma_\xi$ of the action of $\mathcal{K}(0^{k-1})$ on $\partial T$ is at least $k/2$.
\end{prop}

\indent One also verifies that, for almost every $\xi\in E_1$, the orbital rooted Schreier graph $(\Gamma_\xi,\xi)$ satisfies the assumptions of Theorem \ref{Growth} with $\beta=\beta'=k$, which then implies the following:

\begin{thm}
For $k\geq 2$, the ASM on the sequence $\{\Gamma_n\}_{n\geq 1}$ of Schreier graphs of the action of
$\mathcal{K}(0^{k-1})$ is critical in the random weak limit (in the sense of Definition \ref{defCritRWL}) with critical exponent $\delta=2/k$.
\end{thm}

\begin{cor}
$\{\mathcal{K}(0^{k-1})\}_{k\geq 2}$ is a family of self-similar groups such that the ASM on the associated sequences of Schreier graphs is critical in the random weak limit, and the critical exponent $\delta>0$ can be arbitrarily small.
\end{cor}

\begin{rem} \label{RemOnDiam} \rm
Another quantity related to the size of avalanches, the diameter of the subgraph spanned by vertices touched by the avalanche, can be studied in a very similar way than the mass. For all examples of Schreier graphs we consider in this paper, one can slightly modify the proof of Theorem \ref{ThmCP} to get bounds for the probability distribution of the diameter of avalanches, instead of the mass. Since the examples we consider satisfy, almost surely, the assumptions of Theorem \ref{Growth}, one can deduce that the critical exponent $\delta'>0$ defined with respect to the diameter of avalanches (see Definition \ref{DefCrit}) is related to the growth degree $\alpha$ of the graph by $\delta'=1/\alpha$.
\end{rem}

\section*{Acknowledgements}

Figure \ref{JULIADOUADY} is reproduced under the terms of Creative Commons Attribution-ShareAlike 3.0 license. Figures \ref{AUTOMATBASILICA}, \ref{SUBSTRULES}, \ref{FINITESCHREIER} and \ref{ONEENDEDGRAPH} were published for the first time in \cite{DanDonMat09}.



\begin{thebibliography}{0}
\bibitem{AlLyons07}
  Aldous, D., and R. Lyons.
  Processes on unimodular random networks.
  \emph{Electron. J. Probab.} 12, no. 54 (2007): 1454-1508.

\bibitem{AliDhar95}
  Ali, A., and D. Dhar.
  Structure of avalanches and breakdown of simple scaling in the {A}belian sandpile model in one dimension.
  \emph{Phys. Rev. E} 52, no. 5 (1995): 4804-4816.

\bibitem{AthJar04}
  Athreya, S. R., and A. A. Jarai.
  Infinite volume limit for the stationary distribution of abelian sandpile models.
  \emph{Commun. Math. Phys.} 249 (2004): 197-213.

\bibitem{BachHarpNagn97}
  Bacher, R., P. de la Harpe, and T. Nagnibeda.
  The lattice of integral flows and the lattice of integral cuts on a finite graph.
  \emph{Bull. Soc. math. France} 125 (1997): 167-198.

\bibitem{BakTangWies88}
  Bak, P., K. Tang, and K. Wiesenfeld.
  Self-organized criticality.
  \emph{Phys. Rev. A} 38 (1988): 364-374.

\bibitem{amenability}
  Bartholdi, L., and B. Vir\'{a}g.
  Amenability via random walks.
  \emph{Duke Math Journal} 130 (2005): 39-56.

\bibitem{RandWeakLimit}
  Benjamini, I., and O. Schramm.
  Recurrence of distributional limits of finite planar graphs.
  \emph{Electronic Journal of Probability} 6, no. 23 (2001): 1-23.

\bibitem{BenLyPeSch01}
  Benjamini, I., R. Lyons, Y. Peres, and O. Schramm.
  Uniform spanning forests.
  \emph{Ann. Probab.} 29, no. (2001): 1-65.

\bibitem{Biggs99}
  Biggs, N. L.
  Chip-firing and the critical group of a graph.
  \emph{Journal of Algebraic Combinatorics} 9, no. 1 (1999): 25-45.

\bibitem{BLSch}
  Bj\"{o}rner, A., L. Lovasz, and P. Shor.
  Chip-firing games on graphs.
  \emph{European J. Combin.} 12, no. 4 (1991): 283-291.

\bibitem{Bond07}
  Bondarenko, I.
  Groups generated by bounded automata and their Schreier graphs.
  \emph{PhD Thesis.} Texas A\&M Univ., 2007.

\bibitem{BondDanNag11}
  Bondarenko, I., D. D'Angeli, and T. Nagnibeda.
  Ends of Schreier graphs of self-similar groups.
  \emph{Preprint.}

\bibitem{ChenShed07}
  Chen, W., and T. Schedler.
  Concrete and abstract structure of the sandpile group for thick trees with loops.
  \emph{arXiv:math/0701381} (2007).

\bibitem{ChungEllis02}
  Chung, F., and R. B. Ellis.
  A chip-firing game and Dirichlet eigenvalues.
  \emph{Discrete Math.} 257, no. 2-3 (2002): 341-355.

\bibitem{DaerdenVander04}
  Daerden, F., and C. Vanderzande.
  Sandpiles on a Sierpi\'nski gasket.
  \emph{Physica A} 256, no. 3-4 (1998): 533-546.

\bibitem{DanDonMat09}
  D'Angeli, D., A. Donno, M. Matter, and T. Nagnibeda.
  Schreier graphs of the Basilica group.
  \emph{J. Mod. Dyn.} 4, no. 1 (2010): 167-205.

\bibitem{Dhar90}
  Dhar, D.
  Self-organized critical state of sandpile automaton models.
  \emph{Phys. Rev. Letters} 64, no. 14 (1990): 1613-1616.

\bibitem{Dhar99}
  Dhar, D.
  Theoretical studies of self-organized criticality.
  \emph{Physica A} 369 (2006): 29-70.

\bibitem{DharMajumdar90}
  Dhar, D., and S. N. Majumdar.
  Abelian sandpile model on the Bethe lattice.
  \emph{J.Phys. A: Math. Gen.} 23 (1990): 4333-4350.

\bibitem{DharMajum91}
  Dhar, D., and S. N. Majumdar.
  Height correlations in the Abelian sandpile model.
  \emph{J.Phys. A: Math. Gen.} 24 (1991): 357-362.



\bibitem{Diestel}
  Diestel, R.
  \emph{Graph Theory,} 3rd ed., Springer, 2006.

\bibitem{Grig05}
  Grigorchuk, R.
  Solved and unsolved problems around one group.
  In: \emph{Infinite Groups: Geometric, Combinatorial and Dynamical Aspects}.
  L. Bartholdi, T. Ceccherini-Silberstein, T. Smirnova-Nagnibeda, and A. \.{Z}uk (Eds.). Progr. Math. 248: 117-218, Birkh\"auser, 2005.

\bibitem{GNS}
  Grigorchuk, R., V. Nekrashevych, and V. I. Sushchanskii.
  Automata, dynamical systems and groups.
  \emph{Tr. Mat. Inst. Steklova} 231 (Din. Sist. Avtom. i Beskon. Gruppy) (2000): 134-214.

\bibitem{grizuk}
  Grigorchuk, R., and A. \.{Z}uk.
  On a torsion-free weakly branch group defined by a three-state automaton.
  \emph{International J. Algebra Comput.} 12, no. 1-2 (2002): 223-246.

\bibitem{Grom}
  Gromov, M.
  \emph{Structures m\'etriques pour les vari\'et\'es riemanniennes.}
   J. Lafontaine, and P. Pansu (Eds.). CEDIC, Paris: 1981.

\bibitem{Sierpinski}
  Kutnjak-Urbanc, B., S. Zapperi, S. Milosevic, and H. Eugene Stanley.
  Sandpile model on the Sierpi\'nski gasket fractal.
  \emph{Phys. Rev. E} 54, no. 1 (1996): 272-277.

\bibitem{BorRoss02}
  Le Borgne, Y., and D. Rossin.
  On the identity of the sandpile group.
  \emph{Discrete Math.} 256, no. 3 (2002): 775-790.

\bibitem{Lev09}
  Levine, L.
  The sandpile group of a tree.
  \emph{European Journal of Combinatorics} 30, no. 4 (2009): 1026-1035.

\bibitem{Lor08}
  Lorenzini, D.
  Smith normal form and laplacians.
  \emph{J. Combin. Theory B} 98, no. 6 (2008): 1271-1300.

\bibitem{MaesRedSaa02}
  Maes, C., F. Redig, and E. Saada.
  The Abelian sandpile model on an infinite tree.
  \emph{Ann. Probab.} 30 (2002): 2081-2107.

\bibitem{MaesRedSaaMoff00}
  Maes, C., F. Redig, E. Saada, and A. Van Moffaert.
  On the thermodynamic limit for a one-dimensional sandpile process.
  \emph{Markov Process. Related Fields} 6 (2000): 1-21.

\bibitem{Mat09}
  Matter, M.
  Abelian sandpile model on cacti graphs.
  \emph{In preparation}.

\bibitem{MeestRedZnam01}
  Meester, R., F. Redig, and D. Znamenski.
  The Abelian Sandpile: a mathematical introduction.
  \emph{Markov Processes Relat. Fields} 7 (2001): 509-523.

\bibitem{Nekr09}
  Nekrashevych, V.
  Combinatorics of polynomial iterations.
  In: \emph{Complex dynamics--Families and Friends}, 169-214.
  D. Schleicher (Ed.).  A. K. Peters, 2009.

\bibitem{Nek}
  Nekrashevych, V.
  Iterated monodromy groups.
  \emph{arXiv:math/0312306}

\bibitem{NekBook}
  Nekrashevych, V.
  \emph{Self-Similar Groups.}
  American Mathematical Society, 2005.

\bibitem{Priezz00}
  Priezzhev, V. B.
  The upper critical dimension of the Abelian sandpile model.
  \emph{J. Stat. Phys.} 98 (2000): 667-684.

\bibitem{PriezzKitIva96}
  Priezzhev, V. B., D.V. Ktitarev, and E.V. Ivashkevich.
  Formation of avalanches and critical exponents in an Abelian sandpile model.
  \emph{Phys. Rev. Let.} 76, no. 12 (1996): 2093-2096.

\bibitem{RedigHouches05}
  Redig, F.
  Mathematical aspects of the Abelian sandpile model.
  \emph{Lecture notes}. Les Houches, 2005.

\bibitem{RogTep09}
  Rogers, L. G., and A. Teplyaev.
  Laplacians on the Basilica Julia set.
  \emph{Commun. Pure Appl. Anal.} 9, no. 1 (2010): 211-231.

\bibitem{RuSen}
  Ruelle P., and S. Sen.
  Toppling distributions in one-dimensional abelian sandpiles.
  \emph{J. Phys. A: Math. Gen.} 25, no. 22 (1992): 1257-1264.

\bibitem{Serre}
  Serre, J. -P.
  \emph{Trees.} Translated from the French original by John Stillwell. Corrected 2nd printing of the 1980 English translation. Springer Monographs in Mathematics. Springer-Verlag, 2003.
  
\bibitem{Kozakova}
  Spakulova, I.
  Percolation and Ising model on tree-like graphs.
  \emph{PhD thesis}. Vanderbilt Univ. 2008.

\bibitem{Tannery}
  Tannery, J.
  \emph{Introduction \`a la th\'eorie des fonctions d'une variable.} I Chapter 3.
  Librairie scientifique A. Hermann, 1904.

\bibitem{Toump07}
  Toumpakari, E.
  On the sandpile group of regular trees.
  \emph{European J. Combin.} 28, no. 3 (2007): 822-842.
\end{thebibliography}
\end{document}